\documentclass[reqno, 11pt]{amsart}

\newtheorem{theorem}{Theorem}[section]
\newtheorem{lemma}[theorem]{Lemma}
\newtheorem{proposition}[theorem]{Proposition}
\newtheorem{corollary}[theorem]{Corollary}

\theoremstyle{definition}
\newtheorem{definition}[theorem]{Definition}
\newtheorem{ex}[theorem]{Example}

\newtheorem{remark}[theorem]{Remark}
\newtheorem{algo}[theorem]{Algorithm}

\usepackage{pdfsync}
\usepackage{hyperref}
\usepackage{pifont}

\usepackage{times}

\usepackage{amsxtra}
\usepackage{amssymb}
\usepackage{amsbsy}
\usepackage{amsthm}

\DeclareMathOperator{\Hom}{Hom}

\numberwithin{equation}{section}

\usepackage{bbm}
\usepackage{url}
\usepackage{euscript}
\usepackage{pb-diagram}
\usepackage{lamsarrow}
\usepackage{pb-lams}
\usepackage{pifont}

\usepackage{graphicx}
\usepackage{epstopdf}

\def\inpr{\mathbin{\hbox to 6pt{\vrule height0.4pt width5pt depth0pt \kern-.4pt \vrule height6pt width0.4pt depth0pt\hss}}}

\newcommand{\ep}{\varepsilon}

\newcommand{\si}{\sigma}

\newcommand{\ra}{\rightarrow}
\newcommand{\hra}{\hookrightarrow}
\newcommand{\Llra}{\Longleftrightarrow}

\newcommand{\eA}{\EuScript{A}}
\newcommand{\eB}{\EuScript{B}}
\newcommand{\eC}{\EuScript{C}}

\newcommand{\eE}{\EuScript{E}}
\newcommand{\eH}{\EuScript{H}}
\newcommand{\eJ}{\EuScript{J}}
\newcommand{\eN}{\EuScript{N}}
\newcommand{\eP}{\EuScript{P}}
\newcommand{\eR}{\EuScript{R}}
\newcommand{\eS}{\EuScript{S}}
\newcommand{\eV}{\EuScript{V}}
\newcommand{\eX}{\EuScript{X}}
\newcommand{\eZ}{\EuScript{Z}}

\newskip\aline \newskip\halfaline
\def\skipaline{\vskip\aline}

\def\qedbox{$\rlap{$\sqcap$}\sqcup$}
\def\qed{\nobreak\hfill\penalty250 \hbox{}\nobreak\hfill\qedbox\skipaline}

\def\proofend{\eqno{\mbox{\qedbox}}}

\newcommand{\bse}{{\boldsymbol{e}}}
\newcommand{\bn}{\boldsymbol{n}}
\newcommand{\bv}{{\boldsymbol{v}}}
\newcommand{\bsi}{{\boldsymbol{\sigma}}}
\newcommand{\btau}{{\boldsymbol{\tau}}}

\newcommand{\bom}{{\boldsymbol{\omega}}}
\newcommand{\Bom}{{\boldsymbol{\Omega}}}

\newcommand{\bsN}{\boldsymbol{N}}

\newcommand{\bsV}{\boldsymbol{V}}

\newcommand{\bE}{\mathbb{E}}
\newcommand{\bR}{\mathbb{R}}
\newcommand{\bZ}{\mathbb{Z}}

\newcommand{\pato}{\partial_{\mathrm{top}}}
\newcommand{\pa}{\partial}
\newcommand{\dual}{{\spcheck{}}}

\DeclareMathOperator{\dist}{dist}
\DeclareMathOperator{\cl}{cl}
\DeclareMathOperator{\supp}{supp}
\DeclareMathOperator{\ori}{or}
\DeclareMathOperator{\codim}{codim}
\DeclareMathOperator{\stack}{\textsf{stack}}
\DeclareMathOperator{\jump}{\textsf{jump}}
\DeclareMathOperator{\polygon}{\textsf{polygon}}

\newcommand{\ve}{\varepsilon}
\newcommand{\chio}{\chi_o}
\newcommand{\vfi}{\varphi}
\newcommand{\lan}{\langle}
\newcommand{\ran}{\rangle}

\begin{document}

\date{Started April 12, 2010 
Last modified on {\today}.}
\title{Planar Pixelations and Image Recognition}
\author{Brandon Rowekamp}
\address{Department of Mathematics, University of Notre Dame, Notre Dame, IN 46556-4618. 
}
\email{browekam@nd.edu}
\urladdr{\url{http://www.nd.edu/~browekam}}

\begin{abstract}Any subset of the plane can be approximated by a set of square pixels.  This transition from a shape to its pixelation is rather brutal since it destroys geometric and topological information about the shape.  Using a technique inspired by Morse Theory,  we    algorithmically produce a PL approximation of the original shape using only information from its pixelation.  This approximation converges to the original shape in a very strong sense: as the size of the pixels goes to zero we can recover  important geometric and topological invariants of the original shape  such as Betti numbers,  area, perimeter and curvature measures.
\end{abstract}

\maketitle

\tableofcontents

\section*{Introduction}
\label{s: 0}
\setcounter{equation}{0}
A common problem in computational topology is to try to recover an object embedded in $\bR^n$ when only some distorted version of it is known.  Inspired by digital imaging we consider specifically the problem of a pixelated subset of the plane,  which is to say that it has been replaced by a set of square pixels on a grid.  Since it is common to represent images by a grid of small pixels, it seems intuitive that any object could be recovered from its pixelations.  Indeed,  when the resolution of an image is fine it is difficult to detect any difference from the original using only the human eye, and we may expect to be able to recover even deep geometric invariants.lated subset of the plane,  which is to say that it has been replaced by a set of square pixels on a grid.  Since it is common to represent images by a grid of small pixels, it seems intuitive that any object could be recovered from its pixelations.  Indeed,  when the resolution of an image is fine it is difficult to detect any difference from the original using only the human eye, and we may expect to be able to recover even deep geometric invariants.

  A pixelation associated to a subset of a plane is simply the set of all pixels in a grid of a certain size which touch that subset (see Figure \ref{fig: pixelation0}).  By changing the size of the grid, we can view pixelations of the object with finer and courser resolutions.  We call the pixelation associated to a set $S \subset \bR^2$ created from a grid of side-length $\ep$ $P_\ep(S)$, or the $\ep$-pixelation of $S$ (see Definition \ref{def: pix}).  The variable $\ep$ is called the \emph{resolution} of the pixelation.

\begin{figure}[h]
\centering{\includegraphics[height=2in,width=2in]{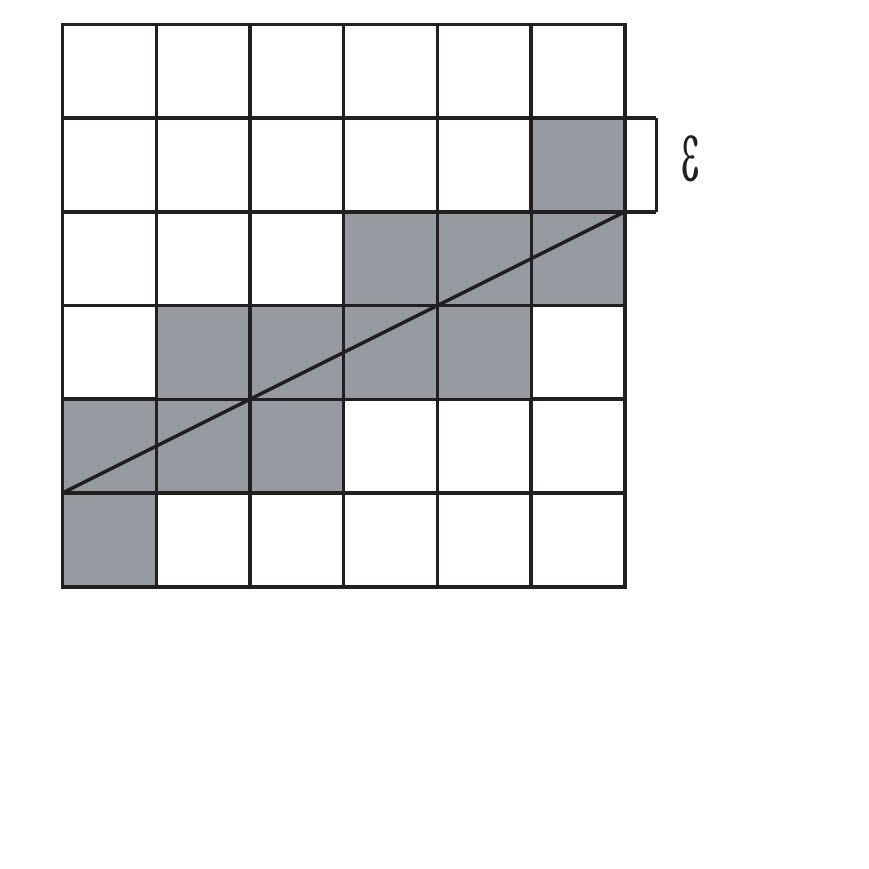}}
\caption{\sl An example pixelation of a line with slope $\frac{1}{2}$}
\label{fig: pixelation0}
\end{figure}

For small resolutions these pixelations seem to resemble the original shape.  However, we need a precise notion of what it means to resemble the original object.  At the very least it seems that the important topological and geometric invariants of the pixelation will be close to the corresponding variants of the original set.  Even when viewing simple examples, we can quickly see that not all invariants will converge for small resolutions.  Most obviously, the total curvature of $P_\ep(S)$ will tend to increase to infinity, since the pixelation contains many corners.  Similarly length of the boundary may not converge.  To see this consider the line segment from $(0,0)$ to $(1,1)$.  The length of the line is $\sqrt{2}$, but each of its pixelations will have a boundary whose length is close to $4$.

We see that pixelations do destroy various geometric invariants.  However, it may be that convergence of these invariants is too much to ask for.  A much simpler request is for the pixelation to eventually converge in Euler characteristic.  Unfortunately not even this is guaranteed.  Some pixelations will contain fake cycles which do not correspond to any cycle of the original set (indeed arbitrarily many may appear even if the original set was contractible).  Worse yet, these fake cycles may not disappear for small resolutions.  Section \ref{s: a} has further information on this phenomenon, but for now it is sufficient to consider the example of a line of slope $1$ and a line of slope $\frac{2}{3}$ (which is shown in Figure \ref{fig: 2hole0}).  The $\ep$-pixelations have certain scale invariance properties which guarantee that a fake cycle will always appear, even for very small $\ep$.

\begin{figure}[ht]
\centering{\includegraphics[height=2in,width=2in]{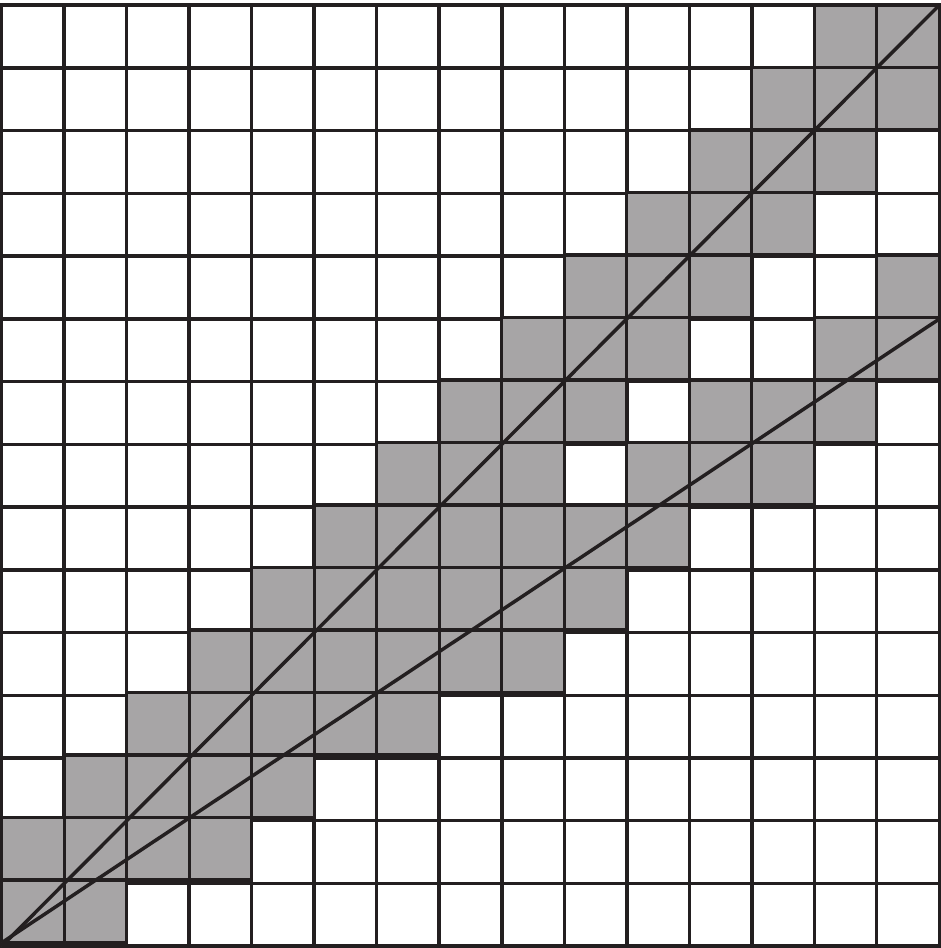}}
\caption{\sl The pixelation of the angle $A\left(1,\frac{2}{3}\right)$ contains two holes.}
\label{fig: 2hole0}
\end{figure}

We see that the process of pixelation can destroy a great deal of mathematical information of the set.  However, to the eye, the pixelation can look very close to the original set.  Therefore it seems that it should be possible to use only the pixelation to reconstruct the original set.  In this paper we will create an explicit algorithm which can reconstruct a PL set  $S$ from only its pixelations.  However, to motivate this algorithm and to generate the theorems needed to show that this algorithm works, we will need to examine simpler examples of pixelations.  

In Section \ref{s: 2} we deal with pixelations of the graphs of $C^2$ functions.  Theorem \ref{thm: ivt}, an analogue of the intermediate value theorem, tells us that these pixelations must be contractible.  Therefore we do not have to worry at this stage about fake cycles.  To approximate the graph of a function we attempt to use a secant line approximation.  This is done by connecting points within the pixelation of the function to create a PL approximation.  However, additionally error will be generated from the fact that our sample points are being selected not from the graph of function, but merely near to the graph.

Theorem \ref{thm: onesecant} gives us bounds on the error between a secant line of our function and the approximate secant line from the pixelation.  Note that these bounds depend both on the resolution $\ep$ and the width of the secant line.  We want to choose our PL approximation in a strict algorithmic way such that these errors vanish.  To do this we use the notion of a spread function.

A \emph{spread function} is simply a function $\sigma:\bR^+\to  \bZ^+  $ such that the following conditions hold:
\[
\lim_{\ep \searrow 0} \sigma(\ep) = \infty, \;\;\lim_{\ep \searrow 0} \ep \sigma(\ep) = 0.
\]
The spread function tells  how wide the secant lines in our approximations should be in terms of the resolution of the pixelation.  That is, since a pixel is defined to be $\ep$ wide, the secant lines should be arranged to be about $\ep \sigma(\ep)$ wide.  Theorem \ref{thm: profileApprox} shows that the two conditions of the spread $\sigma$ guarantee Sobolev $W^{1,p}$ convergence of the secant line approximation.


In Section \ref{s: 3} we generalize the approximation technique to handle \emph{elementary regions}, which are defined to be regions which line in between the graphs of two $C^2$ functions (note that the graph of a $C^2$ function is just a special case of an elementary region).  In this case we also cannot have fake cycles.  We can easily generalize the secant line approximation of functions to elementary regions by approximating the top and bottom regions of the elementary region.  It is then simple to show that we maintain the previous Sobolev $W^{1,p}$ convergence on the boundary of the approximation.

Here the notion of a spread becomes very useful.  By restricting the spread function $\sigma$, we can dramatically improve the convergence of the approximation.  Indeed if require that
\[
\lim_{\ep \searrow 0} \ep (\sigma(\ep))^2 = \infty,
\]
so that $\sigma$ increases very quickly, but still does not increase faster than $\ep$ decreases, we can recover the total curvature of the boundary of the elementary region.  This is shown in Propositions \ref{prop: c2curv}, \ref{prop: piece-curv} and Corollary \ref{cor: piece curv} for sets of increasing generality.  However each of these sets is defined as the region between the graphs of two functions.

At this point although we have created an algorithm which strongly approximates an elementary regions, but we have completely ignored simple sets such as the union of two intersecting  lines.  Such a set  is potentially difficult to approximate, since its pixelation may have more cycles than the original set.  In Section \ref{s: 4} we tackle the problem of homotopy type as it relates to the pixelations of $PL$ sets.

Corollary \ref{cor: seperate} of Theorem \ref{thm: inc} shows that eventually there is a bijective correspondence between the connected components of the pixelation and the connected components of the original set $S$.   A cycle in a (reasonably behaved) planar set corresponds to a  hole, i.e., a bounded connected component of  the complement of the set.  We can again use Theorem \ref{thm: inc} to show that every  hole of the original set  $S$ will eventually correspond to a hole of the pixelation.  Therefore any defect in homotopy type is caused by the addition of holes in the process of pixelation.  This means that to ensure convergence in homotopy type, we need only to delete the extra holes using only information from the pixelation.

An intuitive way to distinguish fake holes from real holes is to note that any real hole must take up an actual area in the plane.  However, fake holes tend to be composed of a relatively  small number of really small  pixels.  Therefore for small resolutions $\ep$ the fake cycles will   have small areas.   Appealing as it may sound, this idea is difficult to implement rigorously because we do not have  an accurate way of defining what ``small area'' means.     We rely instead on a more robust approach inspired from Morse theory.

We consider  the linear function on the  Cartesian plane that associates to each point its $x$-coordinate.  For   simplicity  we assume that its restriction $\ell_S$  to our    PL set $S$ is a  stratified Morse function in the sense of Goresky-MacPherson, \cite{GM}. In our case this simply means that no two vertices of our $PL$ set $S$ lie on the same vertical line.

The topology of the level sets of $\ell_S$ is  determined by the counting   function $\bn$
\[
\bn(x):=\mbox{the number of components of $\ell_S^{-1}(x)$}.
\]
A pixelated version of $\bn$ is the function $\bn_\ep$, where $\bn_\ep(x)$ is  the number of connected components of the column of $P_\ep(S)$ located at $x_0$.  Note that for a cycle to appear in $S$, the function $\bn_S$ must vary and likewise with the pixelation and $\bn_\ep$.  Therefore determining whether a cycle in $P_\ep(S)$ is a fake or really corresponds to a cycle in $S$ is a matter of determining how closely the function $\bn_\ep$ agrees with $\bn_S$.

The Separation Theorem (Theorem \ref{thm: sep}) states that if $x$ is not a critical value of $\ell_S$ then $\bn(x)=\bn_\ep(x)$ for all $\ep$ sufficiently small.   The critical values of $\ell_S$ correspond to   points  jumping points of $\bn$, i.e., points of of  discontinuity  of $\bn$.  Using the Separation  Theorem we prove several results indicating that the jumping points of $\bn_\ep$ and the jumping points of $\bn_S$  are not far apart. This culminates in Proposition \ref{pro: pix noise bound} which states that in the PL case jumping points of $\bn_S$ must appear close to jumping points of $\bn_\ep$ and vice versa.  Here ``close'' means within $\ep \nu(S)$ where $\nu(S)$ is an integer determined by $S$  and called the \emph{noise range}.

We ultimately want to classify an interval around every jumping point of $\bn_\ep$ as ``noise.''  To do this we must estimate the noise range $\nu(S)$, but this is determined by the original set $S$ and thus unknown given only the pixelation.  However if we estimate the noise range using a spread $\sigma$ note that for small $\ep$ we can guarantee that $\sigma(\ep) > \nu(S)$ while $\ep \sigma(\ep)$ is small.  Therefore if we classify an interval of width about $\ep \sigma(\ep)$ around each jumping point of $\bn_\ep$ as noise, we can ensure that eventually outside of the noise $\bn_\ep = \bn_S$ while the noise region remains small.

The noise intervals are important since they are chosen to eventually contain all fake cycles of $P_\ep(S)$ while taking up a vanishingly small portion of the real line.  Therefore to approximate homotopy type within noise intervals we need only cover every cycle (since any cycle appearing will be a fake cycle).  Since noise does not take up much of the plane, we do not need to be careful as to how we approximate within noise intervals, so we will do this by covering each connected component of the noise with the smallest rectangle which covers it.

Call the  components  of the complement of the noise intervals  \emph{regular intervals}.  By the definition of noise, the regular intervals contain no jumping points of $\bn_\ep$.  Therefore $\bn_\ep$ is continuous on the regular intervals, which implies that the parts of $P_\ep(S)$ which lie over regular intervals look like strips that span the entire regular interval.  Each strip can be interpreted as the pixelation of an elementary set and we have explained how to deal with such objects.

Putting together all these facts we obtain  Algorithm \ref{alg: process}  which associates to each $\ep$-pixelation a $PL$ set.    This algorithm works by using the function $\bn_\ep$ and the spread $\sigma$ to divide the pixelation into noise and regular intervals, and then approximates within each interval using the appropriate results.  That is to say within noise intervals it covers each connected component with a rectangle and within regular intervals it connects the tops and bottoms of every $\sigma$-th column.  An example of the result of this algorithm can be seen in Figure \ref{fig: stage5} and the Algorithm is restated in terms of a computer program in Appendix \ref{s: d}.

\begin{figure}[ht]
\centering{\includegraphics[height=2in,width=2in]{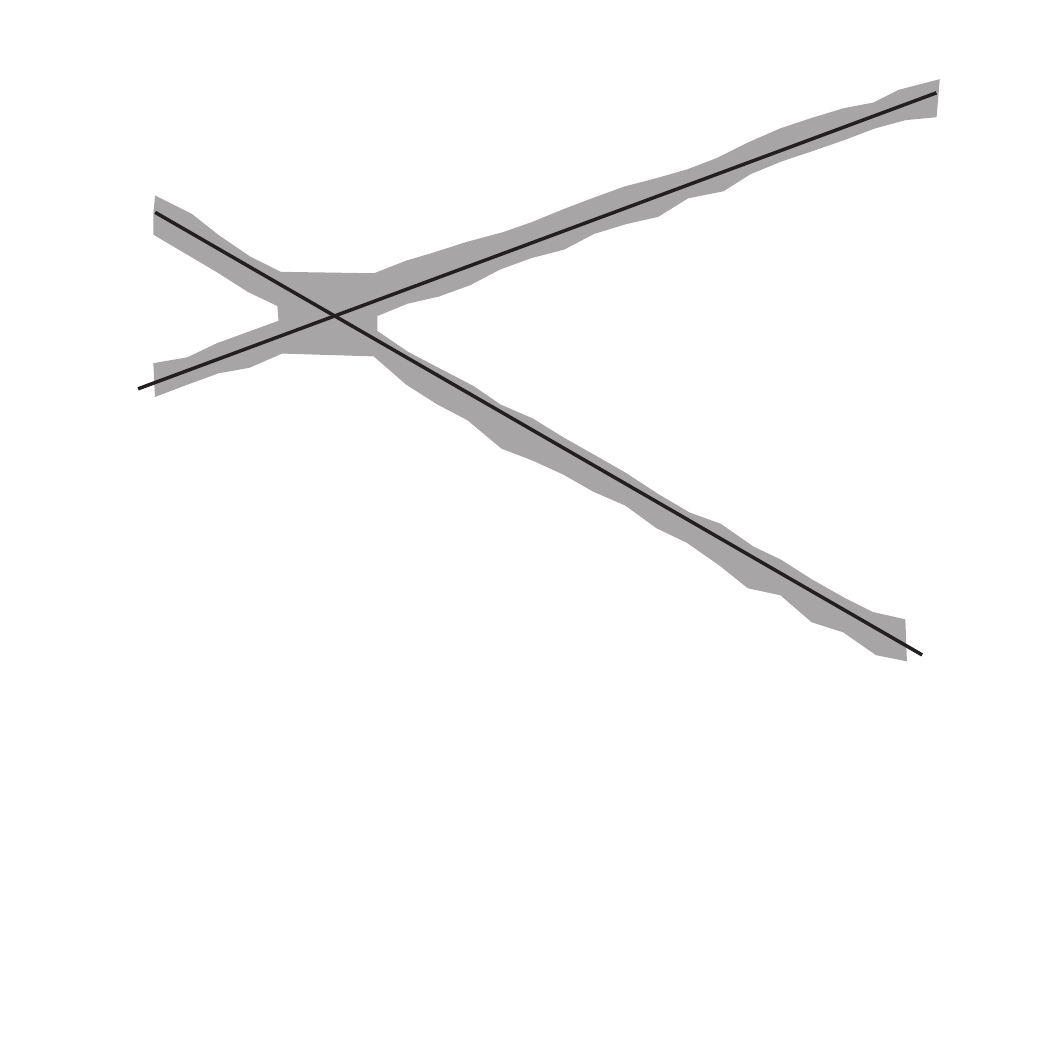}}
\caption{\sl  The result of Algorithm \ref{alg: process} as applied to the intersection of two red lines.  Note that near the point of intersection the approximation resembles a rectangle, since this lies within a noise interval.}
\label{fig: stage5}
\end{figure}

We wish to state that Algorithm \ref{alg: process} captures all the important geometric and topological invariants.  To do this we use the language of \emph{normal cycles}. These are $1$-currents   with supports contained in  the bundle of unit tangent vectors. Appendices \ref{s: b} and \ref{s: c} are included for readers unfamiliar with these objects.  For   purposes  of this  general  discussion it suffices to mention  that the normal cycle of a subanalytic subset $S$ of the plane is a piecewise $C^1$  closed curve   in the unit tangent sphere bundle object which encodes geometric and topological information  about of the set $S$ through integration of certain canonical 1-forms.  

In general the normal cycle can be thought to be approximated by the collection of all unit outer normal vectors of the set.  For example, the normal cycle of a bounded domain with $C^2$-boundary is  the graph of the Gauss map of the boundary.  The normal cycle of the square resembles Figure \ref{fig: normal}.  Of course the normal cycle is actually a subset of the sphere bundle of the plane, not the plane itself.  The fiber of the normal cycle over a point in the  interior of an edge each point will consist of only a single outer normal vector. The fiber over  a corner  consists of an entire quarter-circle (all possible outer normal vectors for the corner).  Therefore the length of the normal cycle will be the perimeter of the square plus $2\pi$, or the perimeter of the square plus its total curvature.  
\begin{figure}[ht]
\centering{\includegraphics[height=2in,width=2in]{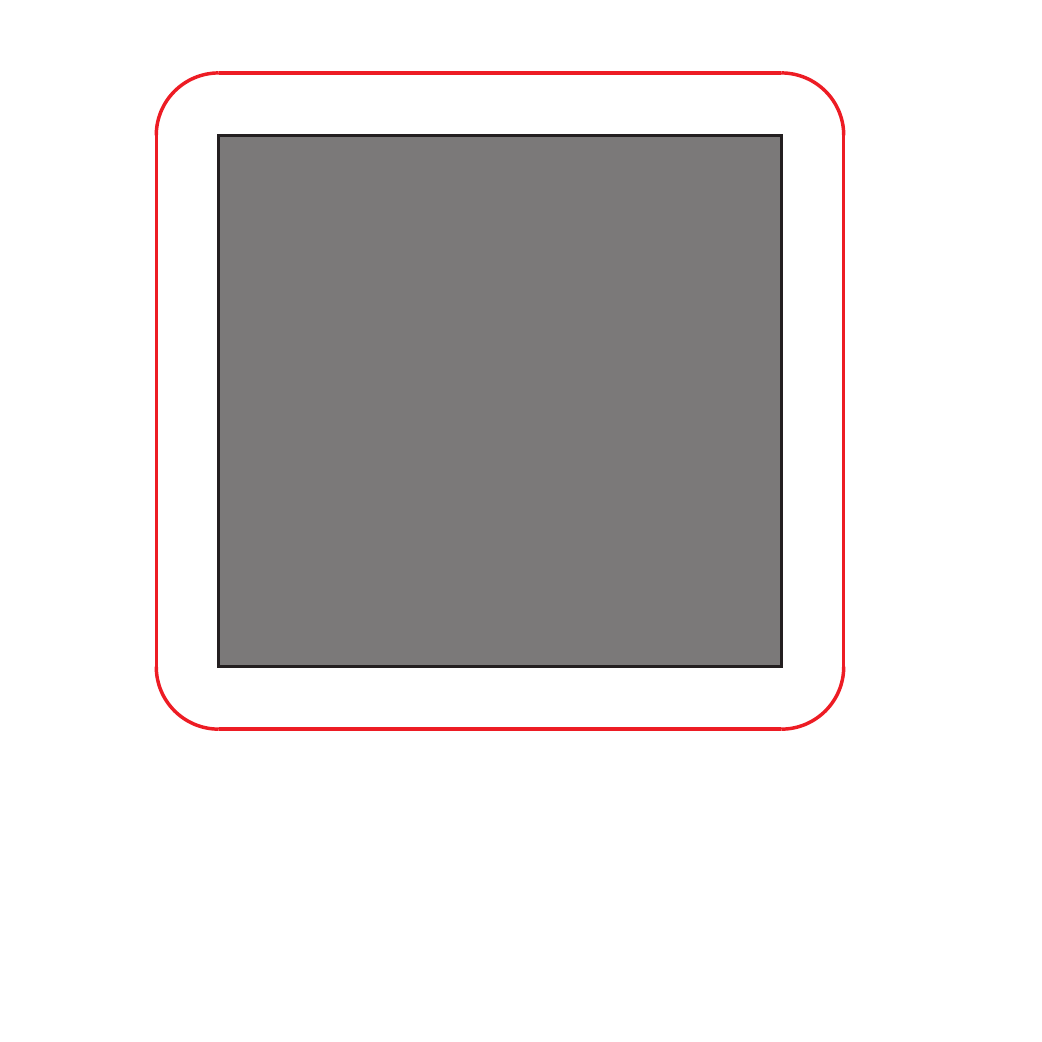}}
\caption{\sl  A representation of the normal cycle of the square.}
\label{fig: normal}
\end{figure}

If we prove that the normal cycles of our approximations converge weakly to the normal cycle of the original set, we will have shown that our approximation recovers the important geometric information of the set.  To prove this we use the powerful General Convergence Theorem proved by Joseph Fu in \cite{Fu1}.  This result is restated in this paper as Theorem \ref{thm: fu}.  This theorem tells us that to have convergence in normal cycles we must prove three things about our approximations.  First, they must all fit within some compact subset of the plane.  This is clear from the Algorithm \ref{alg: process}.  Secondly, the mass of the normal cycles of the approximations cannot explode.  Since the mass of the normal cycle is highly dependent on the total curvature and the perimeter of the set, this follows from Corollary \ref{cor: piece curv}.  Finally we must have the Euler characteristic of the approximation converge to the Euler characteristic of the original set when restricted to an arbitrary half plane.  This requirement  of the approximation theorem is   the most challenging, but it turns  to be true as well.  Therefore Fu's General Convergence Theorem implies convergence of our approximations in normal cycle (Theorem \ref{th: main}).

Theorem \ref{th: main} is a powerful result with many corollaries, since the normal cycle encodes many important invariants.  Most particularly it implies that the approximation generated by Algorithm \ref{alg: process} will recover both the Euler Characteristic as well as the total curvature on the boundary of any generic PL set.

\newpage
\section{Definition of Pixelations}
\label{s: 1}
\setcounter{equation}{0}
In this section we give the definition of a pixelation and detail the most basic of properties of a pixelation.  Questions about recovery of the original set are not tackled until future sections.

\begin{definition} (a) Let $\ve>0$.  Then we define an \emph{$\ve$-pixel} to be the square in $\bR^2$ of the form
\[
S_{i,j}(\ve)=[(i-1)\ve, i\ve)]\times [(j-1)\ve, j\ve)]\subset \bR^2,\;\; i,j\in \bZ.
\]
The \emph{$\ep$ grid of pixels} is the collection of all such $\ep$-pixels.

\noindent (b) A union of finitely many $\ve$-pixels is called an $\ve$-\emph{pixelation}.  The variable $\ve$ is called the $\emph{resolution}$ of the pixelation.

\noindent (c) For any bounded subset $S\subset \bR^2$ we define the $\ve$-\emph{pixelation} of$S$ to be the union of all the $\ve$-pixels that intersect $S$. We denote the $\ve$-pixelation of $S$ by $P_\ve(S)$. The pixelation of a function $f$ is defined to be the piexelation of  its graph $\Gamma(f)$. We will denote this pixelation by $P_\ve(f)$.
\label{def: pix}
\end{definition}

The variable $\ep$ is called the resolution, because we think of a pixelation as a representation of a computerized image.  A smaller choice of $\ep$ will cause a shape to be approximated by a greater number of pixels, which is like saving an image as a higher resolution file.

Observe that  if $\|\bullet\|_\infty: \bR^2\ra [0,\infty)$ is the norm
\[
\|(x,y)\|=\max\{|x|,|y|\},
\]
then the the pixel $S_{i,j}(\ve)$ can be identified  with the $\|\bullet\|_\infty$-closed ball of radius $\ve/2$ and center
\[
c_{i,j}(\ve):=\left(\, (i-\frac{1}{2})\ve, (j-\frac{1}{2})\ve,\right).
\]  
Therefore a pixelation can be thought of as set of points chosen near the original shape.  However, rather than being chosen randomly these points are chosen to be the closest points on a regular lattice.  This by no means makes the problem of a recovering a set from its pixelation trivial.  As we will see, even the homotopy type can be lost during the process of pixelation.  Still, the added structure of a pixelation will allow us to make stronger approximations than can be done with regular pixelations.  For example it is reasonable to expect geometric information like total curvature of the boundary to be preserved in an approximation.

The $\ep$-pixelation associated to a bounded set $S$ is the collection of pixels in this grid that contain a part of $S$.  A pixelation is a sort of ``fattening'' of the set, since we include a pixel if any part of the set lies within the pixel (see Figure  \ref{fig: pixelation}.)  In this way it is somewhat similar to a tube around the set.  However a pixelation differs from a tube in that the boundary of any pixelation will have corners, and that the distance from the boundary to the set may vary significantly (since pixelations are formed by taking a collection of squares in a grid.)  As we will see in the following sections, the fact that pixelations approximate via a grid means that the geometric properties of pixelations may not converge to the geometric properties of the original set as $\ep \to 0$ (for example, the normal cycle of the pixelation will not converge to the normal cycle of the original set.)  The good news is that geometric information about the original set can be obtained from its pixelations in indirect ways.

\begin{figure}[h]
\centering{\includegraphics[height=3in,width=3in]{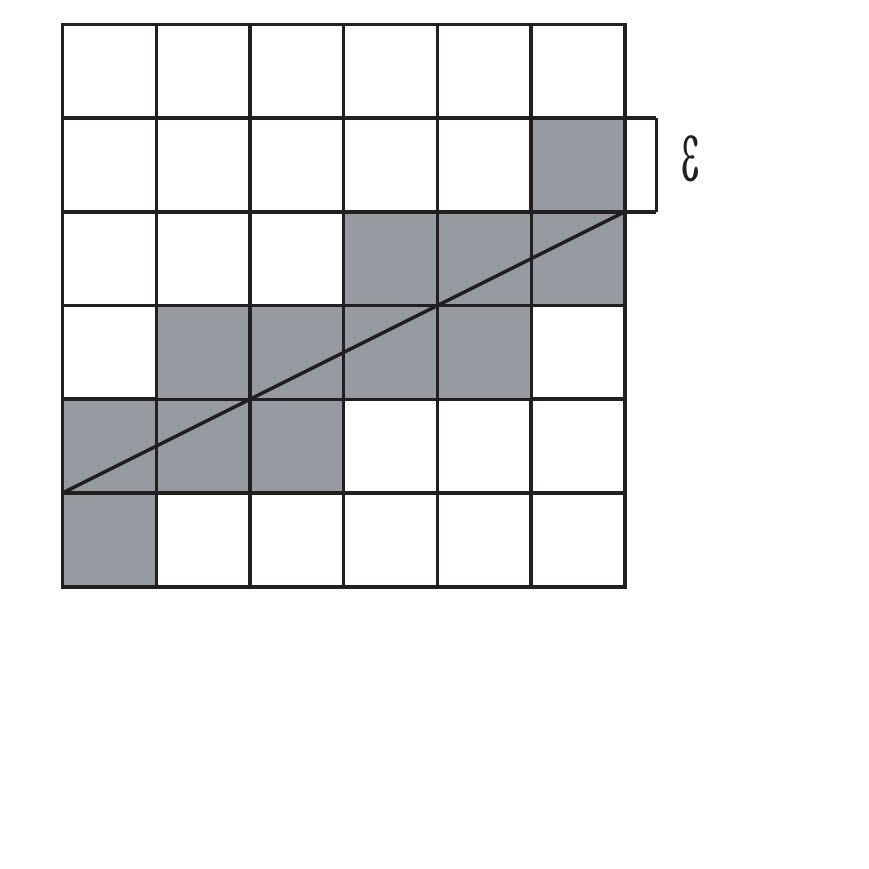}}
\caption{\sl An example pixelation of a line with slope $\frac{1}{2}$}
\label{fig: pixelation}
\end{figure}

It is easy to see that the pixelation of an set can be fundamentally different from a set.  For example, every pixelation of a line segment will have many corners (as long as the line is not parallel to the $x$ or $y$-axis.)  This implies that as $\ep$ becomes small, the total curvature of the boundaries of pixelations of a line segment will increase in an unbounded fashion.  This contradicts the fact that the line itself has no curvature and shows that essential geometric features of the set are not preserved in the pixelation.

In fact, though it is more difficult to see, pixelations do not preserve the topological properties of the set.  Consider the set $S$ consisting of a line of slope $\frac{1}{2}$ and a line of slope $1$ which both start at the origin.  Then it can be shown that for every $\ep$, $P_\ep(S)$ contains a hole which prevents $P_\ep(S)$ from being contractible (see Figure \ref{fig: halfhole}.)  The position of the hole depends on $\ep$, but it will always exist for any $\ep$.  This example is far from unique.  In fact, by altering the slope of the two lines we can create a set whose pixelations contain any desired number of holes.   This shows that pixelations do not preserve the topological properties of the underlying set.

\begin{figure}[h]
\centering{\includegraphics[height=3in,width=3in]{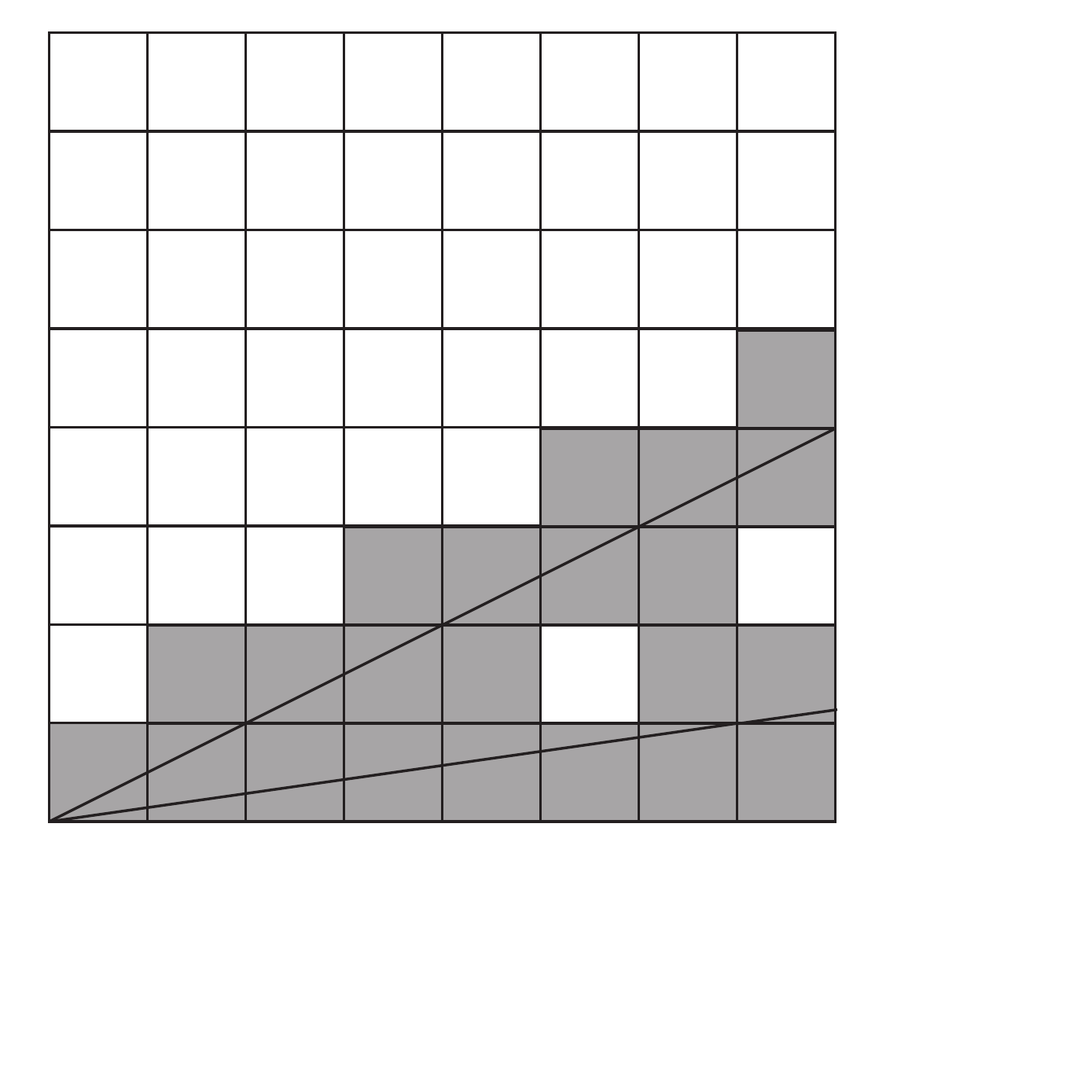}}
\caption{\sl If $S$ consists of a line of slope $\frac{1}{2}$ and a line of slope $1$ starting at the origin, then $P_\ep(S)$ is not contractible for any $\ep$.}
\label{fig: halfhole}
\end{figure}

However, when $\ep$ is very small, it is easy for the human eye to detect the kind of shape that probably generated a pixelation.  Mathematically this corresponds to the fact that it is possible to create an approximation from the information in the pixelation which will respect topological and geometric features.  The approximation can be algorithmically generated, but we delay defining it until much later in the paper (since the motivation for it rests on many properties of pixelations.) 

The remainder of the paper will examine specifically what pixelations preserve and what they destroy.  Furthermore it will explain how to recover the information ``lost'' in the process of creating a pixelation.

We end this section with a few basic properties of pixelations which will be useful.

\begin{proposition}
For any two sets $U$ and $V$ we have
\[P_\ep(U \cup V) = P_\ep(U) \cup P_\ep(V)\]
\label{pro: union}
\end{proposition}

\begin{proof}
This follows immediately from the definition of a pixelation as the union all pixels which intersect a set.
\end{proof}

\begin{proposition}
For each $\ep$ define $T_\ep^\infty(S)$ to be
\[
T_\ep^\infty(S) := \{x \in \bR^2: \exists p \in S \text{ such that } \|x - p\|_\infty \le \ep\}
\]
i.e. the tube of radius $\ep$ of $S$ in the $\| \bullet \|$ norm.
Then
\[
P_\ep(S) \subset T_\ep^\infty(S)
\]  
\end{proposition}

\begin{proof}
Suppose $x \in P_\ep(S)$.  Then $\exists p \in S$ such that $x$ and $p$ lie within the same $\ep$ pixel.  But by the definition of pixels, this implies that $\|x-p\|_\infty \le \ep$.  Therefore $x \in T_\ep^\infty(S)$.
\end{proof}

\newpage
\section{Pixelations of Functions and First-Order Approximation}
\label{s: 2}
\setcounter{equation}{0}
In order to reach our overall goal of recovering sufficiently nice sets from their pixelations, we must start by examining simple sets.  The graphs of $C^2$ functions on compact domains provide a good starting point.  We will find that in this case many analogues of basic theorems from calculus appear in a pixelated version, and these facts will allow us to approximate the graphs of these functions.

Before we proceed in our investigation we need to introduce a basic vocabulary that will be used throughout the paper.

\begin{definition}  Fix $\ve>0$ and a bounded set $S\subset \bR^2$.

\begin{enumerate}

\item A point $a\in\bR$ will be called $\ep$-generic if $x\in\bR\setminus \ep\bZ$.   For such a point $a$ we denote by  $I_\ve(a)$ the interval  of the form $[n\ve, (n+1)\ve]$, $n\in\bZ$ that contains $a$. 

\item  For an interval $[a,b] \subset \bR$ we define the vertical strip
\[
\eS_{[a,b]} := [a,b] \times \bR
\]
For every $k\in \bZ$ we  denote by $\eS_{\ep,k}$ the vertical strip
\[
[k \ve, (k+1) \ve] \times \bR = \eS_{[k\ep, (k+1)\ep]}
\]
 For any $\ep$-generic point $a\in\bR$ we denote by $\eS_\ep(a)$ the strip  
\[
I_\ve(a)\times\bR=S_{\ve,k}, \;\; k:=\lfloor x/k\rfloor.
\]

\item A \emph{column} of $P_\ve(S)$ is   the intersection  of $P_\ve(S)$ with a vertical strip $\eS_{\ep,k}$, $k\in\bZ$.  The connected components of  a column are called \emph{stacks}.

\item For every $\ep$-generic $a \in \bR$, we define the \emph{column} of a pixelation $P_\ep(S)$ over $a$ to be the set  
\[
C_\ve(S,a):=\eS_\ve(a)\cap P_\ve(S).
\]
In other words,   $C_\ep(S,a)$  is the union of the pixels in $P_\ep(S)$ which intersect the vertical  line $\{x=a\}$.   When $S$ is the graph of a function $f$, we will use the notation $C_\ep(f, a)$ to denote   the column over $a$ of the pixelation $P_\ve(f)$.

\item  We define the \emph{top, bottom} and respectively \emph{height}  of  a column to be the quantities
\[
T_\ep(S,a):= \sup \{y: (a,y)\in C_\ep(S,a)\},
\]
\[
B_\ep(S,a):=\inf \{y: (a,y)\in C_\ep(S,a)\},
\]
and respectively 
\[
h_\ep(S,a) := \frac{1}{\ve}\bigl(\,T_\ep(S,a) - B_\ep(S,a)\,\bigr).
\]
\end{enumerate}
 \qed
 \label{def: bu}
\end{definition}

Since stacks are constrained to be within the same column, a stack can usually be thought of as two pixels together with all the pixels that fall vertically in between them.


This language allows us to state an analogue of the Intermediate Value Theorem.

\begin{theorem}[Pixelated Intermediate Value Theorem] If $f:[a,b]\ra \bR$ is a continuous function, then for every $x\in[a,b]\setminus \ve\bZ$, the column  $C_\ep(f,x)$ consists of exactly one stack.
\label{thm: ivt}
\end{theorem}

\begin{proof} We argue by contradiction. Suppose that the column $C_\ep(f,x)$  has at least two stacks.  This means that within the column over $x$ there are two stacks with a gap of empty pixels in between them.  This implies that there  exist  an  interval $[c \ep, d \ep]$, not contained in the range of $f$,  and real numbers $ x_1, x_2 \in I_\ve(x)$ such that 
\[
f(x_2) < c \ep < d\ep <f(x_1).
\]
  Since $f$ is continuous, this contradicts the Intermediate Value Theorem.  Therefore the column over $x$ has only one stack.
\end{proof}


Note that the Pixelated Intermediate Value Theorem immediately implies that a cycle cannot appear in the pixelation of a continuous function.  Therefore the pathological behavior witnessed in Appendix \ref{s: a} will not occur in the pixelations of $C^2$ functions, which is our first hint that pixelations of $C^2$ functions are the nicest type of pixelation to work with.

This analogue of the Intermediate Value Theorem tells us that we can think of the pixelation of a function that assign a stack of pixels to each value in the domain of the function.  The stacks themselves are worth investigating.  To generate a stack in the pixelation, the function must attain values near the top and bottom of the stack within that column.  This means that the average change of a function over a column must be related to the height of the column.  This relation will be shown in the following two theorems.  The first tells us that the height of a stack is bounded by the derivative of the function.

\begin{proposition}
Let $f: [a,b]\ra \bR$ be a Lipschitz continuous  function. We denote  by $\|f'\|_\infty$ the best Lipschitz constant, i.e.,
\[
\|f'\|_\infty:=\sup_{x\neq y} \frac{|f(x)-f(y)|}{|x-y|}.
\]
Then  for any $\ep>0$  and any $x\in [a,b]\setminus \ep\bZ$ we have
\[
h_\ep(f,x) \le ||f'||_\infty+2.
\]  
\label{thm: difbound}
\end{proposition}

\begin{figure}[h]
\centering{\includegraphics[height=3in,width=3in]{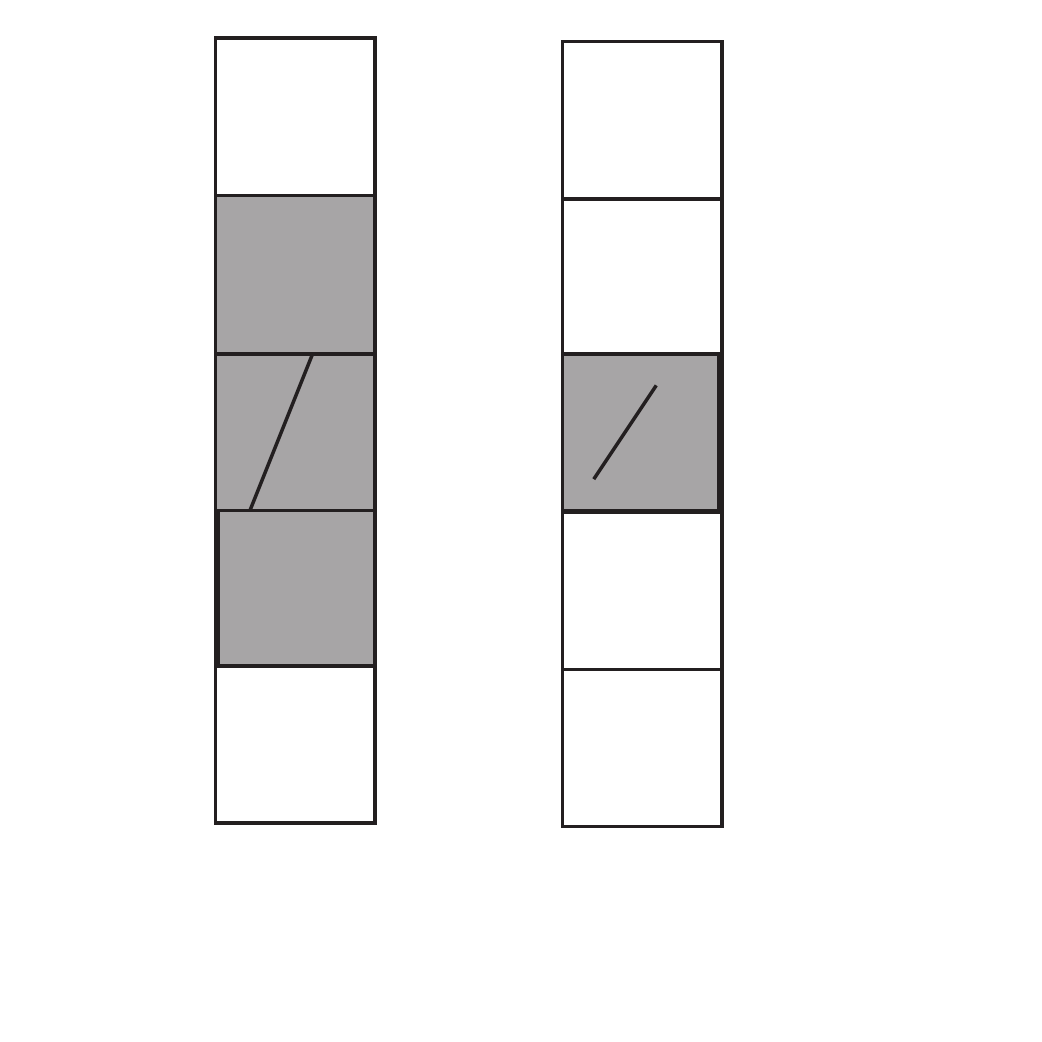}}
\caption{\sl If the line on the left were any shorter it would generate the pixelation on the left. Thus, the minimum range of a function in a column is two less than the height of the column.}
\label{fig: topbottom}
\end{figure}

\begin{proof} We argue by contradiction. This means  that there exist $\ep>0$ and $x\in [a,b]\setminus \ep\bZ$ such that 
\begin{equation}
h_\ep(f,x) > ||f'||_\infty + 1.
\label{eq: 241}
\end{equation}
 Set $n:=\lfloor x/\ep\rfloor\in\bZ$ so that   $n\ve < x< (n+1)\ve$. The inequality (\ref{eq: 241})  is equivalent to the condition
\[
T_\ep(f,x) - B_\ep(f,x) > \ep(||f'||_\infty+2).
\]
This implies that 
\[
\exists c, d\in \bigl( n\ve, (n+1)\ve\,\bigr)\;\;\mbox{such that}\;\;|f(c)-f(d)|> ||f'||_\infty \ep,
\]
 because the pixelation can be one pixel higher or lower than the graph of $f$; see Figure \ref{fig: topbottom}. We have thus  found $c$ and $d$ such that 
\[\frac{|f(c) - f(d)|}{|c-d|} > \frac{||f'||_\infty\ep}{\ep} = ||f'||_\infty\]
But this contradicts the  definition of $\|f'\|_\infty$. 
\end{proof}

\begin{remark} (a) Observe that if $x\in[a,b]\cap \ep\bZ$ then
\[
h_\ep(f,x)\leq 2(\|f'\|_\infty+1).
\]
(b) The above  theorem can be used to bound the height of a stack by the derivative  of $f$ over that stack by considering the restriction of $f$ to that column.\qed
\end{remark}

The next theorem suggests that  the height of a stack is close, in a crude way, to the absolute value of the derivative in the column for sufficiently small $\ep$.  The intuition is that if $f$ is $C^2$, then any sufficiently small section of its graph will be very close to the graph of a line, and the height of a stack in the pixelation of the graph of a line must be very close to the slope of that line.

\begin{proposition}
Let $f$ be a $C^2$ function on $[a,b]$.  Then, for any  $\ep>0$ and any $x_0\in [a,b]\setminus \ep\bZ$ we have
\[
|f'(x_0)|\leq  h_\ep(f,x_0)| + \|f''\|_\infty\ve +2.
\]
\end{proposition}

\begin{proof} Recall that $I_\ve(x_0)$ indicates the interval of the form $[n\ve, (n+1)\ve]$, $n\in\bZ$, that contains $x_0$.  Let $L$  be the linearization of $f$ at $x_0$, i.e., the function
\[
L:\bR\ra \bR,\;\;L(x):=(x - x_0)f'(x_0) + f(x_0).
\]
 From Taylor's Theorem we deduce that \[
|L(x)-f(x_0)| \le \frac{||f''||_\infty}{2} (x-x_0)^2 \le \frac{1}{2} ||f''||_\infty \ep^2\leq \ep,\;\;\forall x\in I_\ve(x_0)\cap[a,b].
\]
This  implies that
\[
T(L,x_0)\leq T(f,x_0)+\frac{1}{2}\|f''\|_\infty\ve^2+\ve,\;\; B(L,x_0)\geq B(f,x_0)- \frac{1}{2}\|f''\|_\infty\ve^2-\ve.
\]
Hence,
\[
h_\ve(L,x_0) \leq h(f,x_0)+ \|f''\|_\infty\ve+2.
\]
On the other hand,   $L$ changes  by $f'(x_0) \ep$ over the width of a pixel, which implies
\[
|f'(x_0)| \le  h_\ep(L,x_0) 
\]
\end{proof}

We would now like to approximate $f$ through its pixelations.  An intuitive way to do this is by creating a $PL$ function whose graph lies within $P_\ep(f)$,  since this approximation would generate a similar pixelation.  However, in practice this method requires making many choices, and it can be cumbersome to check that the graph of the function lies within the pixelation of $f$.  A better approach is to form a $PL$ function by connecting points within $P_\ep(f)$ by straight line segments.  In order to preserve higher order properties of the function, we must be careful when choosing these points.  

For example,  suppose that the vertices of the $PL$ approximation are   centers of the top pixels in  each column.  Then the slopes of each line making up the $PL$ function would have integer slopes, and could never approximate the derivative of $f$ arbitrarily well.  We could increase the accuracy of the approximation by connecting the maxima of every other column, or every third column.  However, doing this decreases the speed in which the approximation converges pointwise to $f$.  This is the central tension involved in approximating a function from its pixelation.

Before we further describe how to approximate a function from its pixelation, we show how to at least find the average change of the function, i.e.,  the slope of the secant line connecting the first and last points of $f$.

\begin{theorem}
Let $f$ be a $C^2$ function on the interval $[a,b]$. Fix $\ep > 0$.
Let $c, d \in \bR$ such that $(a,c) \in P_\ep(f)$ and $(b,d) \in P_\ep(f)$.
\[
g_\ep (x) := (x - a)\frac{d-c}{b-a} + c.
\]
Then
\begin{equation}
|f(x) - g_\ep(x)| \le 7(||f'||_\infty +2) \ep +  3||f'||_\infty(b-a)+||f''||_\infty (b-a)^2
\label{eq: 271}
\end{equation}
and
\begin{equation}
|f'(x) - g_\ep'(x)| = |f'(x) - \frac{d-c}{b-a}| \le (b-a) ||f''||_\infty + \frac{2(||f'||_\infty+2) \ep}{(b-a)},
\label{eq: 272}
\end{equation}
for all $x \in[a,b]$.
\label{thm: onesecant}
\end{theorem}

\begin{proof}
Note that it suffices to prove the above inequalities in the special case when $x\in [a,b]\setminus\ve\bZ$. For ease of notation let 
\[
M := ||f'||_\infty,\;\;N := ||f''||_\infty,\;\; m_f: = \frac{f(b) - f(a)}{b-a}
\]
and note that $m_f$ is the slope of the secant line to $f$ over the interval $[a,b]$.  Define
\[
m_g := \frac{d-c}{b-a},
\]
and note that this is the slope of $g_\ep$. Note that $c$ and $d$ differ from $f(a)$ and $f(b)$ by at most the height of a stack.  

\begin{lemma}
With $m_f$ and $m_g$ defined as above,
\begin{equation}
|m_f - m_g| \le \frac{2 (||f'||_\infty + 2)\ep}{b-a}.
\label{eq: 281}
\end{equation}
\label{lem: secant slope}
\end{lemma}

\noindent {\bf Proof of Lemma \ref{lem: secant slope}.}  Theorem \ref{thm: difbound} implies that a height of a stack in $P_\ep(f)$ is at most $M + 2$ so 
\[
|f(a) - c|, |f(b)-d| \le (M + 2) \ep.
\]
This immediately implies
\[ 
|m_f-m_g| = |\frac{f(b) - f(a)}{b-a} - \frac{d-c}{b-a}|\le \frac{1}{b-a} (|f(b) - d| + |f(a) - a|) \le \frac{2(M + 2) \ep}{b-a}.\proofend
 \]

The Mean Value Theorem implies that $\exists \xi \in [a,b]$ such that $f'(\xi) = m_f$.  Let  $h(x)$  be the linearization of  $f$ at  $\xi$,i.e.,
\[
h(x):= (x - \xi)m + f(\xi).
\]
  Thus we have
\begin{multline}
|g_\ep (x) - h(x)| = |(x - a) m_g + c - (x - \xi) m_f - f(\xi)| \\ \le |(x-a)m_g - (x-\xi)m_f| + |d - f(\xi)|
\label{eq: bound}
\end{multline}
First we will bound the first term on the right hand side of (\ref{eq: bound}).
\[
|(x-a)m_g - (x-\xi)m_f| = |(x-a) m_g - (x-\xi)m_g + (x-\xi)m_g - (x-\xi)m_f| \]
\[\le |m_g|(|x-a|+|x-\xi|)+|m_g-m_f|(|x-\xi|)
\]
To proceed further we need to use the following facts.
\begin{itemize}

\item  $|x-a|, |x-\xi|\leq |b-a|$.

\item  $|m_g - m_f| \le \frac{2(M+2) \ep}{b-a}$.  

\item Since $|m_f| \le M$, we have
\[
|m_g| \le M + \frac{2(M+2)\ep}{b-a}.
\]
\end{itemize}
Therefore
\[|(x-a)m_g - (x-\xi)m_f| \le 2 (b-a)(M + \frac{2(M+2)\ep}{b-a})+(b-a)\frac{2(M+2)\ep}{b-a}\]
\[= 2 M(b-a) + 6(M + 2)\ep.
\]
Now we bound the other term from (\ref{eq: bound}), $|c-f(\xi)|=|c-h(\xi)|$.   We have
\[
|c-f(\xi)|\leq |c-f(a)|+|f(a)-h(a)|+ |h(a)-h(\xi)|.
\]
Note that $|f(a) - c| \le (M+2)\ep$, since stacks are at most $M + 2$ pixels tall. Taylor's Theorem  implies that  
\[
|f(a) - h(a)| \le \frac{N (b-a)^2}{2}.
\]
 Finally $|h(\xi) - h(a)| \le M (b-a)$, since $h$ has slope at most $M$,  and $|\xi - a| \le (b-a)$. We conclude  that 
\[
|f(\xi) - c| \le (M+2) \ep + \frac{N (b-a)^2}{2} + M(b-a).
\]
Combining all of these bounds in (\ref{eq: bound}) we deduce 
\[
|g_\ep(x) - h(x)| \le 3 M (b-a) + 7(M+ 2)\ep + \frac{N(b-a)^2}{2}.
\]
Finally, 
\[
|f(x) - h(x)| \le \frac{N (b-a)^2}{2},
\]
 so that
\[
|f(x) - g(x)| \le 7(M+2) \ep +  3 M(b-a)+N (b-a)^2.
 \]
which is the bound (\ref{eq: 271}). 

The bound (\ref{eq: 272})   is obtained as follows.
\[
|f'(x) - g_\ep'(x)| = |f'(x) - m_g|\leq |m_g - m_f| +|m_f-f'(x)|
\]
\[
\stackrel{(\ref{eq: 281})}{\le }\frac{2(M + 2) \ep}{b-a}+|f'(\xi)-f'(x)|
\leq  \frac{2(M + 2) \ep}{b-a}+N\ep.
\]
\end{proof}

This Theorem tells us that the slope of a secant line to $f$ can be accurately approximated up to its derivative (by restricting $f$ to various intervals we can approximate any secant line.)  We would like approximate $f$ through the use of multiple secant lines.  However to accurately approximate a small secant line, we will need $\ep$ to be very small.  On the other hand, to accurately approximate $f$ we will need to use many (and therefore small) secant lines.  Thus we will need to know the size of a secant line that can be accurately approximated for a given $\ep$.  Before we can describe how to go about this process, however, we will need additional terms.

\begin{definition} Fix $\ve>0$ and a bounded set $S$. 

\begin{enumerate}
\item An $\ve$-\emph{profile} of $S$ is   a  set  $\Pi_\ve$   of points in the plane  with the following properties.

\begin{enumerate} 

\item Each point in  $\Pi_\ve$ is the center of  an $\ve$-pixel that intersects $S$. 

\item   Every  column of $P_\ve(S)$ contains precisely one point of $\Pi_\ve$.
\end{enumerate}

\item A  \emph{top/bottom $\ve$-profile} is   a profile  consisting of the centers  of  the highest/lowest pixels in  each column of $P_\ve(S)$.

\item An $\ve$-sample of $S$ is a subset of an $\ve$-profile.

\end{enumerate}\qed
\label{def: profile}
\end{definition}

\begin{definition} Suppose $p_1,\dotsc, p_N$ is a finite sequence of points in $\bR^2$.  (The points need not be distinct). We denote by
\[
\langle p_1,p_2,\dotsc, p_n\rangle
\]
the $PL$ curve defined  as the union of the straight line segments $[p_1,p_2],\dotsc, [p_{n-1},p_n]$.\qed
\end{definition}

Observe that   each $\ve$-profile $\Pi_\ve$ of a set is equipped with a linear order $\preceq$. More precisely if $p_1,p_2=(x_2,y_2)$ are points in $\Pi_\ve$, then 
\[
p_1\preceq p_2\Longleftrightarrow x(p_1)\leq x(p_2),
\]
where $x:\bR^2\ra \bR$   denotes the projection $(x,y)\mapsto x$.  In particular, this shows that  any $\ve$-sample of $S$ carries a natural total order.

If  $\Xi$  is an $\ve$-sample of $S$,  then the $PL$-interpolation  determined by sample $\Xi$ is the continuous, piecewise   linear function  $L= L_\Xi$ obtained as follows.

\begin{itemize}

\item Arrange the points in $\Xi$ in increasing order, with respect to the above total order,
\[
V=\{ \xi_0\prec \xi_1\prec\xi_2\prec\cdots \prec \xi_n\},\;\;n+1=|\Xi|.
\]
\item The graph of $L_\Xi$ is the $PL$-curve $\langle \xi_0,\xi_1, \dotsc, \xi_n\rangle$. 

\end{itemize}

In applications, the sample sets $\Xi$ will be chosen to satisfy   certain regularity.    

\begin{definition} \begin{enumerate}

\item If $\sigma $ is a positive  integer and $\Pi_\ep$ is an $\ep$-profile, then an \emph{$\ve$-sample with spread $\sigma$} is a subset  
\[
\Xi=\{\xi_0\prec\cdots \prec\xi_n\}\subset \Pi_\ve(S)
\]
such that  the following hold.

\begin{itemize}

\item The points $\xi_0$ and $\xi_n$ are the left and rightmost points in the profile.  (That is for each $p \in \Pi_\ep$, $x(\xi_0) \le x(p) \le x(\xi_n)$.)
\item For any $p\in \Pi_\ve$, there exists $\xi\in \Xi$ such that $|x(p)-x(\xi)|<\ve\sigma$.
\item 
\[
\frac{1}{2}\sigma  \leq \frac{1}{\ve}|x(\xi_k)-x(\xi_{k-1})|\leq \sigma,\;\;\forall k=1,\dotsc, n.
\]
\end{itemize}
\item A \emph{spread} is an increasing function $\ep\mapsto \sigma(\ep)$ from the positive real numbers to the positive integers with the following properties:
\begin{enumerate}
\item $\lim_{\ep \searrow 0} \sigma(\ep) = \infty$
\item $\lim_{\ep \searrow 0} \ep \sigma(\ep) = 0$
\end{enumerate}
\end{enumerate}\qed
\end{definition}

The next theorem tells how to approximate $f$ using profiles and spreads while maintaining first order properties.


\begin{theorem}
Let $f:[a,b]\ra \bR$ be a $C^2$-function and $\sigma(\ve)$ be a spread function. Fix $\ve>0$ and let $\Xi$ be an $\ve$-sample of the graph of $f$ with spread $\sigma(\ve)$,
\[
\Xi=\{ \xi_0\prec\xi_1\prec\cdots \prec\xi_n\},\;\;\xi_k=(x_k,y_k),\;\;k=0,1,\dotsc, n.
\]
Denote  by $f_\Xi: [x_0, x_n]\ra \bR$ the $PL$-interpolation  determined by $\Xi$.   Then 
\[
\bigl|\, f(x) - f_\Xi(x)\,\bigr| \le 7\bigl(\, \|f'\|_\infty+2\,\bigr)\ep + 3 \| f'\|_\infty\ep\sigma(\ep)+\|f''\|_\infty(\ep\sigma(\ep))^2,\;\;\forall x \in [a,b]\cap[x_0,x_n],
\]
 and 
 \[
 |f'(x) - f_\Xi'(x)| \le \frac{4(\| f'\|_\infty+2)}{\sigma(\ep)} + \| f''\|_\infty\ep\sigma(\ep),  
 \]
for all $x\in [a,b]\cap [x_0,x_n]\setminus \{x_0,\dotsc, x_n\}$. In particular, if
\[
\lim_{\ep \searrow 0} \ep \sigma(\ep)  = 0\;\;\mbox{and}\;\;\lim_{\ep \searrow 0} \sigma(\ep) = \infty,
\]
then as $\ep \searrow 0$, $f_\Xi$ converges to $f$ in the Sobolev norm $W^{1,p}$  for all $p \in [1,\infty)$. 
\label{thm: profileApprox}
\end{theorem}

\begin{proof} Let $x\in [x_0,x_n]$ and  $k\in \{0,1,\dotsc, (n-1)\}$ such that $x\in [x_k,x_{k+1}]$. On this subinterval, the function  $f_\Xi$ is defined by
\[
f_\Xi(x) := (x-x_k)\frac{y_{k+1}-y_k}{x_{k+1}-x_k} +y_k.
\]
Thus $f_\Xi$ is a function of the type described in Theorem \ref{thm: onesecant}, where 
\[
a = x_k, \;\;b = x_{k+1}, \;\;c = y_k, \;\; d = y_{k+1}.
\]
Noting that 
\[
\frac{1}{2}\ve\sigma(\ve)\leq b - a \leq  \ep \sigma(\ep), 
\]
 we obtain the desired error bounds.

Now suppose that $\lim_{\ep \searrow 0} \sigma(\ep) \ep = 0$ and $\lim_{\ep \searrow 0} \sigma(\ep) = \infty$.  Then note that every term in both the error bounds of $\bigl| f - f_\Xi\bigr|$ and $\bigl| f - f_\Xi'\bigr|$ is a constant multiplied by $\ep$, $\frac{1}{\sigma(\ep)}$, or  $\ep \sigma(\ep)$ all of which converge to $0$ as $\ep \searrow 0$.  This implies that 
\[
\lim_{\ep\ra 0}\int_a^b\left( |f(x) - f_\Xi(x)|^p+|f'(x) - f_\Xi'(x)|^p\right) dx= 0,\;\;\forall 1 \le p < \infty. 
\]
 Therefore $f_\Xi$ converges to $f$ in the Sobolev $W^{1,p}$ norm.

\end{proof}
This theorem allows us to approximate $C^2$ functions using only their pixelations.  Note that any profile can be chosen for the approximation.  We can see this in the pixelation by noting that the top profile and bottom profile both converge to the graph of the function, so any profile chosen in between will be squeezed onto the graph.  

To ensure the correct convergence on the derivative, we only need the two basic features of a spread: that as $\ep \searrow 0$, $\sigma(\ep) \to \infty$ and $\ep \sigma(\ep) \searrow 0$.  In general we would like the spread to increase quickly, so as to get an accurate approximation on the derivative.  The two conditions on the spread mean that the quickest increasing spread will look like $\lceil \ep^{\frac{-1}{1+r}} \rceil$ where $r > 0$.

\newpage
\section{Approximations of Elementary Regions and Curvature Approximation}
\label{s: 3}
\setcounter{equation}{0}
In this section we use the properties of pixelations of functions to study the pixelations of simple two dimensional sets.  We deal with the simplest case: that of a region between the graphs of two $C^2$ functions.

\begin{definition} A subset  $S\subset \bR^2$ is  said to be  \emph{elementary} (with respect to the $x$-axis) if its can be   defined as
\[
S=S(\beta,\tau):= \bigl\{\, (x,y): x \in [a,b], \beta(x) \le x \le \tau(x)\,\bigr\} ,
\]
where $\beta,\tau:[a,b]\ra \bR$ are $C^2$ functions  such that $\beta(x) \le \tau(x)$, $\forall x\in[a,b]$.    The function $\beta$ is called the \emph{bottom} of $S$ while $\tau$ is called the \emph{top} of $S$. Note that this includes the situation where $\beta \equiv \tau$, in which case $P_\ep(S) = P_\ep(\beta) = P_\ep(\tau)$.\qed
\label{def: s-type}
\end{definition}

In the remainder of this section $S$ will indicate an elementary set.  We first note that like the pixelation of a function, each column of $P_\ep(S)$ contains only one stack:

\begin{proposition}
If $S=S(\beta,\tau)$ is an elementary set,  then for every  $x \in [a,b]\setminus \ep\bZ$, the column $C_\ep(S,x)$ consists of exactly one stack.
\label{pro: ivt2}
\end{proposition}

\begin{proof}  Fix an $\ep$-generic $x\in [a,b]$. By Theorem \ref{thm: ivt} the columns  $C_\ve(\beta,x)$  and    $C_\ep(\tau,x)$    consist of single stacks. If these two columns intersect, then  the conclusion is obvious. If they do not intersect, then  any  pixel in the strip $S_\ve(x)$ situated below the stack $C_\ve(\tau,x)$ and above    the stack $C_\ve(\beta,x)$ is a pixel of $P_\ve(S)$.  This again proves that the  column $C_\ve(S,x)$ consists of a single stack.\end{proof}

\begin{definition}Fix $\ep>0$ and an elementary  set $S=S(\beta,\tau)$.

\begin{enumerate}

\item  An $\ep$-\emph{upper/lower profile} of $S$ is a profile of the $\ep$-pixelation  of the top/bottom function. An $\ep$-upper/lower \emph{sample} is a  sample of  an upper/lower profile.

\item An $\ep$-upper profile $\Pi^+_\ep$ is said to be \emph{compatible} with an $\ep$-lower profile $\Pi^-_\ep$ if $p^+\in \Pi^+_\ep$ and $p^-\in\Pi_\ep^-$ lie in the same column of $P_\ep(S)$ we have $p^+\in \Pi^+_\ep$ and $p^-\in\Pi_\ep^-$ lie in the same column of $P_\ep(S)$ we have
\[
y(p^-)\leq y(p^+).
\]

\item  An $\ep$-upper  sample $\Xi^+$  is said to be compatible   with an $\ep$-lower  sample $\Xi^-$, if   the following hold.

\begin{itemize}

\item  There exist compatible $\ep$-upper/lower profiles  $\Pi^\pm$ such that $\Xi^\pm\subset\Pi^\pm$.

\item  A strip $\eS_{\ep,k}$  contains a point in $\Xi^+$ if and only if it also contains a point in $\Xi^-$.

\end{itemize}

\item Suppose that $\Xi^\pm_\ep$ are compatible upper/lower samples of $S$
\[
\Xi^\pm_\ve=\{\xi_0^\pm\prec\xi_1^\pm\prec\cdots \prec\xi_n^\pm\}.
\]
The  $PL$-\emph{approximation}  of $S$  determined by these two samples is the $PL$-set   bounded by the simple closed $PL$-curve
\[
\langle \xi_0^-, \xi_1^-,\dotsc, \xi_n^-,\xi_n^+,\xi_{n-1}^+,\dotsc,\xi_0^+, \xi_0^-\rangle.
\]
We denote this approximation by $|S(\Xi_\ve^-,\Xi_\ve^+)|$. \qed
\end{enumerate}
\end{definition}

\begin{figure}[h]
\centering{\includegraphics[height=3in,width=6in]{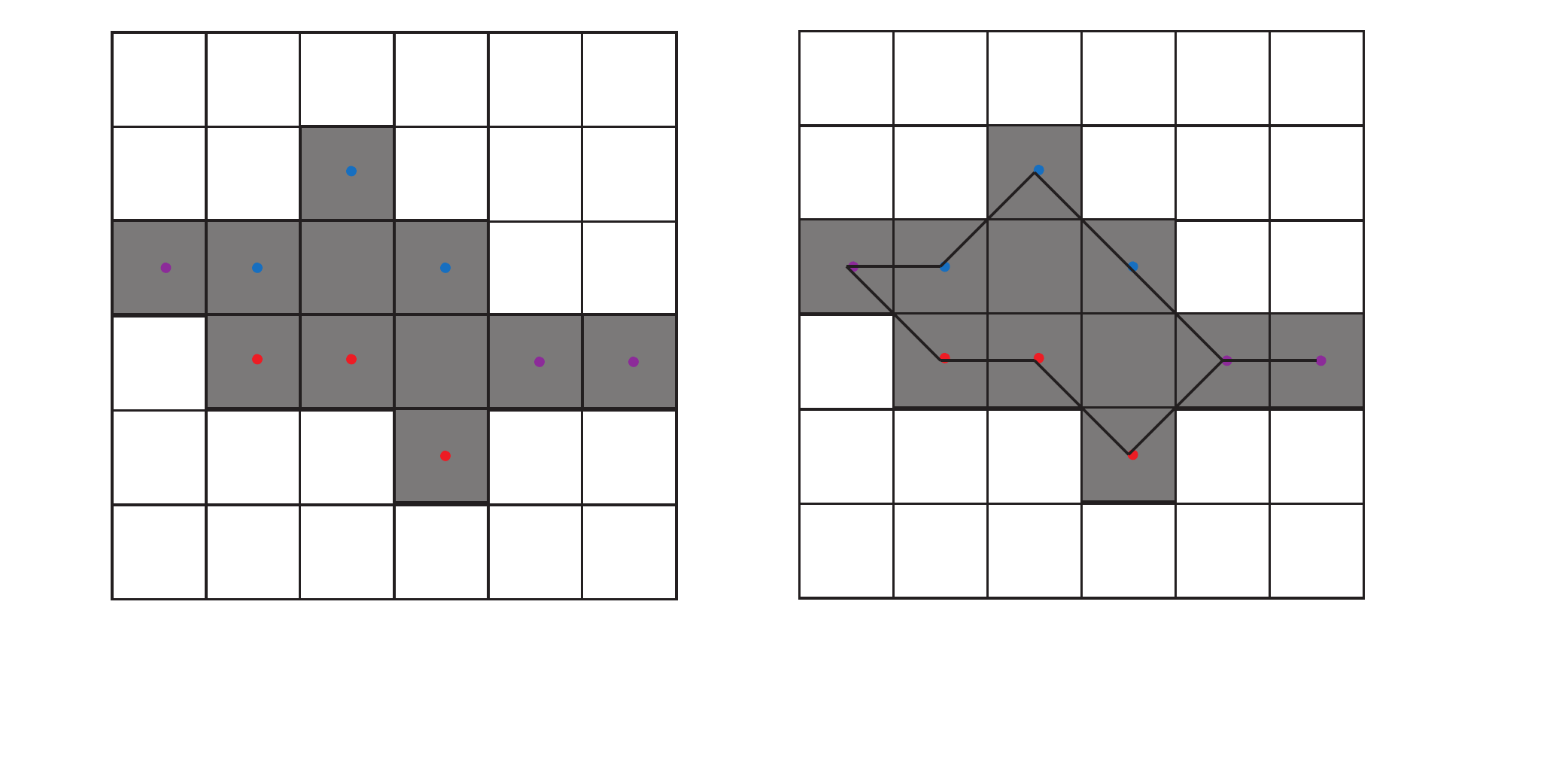}}
\caption{\sl An example of PL approximation.  The left figure shows a pixelation of an elementary subset, together with a compatible upper and lower profile.  The upper profile is indicated in blue while the lower profile is in red.  The right figure shows the PL approximation of compatible samples of these approximations with spread 2.  The PL approximation for the combined upper and lower samples would be the region in between the upper and lower approximations.}
\label{fig: topbottom*}
\end{figure}

These definitions are extensions of the ideas of profiles, samples and PL approximations found in section 2.  Note that the upper and lower profiles are defined from the pixelations of the top and bottom function, which might not be known.  However, from the definition of an elementary set, every pixel on the top of a column must intersect the pixelation of the upper function and similarly every pixel at the bottom of a column must intersect the lower function.  Therefore, an upper profile can be thought of as a profile chosen from the top pixels of $P_\ep(S)$ and a lower profile can be thought of as a profile chosen from the lower pixels of $P_\ep(S)$.

In section 2 we noted that the PL approximation preserves first order properties for a good choice of spread.  For elementary sets, we would like to prove something stronger.  Namely, we would like to show that the total curvature of the set is preserved for good choice of spread.

The total curvature of a $C^2$ immersion $\gamma:[a,b]\ra \bR^2$ is defined as follows.    Define a $C^1$-function $\theta:[a,b]\ra \bR$ such that 
\[
\frac{1}{|\dot{\gamma}(t)|}\dot{\gamma}(t)= (\cos(\theta(t),\sin(\theta(t))),
\]
where a dot $\dot{}$   denotes the $t$-derivative.    We set 
\[
\kappa(t):=|\dot{\theta}(t)|.
\]
The scalar $\kappa(t)$ is called the \emph{curvature} of $\gamma$ at the point $\gamma(t)$.  We define the total curvature  of $\gamma$ to be
\[
K(\gamma):=\int_a^b \kappa(t)\,dt.
\]
Suppose now that  $\gamma:[a,b]\ra \bR^2$  continuous  and  piecewise $C^2$-immersion, i.e., there exist a finite subset $\{t_0,\dotsc, t_\nu\}\subset [a,b]$, $t_{i-1}<t_i$, for
\[
a=t_0<t_1<\cdots <t_{\nu-1}<t_\nu=b
\]
such that the restriction $\gamma_i:=\gamma|_{[[t_{i-1},t_i]}$ is a $C^2$-immersion for any $i=1,\dotsc,\nu$.  The curvature of  $\gamma$ at a jump point $\gamma(t_i)$ is the quantity
\[
\kappa(t_i)= |\theta(t_i^+)-\theta(t_i^-)|:=|\lim_{t\searrow t_i} \theta(t)-\lim_{t\nearrow t_i} \theta(t)|.
\]
We  define the total curvature of  $\gamma$ to be
\[
\begin{split}
K(\gamma)&:=\sum_{i=1}^\nu K(\gamma_i)+\sum_{i=1}^{\nu-1}\kappa(t_i)\\
&+\begin{cases}0, & \gamma(b)\neq\gamma(a)\\
|\theta(t_0^+)-\theta(t_\nu^-)|, &\gamma(b)=\gamma(a).
\end{cases}
\end{split}
\]
For more details we refer to \cite{Mil} and \cite[\S 2.2]{Mor}.

\begin{proposition} Let $h: [a,b]\ra \bR$ be a $C^2$ function. Fix a spread $\si$  satisfying the properties 
\begin{subequations}
\begin{equation}
\lim_{\ep \searrow 0} \ep \sigma(\ep)^2 = \infty,
\label{eq: spi}
\end{equation}
\begin{equation}
\lim_{\ep \searrow 0} \ep \sigma(\ep) = 0.
\label{eq: sp0}
\end{equation}   
\end{subequations}
For every $\ep > 0$ let $\Xi_\ep$ be a $\ep$-sample spread $\sigma$. We denote by $h_\ep$ be the PL interpolation of $\Xi_\ep$. Then
\[
\lim_{\ep \searrow 0}K(\Gamma_{h_\ep})=K(\Gamma_h),
\]
where for any function $f$ we denote by $\Gamma_f$ its graph.
\label{prop: c2curv}
\end{proposition}

\begin{proof}  Let $X_\ep$ be the ordered set of $x$-values of the sample $\Xi_\ep$,
\[
X_\ep = \{x_0< x_1< \dots <x_n\}\subset [a,b],
\]
and set
\[
p_i:= (x_i, h(x_i))\in \Gamma_h.
\]
Define $s_\ep$ to be the PL function with graph $\langle p_0,p_1,\dotsc, p_n\rangle$.   We will prove two things.
\begin{subequations}
\begin{equation}
\lim_{\ep\ra 0} \left(\, K(h_\ep)-K(s_\ep)\,\right)=0,
\label{eq: h-s}
\end{equation}
\begin{equation}
\lim_{\ep\ra 0} K(s_\ep)=K(h).
\label{eq: sh}
\end{equation}
\end{subequations}

\medskip

\noindent  {\bf Proof of (\ref{eq:  h-s})}. For $1 \le i \le n$ let $m_i^h$ indicate the slope of the $i$-th line segment of $h_\ep$ and $m_i^s$ indicate the slope of the $i$-th line segment of $s_\ep$, i.e.,  the slope over  the interval $[x_{i-1},x_i]$.

Since $h_\ep$ has spread $\sigma$, the width of the interval $[x_{i-1},x_i]$ is at most $\ep \sigma$.  Then since $m_i^h$ is the slope of a line connecting two points in the sample $H_\ep$, and $m_i^s$ is the slope of the secant line over the same interval, Lemma \ref{lem: secant slope} implies that
\begin{equation}
|m_i^h - m_i^s|  \le 2\frac{||h'||_\infty+2}{\sigma(\ep)}
\label{eq: slopeBound}
\end{equation}
In particular, $|m_i^h - m_i^s| \to 0$ as $\ep \to 0$ since $\sigma \to \infty$.

Define $\theta_i^h\in (-\frac{\pi}{2},\frac{\pi}{2})$ as the angle that the $(i+1)$-st line component of $s_\ep$ makes with the $x$-axis and likewise $\theta_i^s\in (-\frac{\pi}{2},\frac{\pi}{2})$ as the angle that the $(i+1)$-st line component of $s$ makes with the $x$-axis.  More precisely,  
\[
\tan(\theta_i^h) = m_i^h,\;\; \tan(\theta_i^s) = m_i^s. 
\]
 Therefore, the difference formula for tangent implies:
\[
\frac{m_i^h - m_i^s}{1 + m_i^h m_i^s} = \tan(\theta_i^h - \theta_i^s)
\]
so that
\begin{equation}
\theta_i^h - \theta_i^s = \arctan \left( \frac{m_i^h - m_i^s}{1 + m_i^h m_i^s} \right).
\label{eq: angle bound}
\end{equation}
Now, the estimate  (\ref{eq: slopeBound})  implies  that for small $\ep$, we have
\[
\left|\frac{m_i^h - m_i^s}{1 + m_i^h m_i^s}\right| \le 4\frac{\left(\|h'\|_\infty+2\right)}{\sigma(\ep)}.
\]
Indeed, the denominator is at least $\frac{1}{2}$ because, for $\ep$ sufficiently small, we have
\[
1+m_i^hm_i^s=1+m_i^h\left (\, m_i^h+O\left(\sigma(\ep)^{-1}\right)\,\right)
\]
\[
= 1+(m_i^h)^2+m_i^h O\left(\sigma(\ep)^{-1}\right)=1+(m_i^h)^2+O\left(\|h'\|_\infty\sigma(\ep)^{-1}\right).
\]
Then, using the Taylor expansion of $\arctan$ around $0$, we have
\begin{equation}
\delta_i:=\theta_i^h - \theta_i^s = O\left( \frac{1}{\sigma(\ep)}\right).
\label{eq: O angle bound}
\end{equation}
Recall that 
\[
K(s_\ep)= \sum_{i=1}^n \left|\,\theta_i^s - \theta_{i-1}^s\,\right|,\;\;K(h_\ep)= \sum_{i=1}^n \left|\,\theta_i^h - \theta_{i-1}^h\,\right|.
\]
and we deduce  that
\[
 \left| K(s_\ep)-K(h_\ep)\right|\leq \sum_{i=1}^n (|\delta_i| + |\delta_{i-1}|).
\]
Note that since $\Xi_\ep$ has spread $\sigma$,
\[
n = O\left( \frac{1}{\ep \sigma} \right)
\]
and (\ref{eq: O angle bound}) implies 
\[
\sum_{i=1}^n (|\delta_i| + |\delta_{i-1}|)=   O\left( \frac{1}{\ep \sigma(\ep)^2} \right).
\]
Since we have assumed that $\lim_{\ep \searrow} \ep \sigma(\ep)^2 = \infty$, the last estimate implies (\ref{eq: h-s}).

\medskip

\noindent \textbf{Proof of (\ref{eq: sh}).} Define the function $\theta$ by the equation:
\[
\frac{1}{|1 + (h'(x))^2|}\left(1,h'(x)\right) = (\cos(\theta(x)),\sin(\theta(x)))
\]
The function $\theta(x)$ gives the angle between the $x$-axis and the tangent line of $h$ at the point $x$.
As before let $m_i^s$ indicate the slope of the $i$-th line segment of $s_\ep$ and $\theta_i^s$ be the angle that the $i$-th line segment of $s_\ep$ makes with the $x$-axis.
The Mean Value Theorem implies that for each $i$ there exists a $\xi_i$ on $[x_{i-1},x_i]$ such that
\[
m_i^s = h'(\xi_i)
\]
It then follows that
\[
\theta_i^s = \theta(\xi_i)
\]
since $\tan(m_i^s) = \theta_i^s$ and $\tan(h'(\xi_i)) = \theta(\xi_i)$.  Furthermore, for each $i > 1$, the Mean Value Theorem guarantees the existence of a $\zeta_i$ on $[\xi_{i-1}, \xi_i]$ such that
\[
\theta'(\zeta_i) = \frac{\theta(\xi_i) - \theta(\xi_{i-1})}{\xi_i - \xi_{i-1}}
\]

Recall that the total curvature of $s_\ep$ is calculated by the summation:
\[
K(s_\ep) = \sum_{i=2}^n |\theta_i^s - \theta_{i-1}^s|
\]
Using the definitions of $\xi_i$ and $\zeta_i$ we find that
\begin{multline}
K(s_\ep) = \sum_{i=2}^n |\theta(\xi_i) - \theta(\xi_{i-1})| = \sum_{i=2}^n \left|\frac{\theta(\xi_i) - \theta(\xi_{i-1})}{\xi_i - \xi_{i-1}} \right| (\xi_i - \xi_{i-1}) = \sum_{i = 2}^n |\theta'(\zeta_i)| (\xi_i - \xi_{i-1})
\label{eq: replacement}
\end{multline}
Note that last sum in (\ref{eq: replacement}) is almost a Riemann sum for $|\theta'|$.  It is not a Riemann sum because the values $\{\xi_1, \xi_2, \dots, \xi_n\}$ do not form a partition of $[a,b]$ (since $\xi_1 \geq a$ and $\xi_n \le b$.)  However, it is true that:
\begin{equation}
\lim_{\ep \searrow 0}\left(|\theta'(a)| (\xi_1 - a) +  \sum_{i = 2}^n |\theta'(\zeta_i)| (\xi_i - \xi_{i-1}) + |\theta'(b)| (b - \xi_n)\right) = \int_a^b |\theta'(x)| dx
\label{eq: riemann}
\end{equation}
Both $(b - \xi_n)$ and $(\xi_1 - a)$ are intervals of size at most $\ep \sigma(\ep)$ (since our sample was chosen with spread $\sigma$.)  Therefore both of these terms vanish as $\ep \to 0$ (by our constraint on $\sigma$.)  Therefore equation (\ref{eq: riemann}) simplifies to
\begin{equation}
\lim_{\ep \searrow 0}\left(\sum_{i = 2}^n |\theta'(\zeta_i)| (\xi_i - \xi_{i-1})\right)  = \int_a^b |\theta'(x)| dx
\label{eq: riemann2}
\end{equation}
Combining (\ref{eq: riemann2}) with (\ref{eq: replacement}) shows that
\[
\lim_{\ep \searrow 0} K(s_\ep) = \int_a^b |\theta'(x)| dx = K(h)
\]
which proves (\ref{eq: sh}).


Given both (\ref{eq:  h-s}) and (\ref{eq: sh}) we simply note that as $\ep \to 0$ $K(h_\ep)$ approaches $K(s_\ep)$ which in turn approaches $K(h)$. \end{proof}

In applications, it will be necessary to deal with sets which resemble elementary sets, but whose boundaries are possibly only piecewise $C^2$.  Let us define precisely  the notion of piecewise $C^2$ function.

\begin{definition} A function  $f: [a,b]\ra \bR$ is said to be piecewise $C^2$ if  there exists a finite set
\[
S=\bigl\{a=s_0<s_1<\cdots < s_\ell=b\,\bigr\},
\]
such that the following hold.

\begin{enumerate}
\item The function $f$ is continuous.

\item For any $j=1,\dotsc, \ell$, and any $k=1,2$  the restriction of $f$ to the open interval $(s_{j-1},s_j)$ is a $C^2$ function and the limits
\[
\lim_{x\searrow s_{j-1} }f^{(k)}(x),\;\;\lim_{x\nearrow s_j} f^{(k)}(x)
\]
exist and are finite.
\end{enumerate}
We say that $S$ is the \emph{singular set} of $f$ and that the points  $C_j=(s_j, f(s_j))$, $j=0,\dotsc, \ell$ are the \emph{corners} of   the graph of $f$. The integer $\ell$ is called the \emph{length} of the singularity set. \qed
\end{definition}

 \begin{proposition} Suppose  that $f: [a,b]\ra \bR$ is a piecewise $C^2$ function with singular set
 \[
 S:=\bigl\{ a=s_0<s_1<\cdots <s_\ell=b\,\bigr\}.
 \]
 Fix a spread satisfying the properties (\ref{eq: spi}) and (\ref{eq: sp0}). For every $\ep > 0$ let $\Xi_\ep$ be an $\ep$-sample  with spread $\sigma$ satisfying (\ref{eq: spi}) and (\ref{eq: sp0}). For any $\ep>0$ we denote  by $f_\ep$ be the PL interpolation of $\Xi_\ep$. Then
\[
\lim_{\ep \searrow 0}K(\Gamma_{f_\ep})=K(\Gamma_f).
\]
\label{prop: piece-curv}
 \end{proposition}
 
 \begin{proof} We will follow the same strategy as in the proof of Proposition \ref{prop: c2curv}, with some   expected modifications due to the presence of singularities. For the clarity of exposition  we will assume  that the length $\ell=2$ so that
 \[
 S=\bigl\{\, a,b,s_1;\;\;s_1\in (a,b)\,\bigr\}.
 \]
 Set
 \[
 f^-:=f|_{[a,s_1]},\;\;f^+:=f|_{[s_1,b]}.
 \]
Let
\[
\Xi_\ep= \bigl\{ (x_j,y_j);\;\; j=0,\dotsc, n=n(\ep)\,\bigr\}. 
\]
We denote  by  $X_\ep$  the ordered set of $x$-values of the sample $\Xi_\ep$,
\[
X_\ep = \{x_0< x_1< \dots <x_n;\;\;n=n(\ep)\}\subset [a,b],
\]
and set
\[
p_i:= (x_i, f(x_i))\in \Gamma_h.
\]
Define $g_\ep$ to be the PL function with graph $\langle p_0,p_1,\dotsc, p_n\rangle$. 
We will prove two things.
\begin{subequations}
\begin{equation}
\lim_{\ep\ra 0} \left(\, K(f_\ep)-K(g_\ep)\,\right)=0,
\label{eq: fg}
\end{equation}
\begin{equation}
\lim_{\ep\ra 0} K(g_\ep)=K(f).
\label{eq: gf}
\end{equation}
\end{subequations}

\medskip

\noindent  {\bf Proof of (\ref{eq:  fg})}. For $1 \le i \le n$ let $m_i^f$ denote  the slope of the $i$-th line segment of $f_\ep$ and $m_i^g$ denote  the slope of the $i$-th line segment of $g_\ep$, i.e.,  the slopes of the segments over  the interval $[x_{i-1},x_i]$. Since $f_\ep$ has spread $\sigma$, the width of the interval $[x_{i-1},x_i]$ is at most $\ep \sigma$.  

Since the function $f$ is Lipschitz continuous we deduce  from Theorem \ref{thm: difbound} that
\[
|f(x_i)- y_j|\leq \bigl(\, \|f'\|_\infty+2\,\bigr)\ep.
\]
This  implies that
\[
|m_i^f - m_i^g|  \le 2\frac{||h'||_\infty+2}{\sigma(\ep)}
\]
In particular, $|m_i^f - m_i^g| \to 0$ as $\ep \to 0$ since $\sigma(\ep) \to \infty$ as $\ep\ra 0$. We can now conclude as in the proof of (\ref{eq: h-s}).

\medskip

\noindent  {\bf Proof of (\ref{eq:  gf})}. For every $\ep>0$ there exists a unique $i_\ep\in \{1,\dotsc, n(\ep)\}$ such that
\[
x_{i_\ep-1}<s_1\leq x_{i_\ep}.
\] 
We set $y_*= f(s_1)$, $p_*:=(s_1, y_*)$, and we   denote by $h_\ep$ the  $PL$-function with graph
\[
\langle p_0, \dotsc, p_{i_\ep-1}, p_*, p_{i_\ep}, \dotsc p_{n(\ep)}\rangle.
\]
 We set
 \[
 h^-_\ep:=h_\ep|_{[a,s_1]},\;\;h^+_\ep:=h_\ep|_{[s_1,b]},
 \]
 We denote by $\theta_\ep\in [0,\pi)$ the angle between the two nontrivial  line segments of the graph of $h$ that have $p_*$ as  common vertex. Then
 \[
 K(\Gamma_{h_\ep})=K(\Gamma_{h^-_\ep})+K(\Gamma_{h^+_\ep})+\theta_\ep.
 \]
 From Proposition \ref{prop: c2curv} we deduce that
 \[
 \lim_{\ep\ra 0} K(\Gamma_{h^\pm_\ep})=K(\Gamma_{f^\pm}),
 \]
 while  $\theta_\ep$ converges as $\ep\ra 0$ to the angle  between the  left and right tangents  to the graph of $f$ at $p_*$, so that
 \[
 \lim_{\ep\ra 0} K(\Gamma_{h_\ep})=K(\Gamma_f).
 \]
We thus have to prove that
\begin{equation}
\lim_{\ep\ra 0}\bigl(\, K(\Gamma_{g_\ep})-K(\Gamma_{h_\ep})\,\bigr)=0.
\label{eq: curv-diff}
\end{equation}
To analyze the difference  $K(\Gamma_{g_\ep})-K(\Gamma_{h_\ep})$ we distinguish two cases.

\medskip

\noindent {\bf A.} $s_1=x_{i_{\ep}}$. In this case  ${g_\ep}=h_\ep$ and thus $K(\Gamma_{g_\ep})-K(\Gamma_{h_\ep})=0$.

\medskip

\noindent {\bf B.} $s_1<x_{i}$, $i=i_\ep$.  We introduce the following notation from which we suppress the $\ep$-dependence.
\begin{itemize}

\item  $\alpha_-\in [0,\pi)$  denotes the angle  between the vectors $\overrightarrow{p_{i-2}p_{i-1}}$ $\overrightarrow{p_{i-1}p_{i}}$.

\item  $\alpha_+\in[0,\pi)$  denotes the angle  between the vectors  $\overrightarrow{p_{i-1}p_{i}}$  and $\overrightarrow{p_{i}p_{i+1}}$.

\item  $\beta_-\in [0,\pi)$   denotes the angle  between the vectors $\overrightarrow{p_{i-2}p_{i-1}}$ and  $\overrightarrow{p_{i-1}p_{*}}$.

\item  $\beta_+\in[0,\pi)$   denotes the angle  between the vectors  $\overrightarrow{p_{*}p_{i}}$  and $\overrightarrow{p_{i}p_{i+1}}$.

\item $\theta_\ep\in[0,\pi)$ denotes the angle between the vectors $\overrightarrow{p_{i-1}p_{*}}$ and  $\overrightarrow{p_{*}p_{i}}$.

\item $m_*^-$   denotes  slope  of the  line  $p_{i-1}p_*$.

\item $m_*^+$  denotes the slope of the line $p_*p_i$.

\item $m_*=m_i^g$  denotes the slope of the line $p_{i-1}p_i$. 

\end{itemize}
Then
\[
K(\Gamma_{g_\ep})-K(\Gamma_{h_\ep})=(\alpha_-+\alpha_+)-(\beta_-+\beta_+ + \theta_\ep),
\]
\[
\alpha_\pm = \left| \arctan m^g_{i\pm 1}-\arctan m_*\right|, \beta_\pm=\left|\arctan m_{i\pm1}^h-\arctan m_*^\pm \right|,
\]
Now let us observe that
\[
m^g_{i\pm 1}=m_*^\pm +o(1),
\]
where, following Landau's convention, we denote by  $o(1)$  a quantity that goes to zero as $\ep\ra 0$.  Thus
\[
\beta_\pm =o(1),
\]
 We deduce that
\begin{equation}
\begin{split}
K(\Gamma_{g_\ep})-K(\Gamma_{h_\ep})&=\alpha_++\alpha_--\theta_\ep+o(1)\\
&= \left| \arctan m_*^--\arctan m_*\right|+\left| \arctan m_*^+-\arctan m_*\right|-\theta_\ep+ o(1).
\end{split}
\label{eq: diff-curv1}
\end{equation}
We now  remark that\footnote{Compare with Lemma \cite[Lemma 1.1]{Mil}. A word of warning: while the main conclusion of that Lemma is true (total curvature   never decreases upon adjoining a point to to a $PL$-curve,   the claim  that it does not change for planar curves is not true.}
\begin{equation}
 \left| \arctan m_*^--\arctan m_*\right|+\left| \arctan m_*^+-\arctan m_*\right|=\theta_\ep.
 \label{eq: euclid}
 \end{equation}
 The equality (\ref{eq: euclid}) is a classical fact of Euclidean geometry. More precisely, in the  triangle $p_{i-1}p_*p_i$, the sum of the interior angles at the vertices  $p_{i-1}$ and $p_{i+1}$ is equal to the sum of exterior angle at $p_*$; see Figure \ref{fig: euclid}. Using (\ref{eq: euclid}) in (\ref{eq: diff-curv1}) we obtain (\ref{eq: curv-diff}). This completes the proof of (\ref{eq:  gf}), and thus the proof of the Proposition \ref{prop: piece-curv}
\end{proof}
\begin{figure}[h]
\centering{\includegraphics[height=3in,width=3in]{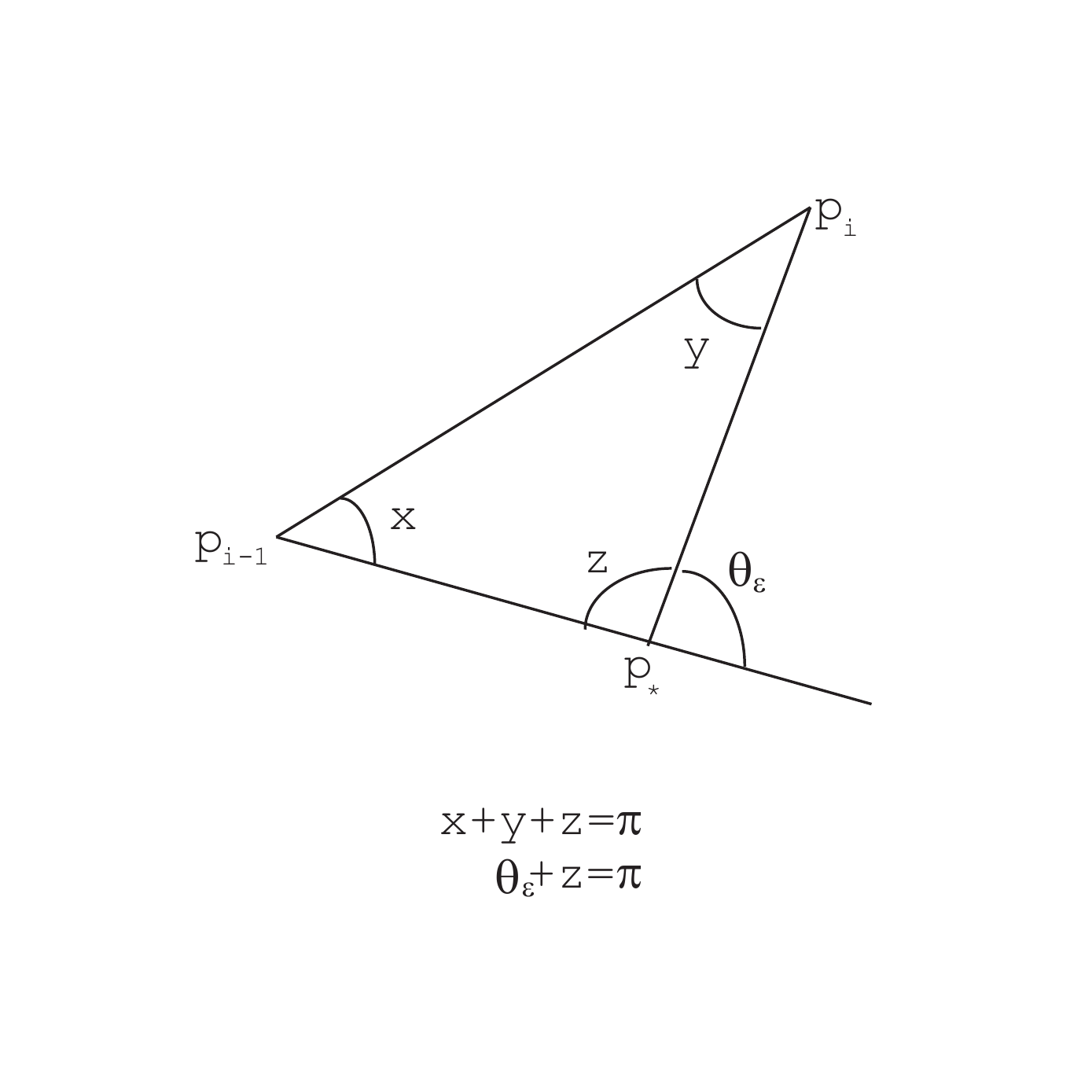}}
\caption{A classical fact of Euclidean geometry}
\label{fig: euclid}
\end{figure}

This will allow us to handle the case of piecewise linear sets.  Therefore the following corollary will be useful.

\begin{corollary}
Let $S(\beta,\tau)$ be a set which can be expressed as
\[
S = \bigl\{\, (x,y) : x \in [a,b], \beta(x) \le y \le \tau(x)\,\bigr\}
\]
where $\beta,\tau:[a,b]\ra \bR$ are piecewise $C^2$ functions such that
\[
\beta(x) <\tau (x)_,\;\;\forall x\in (a,b).
\]
Fix a spread $\si(\ep)$ satisfying  the conditions (\ref{eq: spi}) and (\ref{eq: sp0}). For each $\ep$ we choose  compatible $\ep$-upper/lower profiles  $\Xi^\pm_\ep$ of $S$ with common spread $\si(\ep)$. We denote by ${S}_\ep$ the $PL$ approximation of ${S}$ defined by these samples. Then
\[
\lim_{\ep \searrow 0} K(\partial S_\ep) = K(\partial S).
\]
\label{cor: piece curv}
\end{corollary}

\begin{proof}  Let
\[
\Xi^\pm_\ep=\bigl\{ \xi_0^\pm \prec\xi_1^\pm \prec\cdots \prec\xi_n^\pm\,\bigr\},\;\;\xi_k^\pm= (x_k^\pm, y_k^\pm).
\]
The compatibility condition implies that
\[
x_k^-=x_k^+=:x_k,\;\;y_k^-\leq y_k^+,\;\;\forall k=0,1,\dotsc, n.
\]
Let $\beta_\ep$ be the PL function whose graph is the bottom part of the boundary of $S_\ep$, and $\tau_\ep$ to be the PL function whose graph is the top part of the boundary of $S_\ep$.  Let $m_i^\beta(\ep)$ indicate the slope of the $i$-th line segment of the graph of  $\beta_\ep$ and similarly let $m_i^\tau(\ep)$ indicate the slope of the $i$-th line segment of the graph of $\tau_\ep$. We have
\[
\begin{split}
K(\partial S_\ep) & = |\pi - \arctan(m_1^\beta(\ep)\,)| \\
&+ \sum_{i=2}^n |\arctan(m_i^\beta(\ep)\,) - \arctan(m_{i-1}^\beta(\ep)\,)| \\
&+ |\arctan(m_n^\beta(\ep)) - \pi| + |\pi - \arctan(m_1^\tau)| \\
&+ \sum_{i=2}^n |\arctan(m_i^\tau(\ep)\,) - \arctan(m_{i-1}^\tau(\ep)\,)| \\
&+ |\arctan(m_n^\tau(\ep)\,) - \pi|
\end{split}
\]
which can be rewritten as
\begin{equation}
\begin{split}
K(\partial S_\ep) & = |\pi - \arctan(m_1^\beta(\ep)\,)| + |\arctan(m_n^\beta(\ep)\,) - \pi| \\
&+ |\pi - \arctan(m_1^\tau(\ep)\,)| + |\arctan(m_n^\tau(\ep)\,) - \pi| \\
&+ K(\beta_\ep) + K(\tau_\ep)
\end{split}
\label{eq: epcurv}
\end{equation}
Proposition \ref{prop: piece-curv} implies
\begin{equation}
\lim_{\ep \searrow 0} K(\beta_\ep) = K(\beta)\;\;\mbox{and}\;\; \lim_{\ep \searrow 0} K(\tau_\ep) = K(\tau).
\label{eq: piece conv}
\end{equation}
Now note that each line segment is defined by connecting two points in the pixelation of $\beta$ or $\tau$ over an interval of at most $\ep \sigma(\ep)$.  Since as $\ep \to 0$, $\ep \sigma(\ep) \to 0$, Theorem \ref{thm: onesecant} implies that
\begin{equation}
\begin{split}
\lim_{\ep \searrow 0} m_1^\beta(\ep) = \beta'(a), \;\;\lim_{\ep \searrow 0} m_n^\beta(\ep) = \beta'(b),\\
\lim_{\ep \searrow 0} m_1^\tau(\ep) = \tau'(a), \;\;\lim_{\ep \searrow 0} m_n^\tau(\ep) = \beta'(b).
\end{split} 
\label{eq: slope convergence}
\end{equation}
Combining (\ref{eq: epcurv}), (\ref{eq: piece conv}), (\ref{eq: slope convergence}) we find that
\begin{equation}
\begin{split}
\lim_{\ep \searrow 0} K(\partial S_\ep) &= |\pi - \arctan(\beta'(a))| + |\arctan(\beta'(b)) - \pi| \\
& + |\pi - \arctan(\tau'(a))| + |\arctan(\tau'(b)) - \pi| \\
&+ K(\beta) + K(\tau).
\end{split}
\label{eq: final curv}
\end{equation}
Note that $|\pi - \arctan(\beta'(a))|$ is the value of the angle between the vertical line $x = a$ and the tangent line of $\beta$ at $a$.  Similarly each other difference on the right hand side of (\ref{eq: final curv}) corresponds to an angle at one of the corners of $\partial S$.  Therefore the right hand side of the (\ref{eq: final curv}) is equal to the $K(\partial S)$, so the corollary holds.
\end{proof}

\newpage
\section{Approximations of PL Sets and Morse Theory for Pixelations}
\label{s: 4}
\setcounter{equation}{0}
In previous sections we have dealt only with  the simple regions and we investigated mainly  \emph{geometric} properties of these regions and their pixelation.  In this section we turn our attention to the relationship between the topology of a $PL$-set and those of its pixelations.   

Surprisingly, this is a nontrivial matter. In general, the homotopy type of a region may not be not preserved when taking the pixelation.  Worse, the homotopy type may never be recovered in any pixelation, even for small resolutions $\ep$.  This type of bad behavior can happen even for a simple PL case.  Consider the set $S$ composed of the rays starting from the origin and proceeding in the positive direction with slopes $\frac{1}{2}$ and $\frac{1}{7}$ (see Figure \ref{fig: half}).  A careful examination of $P_\ep(S)$ reveals the existence of cycle.  Worse yet, the pixelations for smaller values of $\ve$ are simply contractions of the larger pixelations.  This means that $P_\ep(S)$ contains a cycle for all small $\ep$.  For a  taste of how much worse can this get we refer to Appendix \ref{s: a}.

\begin{figure}[ht]
\centering{\includegraphics[height=2in,width=2in]{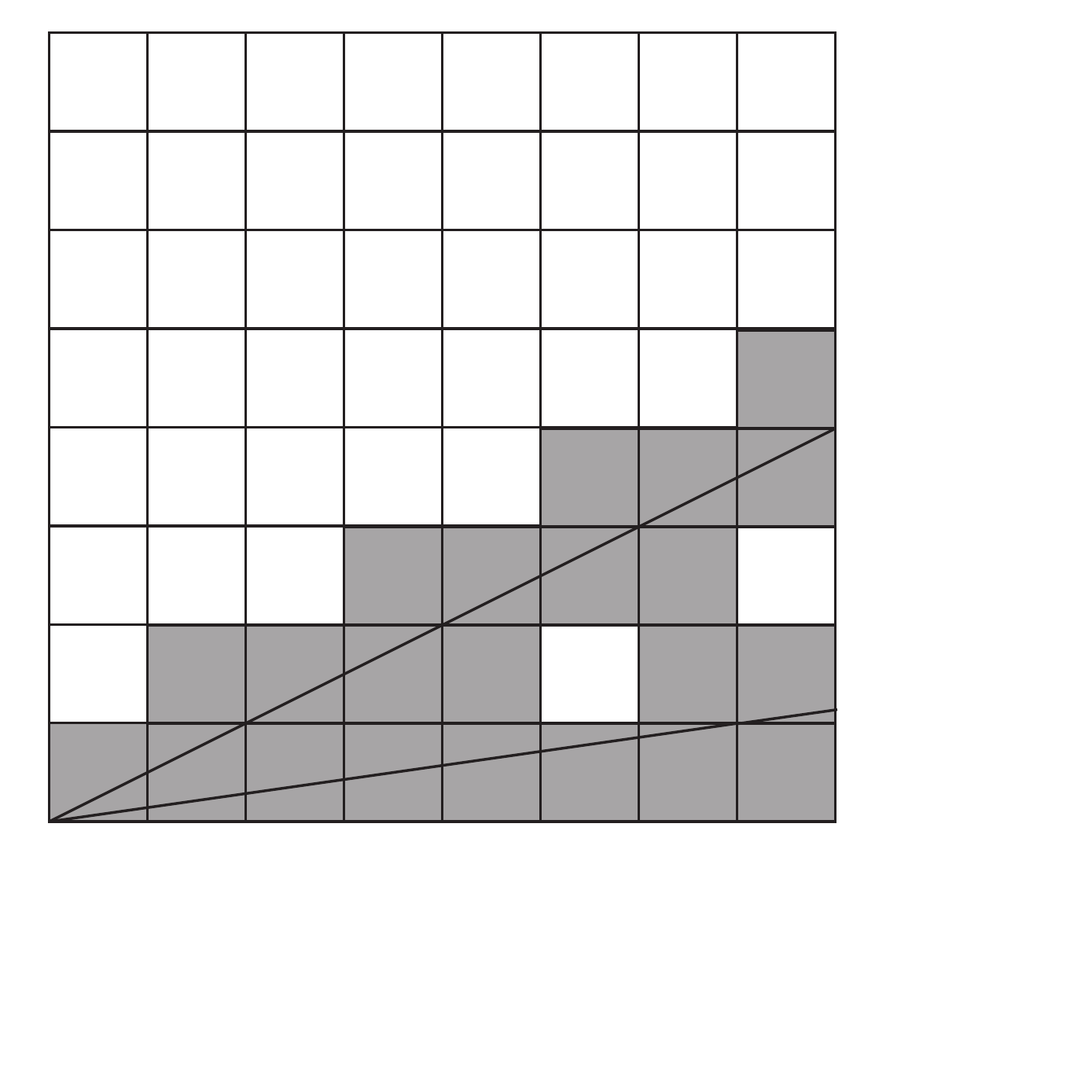}}
\caption{\sl The pixelation of the lines with slope $\frac{1}{2}$ and $\frac{1}{7}$. After this point the pixelations of the upper and lower lines diverge permanently and no more holes are formed.}
\label{fig: half}
\end{figure}

First, some good news. The next two results imply that if two compact planar sets  are disjoint, then their pixelations will also be disjoint for sufficiently fine resolutions. 


\begin{theorem}
Let $S \subset \bR^2$ be a compact set and let $x \in \mathbb{R}^2$ such that $x \notin S$.  Set
\[
m := \inf_{s\in S} d(x,s),
 \]
where $d(\bullet, \bullet)$ is the Euclidean metric on $\bR^2$, so that $m$ is the Hausdorff distance from $x$ to $S$.
Then $\forall \ep < \frac{m}{3}$, $x \notin P_\ep(S)$.
\label{thm: inc}
\end{theorem}

\begin{proof}For $x$ to lie in $P_\ep (S)$, both $x$ and a point of $S$ must lie in the same pixel.  However the furthest two points can be apart in a square of side length $\ep$ is $\ep \sqrt{2}$ (if they lie on opposite corners of the square.)  But since $\ep < \frac{m}{3}$, $\ep \sqrt{2} < \frac{m \sqrt{2}}{3} < m$, so it is impossible for a point of $S$ to lie in the same pixel as $x$.  Therefore $x \notin P_\ep (S)$. 
\end{proof}

\begin{corollary}
Let $S \subset \bR^2$ be a compact set with finitely many connected components.  Then for sufficiently small $\ep$, the number of connected components of $P_\ep(S)$ agrees with the number of connected components of $S$.
\label{cor: seperate}
\end{corollary}

\begin{proof}
First note that $P_\ep(S)$ cannot have more connected components than $S$, since $P_\ep(S)$ contains only pixels that intersect $S$.

Since $S$ is compact with finitely many components, there is a minimum positive distance $\delta$ between components.  For each component $K$ of $S$ define
\[
T_K := \{x \in \bR^2 : \inf_{k \in K} d(x,k) \le \frac{\delta}{2} \}
\]
and note that its topological boundary is
\[
\pato T_K = \{x \in \bR^2 : \inf_{k \in K} d(x,k) = \frac{\delta}{2} \}
\]
Then each point in $\partial T_K$ is $\frac{\delta}{2}$ away from $K$ in the Hausdorff metric.  Furthermore, since $\delta$ is the minimum distance between components of $S$, each point in $\partial T_K$ is at least $\frac{\delta}{2}$ away from any connected component.  Therefore, for $\ep < \frac{\delta}{6}$ we have for each $K$
\[
\partial T_K \cap P_\ep(S) = \emptyset
\]
This implies that for sufficiently small $\ep$ each connected component of $P_\ep(S)$ intersects at most one connected component of $S$.  However, since
\[
S \subset P_\ep(S)
\]
it is also true that each connected component of $S$ is contained within one connected component of $P_\ep(S)$.  Therefore for sufficiently small $\ep$, $P_\ep(S)$ and $S$ have the same number of connected components.
\end{proof}

The corollary guarantees that   the zeroth Betti number of a set coincides with those of its sufficiently fine pixelations.  Theorem \ref{thm: inc} also suggests that, for small $\ep$, the only way that the topological type of $P_\ep(S)$ can disagree with that of $S$ is if $P_\ep(S)$ has additional cycles (since every ``real'' cycle will show up in the pixelation).   We will refer  to these cycles as \emph{holes}.  More formally  the holes are cycles in $P_\ep(S)$ that are not contained in the image of the inclusion induced morphismlations.  Theorem \ref{thm: inc} also suggests that, for small $\ep$, the only way that the topological type of $P_\ep(S)$ can disagree with that of $S$ is if $P_\ep(S)$ has additional cycles (since every ``real'' cycle will show up in the pixelation).   We will refer  to these cycles as \emph{holes}.  More formally  the holes are cycles in $P_\ep(S)$ that are not contained in the image of the inclusion induced morphism
\[
H_1(S,\bZ)\ra H_1\bigl(\, P_\ep(S),\bZ\,\bigr).
\]
Thus recovery of $S$ from $P_\ep(S)$ will depend on distinguishing the cycles of $P_\ep(S)$ that correspond to real cycles from $S$ from those that are merely artifacts of the pixelation. To discard  these holes, we  adopt a strategy inspired from Morse theory.

\begin{definition}
Let $S\subset \bR^2$ be a compact set and $\ep>0$ 

\begin{enumerate}

\item For every $\ep$-generic $x$ we set  
\[
\bn_\ep(x) =\bn_{S,\ep}:= \# \text{ of connected components of }C_\ep(S,x).
\]
(When the  set $S$  is  understood from context we use the simpler notation $\bn_\ep$ instead of $\bn_{S,\ep}$.) We will refer to $\bn_\ep(x)$ as the \emph{stack counter function} of $S$.

\item We define 
\[
\bn(x) =\bn_S(x):= \# \text{ of components of }\bigl\{\,y\in\bR;\;\;(x,y)\in  S\,\bigr\}.
\]
We will refer to $\bn_S(x)$ the \emph{component counter} of  $S$.

\item A \emph{jumping point}  of $\bn_S$ is a  real number $x_0$ such that  
\[
\bn(x_0)\neq \bn(x_0^-):=\lim_{x\nearrow x_0} \bn(x) \;\;\mbox{or}\;\;\bn(x)\neq \bn(x_0^+):=\lim_{x\searrow x_0}\bn(x).
\]
 We denote by $\eJ_S$ the set of jumping points of $\bn_S$. We will refer to $\eJ_S$ as the \emph{jumping set} of $S$.

\item A \emph{jumping point} of $\bn_\ep$ is a  real number $x_0\in\ep\bZ$ such that  
\[
\bn_\ep(x_0^-)\neq \bn(x_0^+).
\]
 We denote by $\eJ_{S,\ep}$ the set of jumping points of $\bn_{S,\ep}$.  We will refer to it as the $\ep$-\emph{jumping set} of $S$.
\end{enumerate}\qed
\end{definition}


The function $\bn_\ep$ tells us how many stacks are in a column.   A cycle has two ``walls'' and a ``gap.'' That is to say, if there is a hole in $P_\ep(S)$, there will be a set of columns which have more components than their neighbors, since for a hole to close up stacks must overlap.  Thus jumps  of $\bn_\ep$ are a first indicator of the presence of    cycles in $P_\ep(S)$.  To decide  whether they are    holes, as opposed to cycles coming from $S$  we will rely on our next key technical result which is a substantial refinement of  Theorem \ref{thm: inc}. 

\begin{theorem}[Separation Theorem]
Let $f,g:[a,b]\ra \bR$ be  two Lipschitz continuous  functions such that $f(x) \le g(x)$, $\forall x\in[a,b]$. Fix $\ep>0$ and suppose that $\exists x_0\in  [a,b]$ such that
\begin{equation}
 g(x_0) - f(x_0 )\leq g(x)-f(x),\;\;\forall x\in[a,b],
\label{eq: gap0}
\end{equation}
Then  for any  $\ep>0$ such that
\begin{equation}
\bigl(\,3+ \min(\|f'\|_{\infty},\|g'\|_{\infty})\,\bigr) \ep \leq g(x_0) - f(x_0 )
\label{eq: gap1}
\end{equation}
and any  $\ep$-generic $x \in [a,b] \setminus \ep \bZ$  the column $C_\ep(G,x)$ has two components.  In other words,  if
\[
\min_{x\in[a,b]}\bigl(\, g(x)-f(x)\,\bigr)\geq  \bigl(\,3+ \min(\|f'\|_{\infty},\|g'\|_{\infty})\,\bigr) \ep ,
\]
then for any $\ep$-generic $x\in [a ,b]$  we have
\[
\bn_{G,\ep}(x)=\bn_G(x).
\]
\label{thm: sep}
\end{theorem}

\begin{proof} 
Note that 
\[
C_\ep(G, x) = C_\ep(f, x) \cup C_\ep(g, x),
\]
and furthermore,  Theorem \ref{thm: ivt} implies that each of these columns is connected.  Therefore $C_\ep(G, x)$ will have two components only if $C_\ep(f,x)$ and $C_\ep(g,x)$ do not intersect.  Since $f \le g$ and $x$ is $\ep$-generic, this will occur when 
\[
T_\ep(f,x) < B_\ep(g,x)
\]
or equivalently,
\[
B_\ep(g, x) - T_\ep(f,x) \geq \ep
\]

Now fix $x\in[a,b]$ and let $i \in \bZ$ such that $i \ep < x < (i+1)\ep$.   Choose $x_f,x_g\in [i\ep, (i+1)\ep)]$ such that
\[
f(x_f)=\max_{x\in [i\ep,(i+1)\ep]} f(x),\;\; g(x_g)= \min_{x\in [i\ep,(i+1)\ep]} g(x).
\]
Therefore we have
\[
B_\ep(g,x)\geq g(x_g)-\ep,\;\;  T_\ep(f,x)\leq f(\,x_f\,)+\ep,
\]
so that
\[
B_\ep(g,x)-T_\ep(f,x)\geq  g(x_g)-f(x_f) -2\ep.
\]
We distinguish two cases.

\medskip

\noindent {\bf Case 1.} $\|g'\|_\infty\leq \|f'\|_\infty$. We have
\[
B_\ep(g,x)-T_\ep(f,x)=g(x_g)-g(x_f)+g(\,x_f\,)-f(\,x_f\, ) -2\ep
\]
\[
\stackrel{(\ref{eq: gap0})}{\geq}  g(x_g)-g(x_f)+ g(x_0)-f(x_0)-2\ep
\]
\[
\geq g(x_0)-f(x_0)-\ep\bigl(\,2+\|g'\|_\infty\,\bigr)\stackrel{(\ref{eq: gap1})}{\geq}\ep.
\]
{\bf Case 2.}   $\|f'\|_\infty\leq \|g'\|_\infty$.   We have
\[
B_\ep(g,x)-T_\ep(f,x)=g(x_g)-f(x_g)+f(\,x_g\,)-f(\,x_f\, ) -2\ep
\]
\[
\stackrel{(\ref{eq: gap0})}{\geq} g(x_0)-f(x_0)+ f(x_g)-f(x_f) -2\ep
\]
\[
\geq g(x_0)-f(x_0)-\ep\bigl(\,2+\|f'\|_\infty\,\bigr)\stackrel{(\ref{eq: gap1})}{\geq}\ep.
\]
\end{proof}

To proceed further we need to  introduce some basic terminology.   
 
We define a  \emph{convex polygon} to be a compact  set in $\bR^2$ that is the intersection of  finitely many  closed half-planes.  Note that points, and straight line segments are examples of  convex polygons.       A \emph{$PL$ set} in $\bR^2$ is a  finite union of convex polygons. Note that the topological boundary of a $PL$ set is a finite union of   straight line segments and  points.   A  \emph{vertex}  of a $PL$ set $S$  is  a point $p$ on the topological  boundary $\pato S$  of $S$ such that,       for all $r>0$ sufficiently small,  the intersection  of $\pato S$ with the closed ball of radius $r$ and center $p$ is not a straight line segment. We denote by $\eV_S$ the set of vertices of a $PL$-set $S$.



\begin{definition}   
\begin{enumerate}
\item A \emph{convex polygon} is a compact subset of $\bR^2$ which is the intersection of finitely many closed half-planes (note that line segments and points are examples of convex polygons).
\item A \emph{$PL$ set} (or \emph{piecewise linear set}) in $\bR^2$ is a finite union of convex polygons.
\item A \emph{vertex} of a $PL$ set $S$ is a point $p$ on the topological boundary $\pato S$ such that for all sufficiently small $r > 0$, the intersection of $\pato S$ with the closed ball of radius $r$ and center $p$ is not a straight line segments.
\item For a $PL$ set $S \subset \bR^2$, the set of vertices $eV_S$ is the collection of all vertices in $S$.
\item A $PL$ subset in $S\subset \bR^2$  is called \emph{generic} if for any two vertices $p_1, p_2 \in \eV_S$ 
\[
x(p_1) \neq x(p_2).
\]
\end{enumerate}\qed
\label{def: generic PL}
\end{definition}
We will restrict our approximation technique to the $PL$ case in order to simplify the conclusion of Theorem \ref{thm: sep}.  Since our technique for determining holes will be motivated by applying Morse theory to projection onto the $x$-axis, we will also require the $PL$ set to be generic (to avoid complications arising from clusters of critical points sharing the same critical value).
Note that for a $PL$ set  $S$, the set of jumping points  is contained in the set of $x$-coordinates  of the vertices  of $S$, that is
\[
\eJ_S\subset \bigl\{ x(p);\;\;p\in\eV_S\,\bigr\}.
\]
This inclusion could be strict. Take for example  a $V$-shaped set, with the  bottom vertex  of the letter $V$ situated at the origin. Then $0$ is not a jumping point of $\bn$.




\begin{theorem}
Let $S$ be a generic $PL$ set with jumping set $\eJ_S$. Then there exist  $\nu =\nu(S)\in \bZ_{>0}$ and $\ep_0=\ep_0(S)>0$, \emph{depending only on $S$}, such that,  if  $0<\ep <\ep_0$ and $x$ is an $\ep$-generic value such that
\[
\dist(x,\eJ_S)\geq \nu\ep,
\]
 then $\bn_{S,\ep}(x) = \bn_S(x)$.
\label{thm: noisebound}
\end{theorem}

\begin{proof} Let  $x_0<x_1<\cdots <x_\ell$ be the jumping points of  $\bn=\bn_S$.   We set
\[
\Delta x_i:=x_i-x_{i-1},\;\;\forall i=1,\dotsc, \ell,\;\;\Delta:=\min_{1\leq i\leq \ell} \Delta x_i.
\]
 Note that $\bn(x)$ is constant on each of the intervals $(x_{i-1},x_i)$.  For $i=1,\dotsc, \ell$   we set
\[
S_i:=\bigl\{ (x,y)\in S;\;\;x\in [x_{i-1},x_i]\,\bigr\}.
\]
The set $S_i$ is a disjoint union of   elementary sets (see Definition \ref{def: s-type} for notations)
\[
S(\beta_{i,j}, \tau_{i,j}),\;\;j=0,\dots, p_i,
\]
``stacked one above the other'', i.e.,
\begin{equation}
\beta_{i,0}(x)\leq \tau_{i,0}(x)< \beta_{i,1}(x) \leq \tau_{i,1}(x)< \cdots < \beta_{i,p_i}(x)\leq \tau_{i,p_i}(x),\;\;\forall x\in(x_{i-1},x_i).
\label{eq: orders}
\end{equation}
From  Proposition \ref{pro: ivt2} we deduce that for any $\ep$-generic $x\in (x_{i-1},x_i)$ we have $\bn(x)=p_i$. 

Each of the functions $\beta_{i,j}$ and $\tau_{i,j}$ is piecewise linear. Let $r_{i,j}$ be the smallest    width\footnote{The width of a line segment is the length of its projection on the $x$-axis.}  of a line segment of the graphs $\beta_{i,j}$ and $\tau_{i,j}$, and we set
\begin{equation}
r_i:= \frac{1}{3}\min_{0\leq j\leq p_i} r_{i,j}. 
\label{eq: ri}
\end{equation}
By definition, $r_i\leq\frac{1}{3}\Delta x_i$.  We set
\[
S[x_{i-1}+r_i, x_i-r_i]:=\bigl\{ \,(x,y)\in S;\;\; x\in [x_{i-1}+r_i, x_i-r_i]\,\bigr\}.
\]
The set $S[x_{i-1}+r_{i-1}, x_i-r_i]$  is a collection of $p_i$ elementary sets stacked  one above the other which have positive Hausdorff distance between them. Theorem \ref{thm: inc} implies that there exists $\delta_i>0$ such that, for $\ep\in(0,\delta_i)$  the pixelation $P_\ep\bigl(\,S[x_{i-1}+r, x_i-r_i]\,\bigr)$ has exactly $p_i$ components. Therefore we have proven that
\begin{equation}
\exists\delta_i>0\;\;\mbox{such that};\;\forall \ep\in (0,\delta_i)\;\; x\in [x_{i-1}+r_i,x_i-r_i]:\;\;\bn(x)=\bn_\ep(x).
\label{eq: away}
\end{equation}
Set 
\[
y_i:=x_{i-1}+r_i, \;\;z_i:= x_i-r_i. 
\]
On the interval $[x_{i-1},y_i]$ each of the functions $\beta_{i,j}$ and $\tau_{i,j}$ is linear and we denote by $m^-(\beta_{i,j})$ and respectively $m^-(\tau_{i,j})$ their slopes. For each $x\in (x_{i-1},x_i)$ and each  $j=1,\dotsc, p_i$ we  define the gaps 
\[
\gamma_{i,j}(x)=\beta_{i,j}(x)-\tau_{i,j-1}(x),\;\;\gamma_{i,j}:=\min\bigl\{ \gamma_{i,j}(x_{i-1}),\gamma_j(y_i)\bigr\}=\min_{x\in [x_{i-1},y_i]}\gamma_{i,j}(x),
\]
\[
\Gamma_{i,j}:=\max\bigl\{ \gamma_{i,j}(x_{i-1}),\gamma_{i,j}(y_i)\bigr\}=\max_{x\in [x_{i-1},y_i]}\gamma_{i,j}(x).
\]
We plan to invoke the Separation Theorem \ref{thm: sep}.  We want to prove that there exists $\ep_0>0$ and $\nu>0$ such that for $\ep<\ep_0$ we have
\begin{equation}
\min\bigl\{ \gamma_{i,j}(x);\; x_{i-1}+\nu\ep\leq x\leq y_i-\nu\ep\,\bigr\}\geq  \Bigl(\,3+\min\bigl(\, |m^-(\beta_{i,j})|, |m^-(\tau_{i,j-1})|\,\bigl)\,\Bigr)\ep.
\label{eq: min-gap}
\end{equation}
Note that if $2\nu\ep<\Delta x_i$, then
\[
\min\bigl\{ \gamma_{i,j}(x);\; x_{i-1}+\nu\ep\leq x\leq y_i-\nu\ep\,\bigr\}=\gamma_{i,j}+\nu |m^-(\beta_{i,j})-m^-(\tau_{i,j-1})|\ep.
\]
We can now rewrite (\ref{eq: min-gap}) as
\begin{equation}
\gamma_{i,j}\geq \Bigl(\,3+\min\bigl(\, |m^-(\beta_{i,j})|, |m^-(\tau_{i,j-1})|\,\bigl) - \nu |m^-(\beta_{i,j})-m^-(\tau_{i,j-1})| \,\Bigr)\ep.
\label{eq: min-gap1}
\end{equation}
To solve the last inequality  we distinguish two cases.

\medskip

\noindent {\bf Case 1.} $\gamma_{i,j}=0$. In this case the slope of $\beta_{i,j}$ must be different from the slope of $\tau_{i,j-1}$ and we choose $\nu=\nu^-_{i,j}$ large enough so that
\[
3+\min\bigl(\, |m^-(\beta_{i,j})|, |m^-(\tau_{i,j-1})|\,\bigl) - \nu |m^-(\beta_{i,j})-m^-(\tau_{i,j-1})| <0,
\]
e.g.,
\[
\nu^-_{i,j}=\left\lfloor\frac{3+\min\bigl(\, |m^-(\beta_{i,j})|, |m^-(\tau_{i,j-1})|\,\bigl) }{|m^-(\beta_{i,j})-m^-(\tau_{i,j-1})|}\right\rfloor+1.
\]
We then choose $\ep^-(i,j)$  small enough such that $2\nu_{i,j}\ep_0<\Delta x_i$, e.g.,
\[
\ep^-_0(i,j)=\frac{\Delta}{10\nu_{i,j}}.
\]
{\bf Case 2.} $\gamma_{i,j}\neq 0$. In this case    we choose  $\ep_0=\ep_0(i,j)$  small enough such that
\[
\gamma_{i,j}>\Bigl(\, 3+\min\bigl(\, |m^-(\beta_{i,j})|, |m^-(\tau_{i,j-1})|\,\bigl) \,\Bigr)\ep_0,
\]
e.g.,
\[
\ep^-_0(i,j)=\frac{\gamma_{i,j}}{2\Bigl(\, 3+\min\bigl(\, |m^-(\beta_{i,j})|, |m(\tau_{i,j-1})|\,\bigl) \,\Bigr)},
\]
Next we choose $\nu=\nu^-_{i,j}$ such that $2\nu\ep_0<\Delta x_i$, e.g.,
\[
\nu^-_{i,j}=\left\lfloor\frac{\Delta }{10\ep_0(i,j)}\right\rfloor.
\]
Finally, we define set
\[
\ep_i^-:=\min_{0\leq j\leq p_i}\ep^-(i,j),\;\; \nu_i^-=\max_{0\leq j\leq p_i}\nu^-_{i,j}.
\]
Theorem \ref{thm: sep} implies
\begin{equation}
\exists \ep_i^->0,\;\;\nu_i^->0\;\;\mbox{such that}\;\;\forall \ep<\ep_i^-,\;\;\forall x\in [x_{i-1}+\nu_i^-\ep, x_{i-1}+r_i]:\;\;\bn(x)=\bn_\ep(x).
\label{eq: away1}
\end{equation}
Arguing in a similar fashion      we deduce
\begin{equation}
\exists \ep_i^+>0,\;\;\nu_i^+>0\;\;\mbox{such that}\;\;\forall \ep<\ep_i^+,\;\;\forall x\in [x_i-r_i, x_i-\nu_i^+\ep]:\;\;\bn(x)=\bn_\ep(x).
\label{eq: away2}
\end{equation}
Now set 
\[
\nu_i:=\max(\nu_i^-,\nu_i^+),\;\;\nu:=\max_{1\leq i\leq \ell}\nu_i,
\]
\[
\ep_i:=\min(\ep_i^-,\ep_i^+,\delta_i),\;\;\ep_0=\min_{1\leq i\leq \ell}\ep_i.
\]
Theorem \ref{thm: noisebound} now follows from (\ref{eq: away}), (\ref{eq: away1}) and (\ref{eq: away2}).
\end{proof}

This theorem tells us that jumps of $\bn_\ep$ occur within $\nu(S)$ pixels  from the jumps in $\bn$.  A priori, it could be possible  that, given a jumping point $x_0$ of $\bn$,  there is no jump   in $\bn_\ep$  within $\nu(S)$ pixels   of $x_0$.
\begin{theorem}
Let $S$ be a generic PL set, and $\ep_0=\ep_0(S)$, $\nu=\nu(S)$ as in Theorem \ref{thm: noisebound}. Then, there exists $\ep_1=\ep_1(S)$ such that  if $\ep<\min(\ep_0,\ep_1)$ and $x_0$ is a jumping point of $\bn=\bn_S$, then $\bn_\ep=\bn_{S,\ep}$ has at least one jumping point  in the interval $[x_0 - \nu \ep, x_0 + \nu \ep]$.
\label{th: critgen}
\end{theorem}

\begin{proof} Since $x_0$ is a jumping point of $\bn$ we have
\[
\bn(x_0^+)\neq \bn(x_0)\;\;\mbox{or}\;\;\bn(x_0)\neq \bn(x_0^-).
\]
We distinguish several cases cases.

\medskip 

\noindent {\bf Case 1.}   $\bn(x_0^-) > \bn(x_0)$. Since $S$ is compact, $\bn(x_0)>0$ so that $\bn(x_0^-)\geq 2$.  For this to happen the vertical line of $x=x_0$ must contain at least one vertex of $S$. Since $S$ is generic,  this line contains precisely one vertex of $S$, which we denote by $p_0$.

We can find $\delta>0$ sufficiently small   such that the interval  $[x_0 - \delta, x_0] $ will contain no new jumping points  of $S$.   The  set
\[
S_{[x_0-\delta,x_0]}:=\bigl\{ (x,y)\in S;\;\;x\in [x_0-\delta,x_0]\,\bigr\}
\]
is   disjoint union of   simple types regions 
\[
S(\beta_j, \tau_j),\;\;j=0,\dots, m=\bn(x_0^-)-1,
\]
``stacked one above the other'', i.e.,
\[
\beta_0(x)\leq \tau_0(x)< \beta_1(x) \leq \tau_1(x)< \cdots < \beta_{m}(x)\leq \tau_{m}(x),\;\;\forall x\in(x_0-\delta,x_0),
\]
 where $\beta_j,\tau_j$ are piecewise linear functions.  Since $\bn(x_0)<\bn(x_0^-)$ we deduce that there exists $j_0=1,\dotsc, m$ such that
 \[
 \beta_{j_0}(x_0)=\tau_{j_0-1}(x_0)\;\;\mbox{and}\;\; \gamma_j:=\beta_{j}(x_0)-\tau_{j-1}(x_0)>0,\;\;\forall j\neq j_0
 \]
 In particular, for any $\ep>0$, the $\ep$-stack of $S(\beta_{j_0},\tau_{j_0})$ over $x_0$ touches the stack of $S(\beta_{j_0-1},\tau_{j_0-1} )$ over $x_0$.

 Now choose   $\ep_1$ sufficiently small so that  for $j\neq j_0$  and $\ep<\ep_1$, the $\ep$-stack of $S(\beta_j,\tau_j)$ over $x_0$ is   disjoint  form the $\ep$-stack of $S(\beta_{j-1},\tau_{j-1})$ over $x_0$. Fix   $\ep<\min(\ep_0,\ep_1)$. The above discussion shows that
 \[
 \bn(x_0)=\bn_\ep(x_0^-).
  \]
  Theorem  \ref{thm: noisebound} now implies that
 \[
\bn_\ep((x_0-\nu\ep)^-)= \bn(x_0-\nu\ep)>\bn(x_0).
 \]
 This proves that  the interval $[x_0-r\ep, x_0]$ contains a jumping point of $\bn_\ep$.

\medskip  

\noindent {\bf Case 2.} $\bn(x_0^+)>\bn(x_0)$. This situation  can be reduced to the previous case   via the reflection
\[
\bR^2\ni (x,y)\mapsto (-x,y)\in\bR^2.
\]

\noindent {\bf Case 3.}  $\bn(x_0^-) < \bn(x_0)$.   The vertical line  $x=x_0$ contains a unique vertex $p_0$ of $S$. Moreover, this vertex   has the property that there exists a tiny disk $D$ centered at $p_0$ such that the intersection of $D$ with the open half-plane $\{x<x_0\}\subset\bR^2$ is empty.   In  particular, this shows that $p_0$ is an isolated point of the set
\[
S_{x\leq x_0}=\bigl\{(x,y)\in S;\;\;x\leq x_0\,\bigr\}.
\]
  If $\bn(x_0^-)=0$, the conclusion is obvious.  We assume that $\bn(x_0^-)>0$.     Choose $\delta>0$ such that the interval $[x_0-\delta,x_0)$ contains no jumping point of $S$. Set
\[
R:= S_{[x_0-\delta,x_0]}\setminus \{p_0\}.
\]
Then $R$ is a  union of   simple regions
\[
S(\beta_, \tau_j),\;\;j=0,1\dotsc, m=\bn(x_0^-)-1,
\]
where $\beta_j$ and $\tau_j$ are piecewise linear functions such that
\[
\beta_0(x)\leq \tau_0(x) <\beta_1(x)\leq \tau_1(x)<\cdots <\beta_m(x)\leq \tau_m(x),\;\;\forall x\in [x_0-\delta,x_0].
\]
We can find  $\ep_1=\ep_1(S)$ such that for any $\ep<\ep_1$ and any  $\ep$-generic $x\in [x_0-\delta,x_0]$ we have:
\begin{itemize}
\item $\bn_{R,\ep}(x)=\bn_S(x)=m+1=\bn_S(x_0^-)$, and

\item  the $\ep$-column of $S_{[x_0-\delta,x_0]}$  over  $x_0$  consists of $\bn(x_0)=m+2$ stacks. 

\end{itemize}

Theorem \ref{thm: noisebound} implies that
\[
\bn_{S,\ep}(x)=\bn_S(x)=\bn_S(x_0^-)=m+1\;\; \forall x\in [x_0-\delta,x_0-\nu\ep]\setminus \bZ\ep.
\]
On the other hand, $\bn_{S,\ep}(x_0^-)=m+2$.    Thus the interval $[x_0-\nu\ep,x_0]$ must contain a jumping point of $\bn_{S,\ep}$.

\medskip

\noindent {\bf Case 4.} $\bn(x_0^+)<\bn(x_0)$. This reduces the the previous case via the reflection
\[
\bR^2\ni (x,y)\mapsto (-x,y)\in\bR^2.
\]
\end{proof}

\begin{remark}    Theorem  \ref{thm: noisebound}   states that the two functions  $\bn$ and $\bn_\ep$ coincide  at points situated at a distance at least $\nu(S)$ pixels away from  the jumping points of $\bn$. On the other hand, Theorem \ref{th: critgen} shows that, for a generic $PL$ set,  within $\nu(S)$ pixels  from a jumping point of $\bn$ there must be jumping points of $\bn_\ep$.\qed
\end{remark}

\begin{definition} Let $S$ be a generic $PL$ set in $\bR^2$. 

\begin{enumerate}

\item We will refer to the integer $\nu(S)$ as the \emph{noise range} of $S$.

\item Let $\ep_0(S)$ and $\ep_1(S)$ as defined in the  Theorems \ref{thm: noisebound} and  \ref{th: critgen}. We set
\[
\hbar(S):=\min\bigl(\, \ep_0(S),\;\ep_1(S)\,\bigr),
\]
and  we will refer to it as the \emph{critical resolution} of $S$.
\end{enumerate}\qed
\label{def: constants}
\end{definition}

The next result explains the roles of the noise range and the critical resolution.

\begin{proposition}
Let $S$ be a generic PL set.   If $w>2\nu(S)$,  $\ep<\hbar(S)$ and  $[x_0 - w \ep, x_0 + w \ep]$ contains no jumping points of $\bn_\ep$, then $\bn_\ep(x_0) = \bn(x_0)$.
\label{pro: pix noise bound}
\end{proposition}

\begin{proof} Suppose that  the interval  $[x_0 - w \ep, x_0 + w \ep]$ contains no jumping  points of $\bn_\ep$, yet  $\bn_\ep(x_0) \neq \bn(x_0)$.  Then Theorem \ref{thm: noisebound} implies that the interval 
\[
[x_0 - \frac{w}{2} \ep, x_0 + \frac{w}{2} \ep]
\]
contains a critical $x_1$ point of $\bn$.  Then Theorem \ref{th: critgen} implies that $\bn_\ep$ has a jumping point on the interval 
\[
[x_1 - \frac{w}{2} \ep, x_1 + \frac{w}{2} \ep]
\]  
But since $x_1$ is at most $\frac{w}{2}$-pixels  from $x_0$, this interval is contained within $[x_0 - w \ep, x_0 + w \ep]$. This contradicts the assumption that the interval contained no  jumping points of $\bn_\ep$ and thus  $\bn_\ep(x_0) = \bn(x_0)$.
\end{proof}

Suppose that $\ep<\hbar(S)$, where $S$ is a generic $PL$ set.   Then the theorems proven in this section up to this point imply that all the jumping points  of $\bn_\ep$ are contained   a fixed numbers of pixels from  the jumping set of $\bn$. This simple observation, correctly implemented,  will be the key to  recovering the topology of $S$ from the topology of its sufficiently fine pixelations.

Consider the  discontinuities of the function $\bn_\ep$.  They can only occur within $\nu(S)$  columns from a jumping point of $\bn$.    We do not know what this integer is from the pixelation, but we know that it exists and it  is independent of $\ep$.  Therefore, we know that the  noise range $\nu(S)$  will eventually be less than a properly chosen spread $\sigma(\ep)$ such that $\si(\ep)\ra \infty$ as $\ep\ra 0$.  Using the spread to estimate $\nu(S)$ will be a dramatic overestimation for small $\ep$.  However since $\ep \sigma(\ep) \to 0$, if we declare any cycle which appears less than $\sigma(\ep)$ columns from a jumping point of $n_\ep(S)$ as a fake cycle, we will avoid declaring any real cycles as fake for small resolutions.

The discontinuities  $\bn_\ep(S)$ (for sufficiently small $\ep$)  are obviously contained in the set 
\[
\{x \in \bR  : [x - \ep \sigma(\ep), x + \ep \sigma(\ep)] \text{ contains a critical value of }n_\ep\}.
\]
We would like to consider this set to be the ``noise portion'' of $S$.  This would mean that we could approximate $S$ from $P_\ep(S)$ over $x$-values outside of this region using the results from section 3.  However, recall that in section 3 elementary sets were approximated by choosing upper and lower samples, which were chosen from the midpoints of pixels.  This means that to use the methods from section 3 to approximate $S$ we need our noise intervals to end at the middle of a pixel.  With that in mind we define the noise interval:
\begin{definition}
Let $S$ be a generic $PL$ set, and let $x_1, x_2, \cdots, x_N$ be the jumping points of $\bn_{S,\ep}$.  For each jumping point $x_j$ let the \emph{noise interval} $I_j(\ep)$ be the interval $[a_j, b_j]$ where 
\[
x_j \in [a_j, b_j], 
\]
\[
a_j(\ep), b_j(\ep)\in\frac{\ep}{2}+\ep\bZ
\]
\[
b_j \text{ is the smallest number in }\frac{\ep}{2}+\ep\bZ\text{ such that }(b_j - x_j) > 2 \ep \sigma(\ep)
\]
\[
a_j \text{ is the largest number in }\frac{\ep}{2}+\ep\bZ\text{ such that }(x_j - b_j) > 2 \ep \sigma(\ep)
\]
Define the set of noise intervals $\Delta_\ep$ as
\[
\Delta_\ep= I_1(\ep)\cup \cdots \cup I_N(\ep).
\]\qed
\end{definition}

From the definition we see that $\Delta_\ep$ is a union of intervals.  Furthermore for small $\ep$, each of these intervals will contain precisely one jumping point of $\bn_S$.   Thus $\Delta_\ep(S)$ has as many components as the cardinality of $\eJ_S$. In particular  this implies that  the measure of $\Delta_\ep$ is bounded from above  by $2|\eJ_S| \ep \sigma(\ep)$.  Since $\ep \sigma(\ep)$ vanishes as $\ep \to 0$, this implies that $\Delta_\ep$ has vanishing measure as $\ep \to 0$.

Consider the set of $x$-values which lie outside the noise.  It is a finite union of intervals such that for all $x$ which lie in these intervals, $\bn_\ep(x) = \bn(x)$.   The part of $S$ situated above a each interval is either empty, or a disjoint union of regions of simply types. These types of regions can be approximated using the methods described in Section \ref{s: 3}.  Therefore, to complete the approximation of $S$, we need only describe  how to  deal with  the noise intervals.

The key observation is that  the measure of $\Delta_\ep$  goes to zero as $\ep \to 0$ with the properly chosen spread.  This means that noise intervals make up a very small part of $S$, and so it will not be necessary to approximate them with as high of degree of accuracy as other parts of $S$.  Indeed, we only seek to ensure that the noise intervals capture the correct homotopy  type for small values of $\ep$.  

Since  $S$ is defined by a finite number of piecewise linear functions, for small enough $\ep$, the part of $S$ above $\Delta_\ep$ consists of contractible connected components.  Theorem \ref{thm: inc} implies that these components separate for small enough $\ep$.  Therefore the easiest way to get the correct topology within the noise intervals is to simply cover each component of the pixelation above a noise interval with a rectangle, destroying any fake cycles (or holes)  from $P_\ep(S)$ (see Figure \ref{fig: noise blocks}.)

\begin{figure}[ht]
\centering{\includegraphics[height=2in,width=2in]{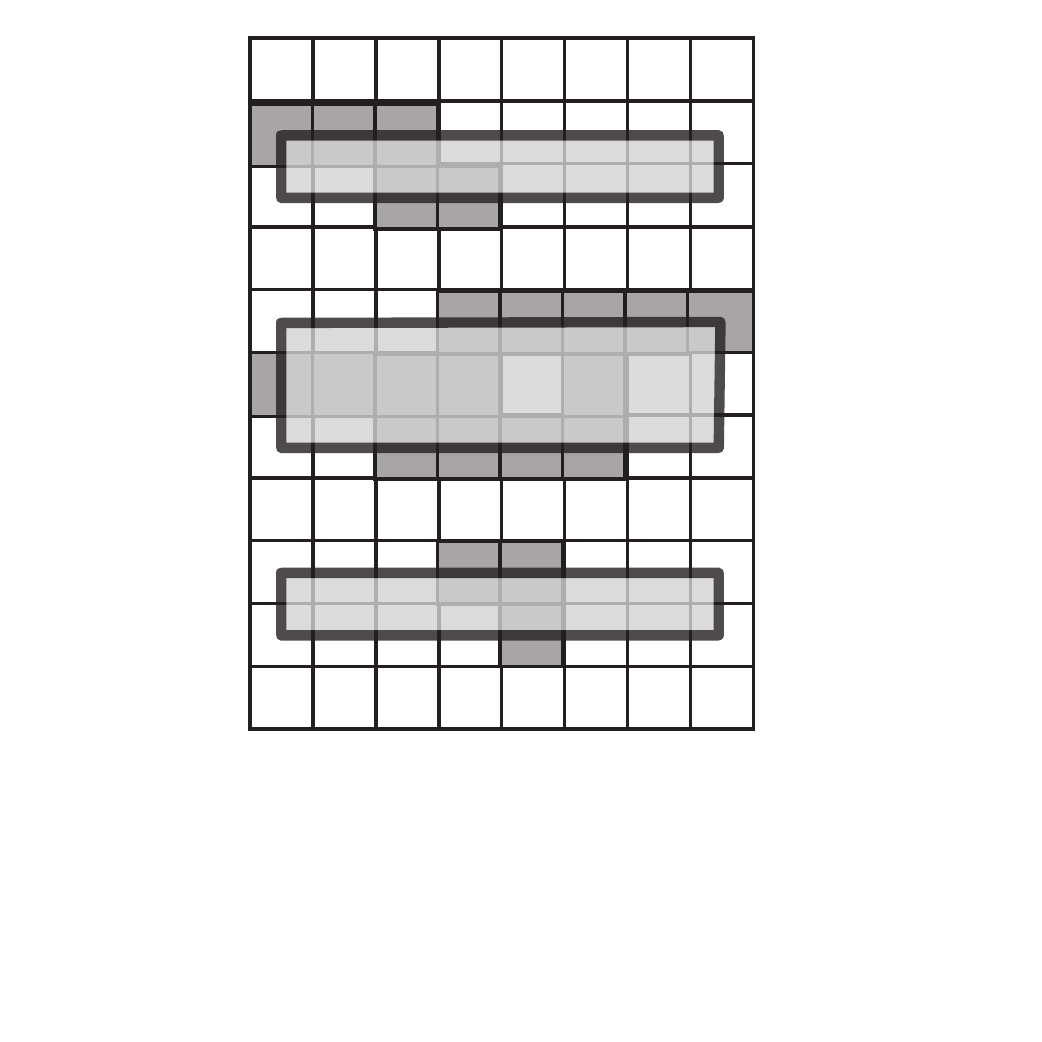}}
\caption{\sl The approximation of a noise interval.  The red rectangles indicate the rectangles that form the approximation.  Since the intersection of a noise interval and a regular interval always lie on the centers of a column, the edges of the rectangles will intersect approximations from the regular intervals.}  
\label{fig: noise blocks}
\end{figure}

 Suppose that  $S$ is a generic  $PL$ set with jump set
\[
\eJ_S=\{x_1< \cdots < x_N\}.
\]
We describe below  an  the algorithm to create a $PL$ approximation $S_\ep$v for $S$  using only information from $P_\ep(S)$ is as follows (this algorithm is restated in a more concrete way in Appendix \ref{s: d}).

\begin{algo}
\begin{enumerate}
\item Choose a spread $\sigma$ such that $\ep \sigma(\ep)^2 \to \infty$ and $\ep \sigma(\ep) \to 0$ as $\ep \to 0$.
\item Let $\Delta_\ep$ be the noise intervals of $P_\ep(S)$.
\item Define $\eR_\ep$ to be the closure of  $\bR\setminus \Delta_\ep$. We call $\eR_\ep$ the \emph{regular} set, and its intervals are called \emph{regular} intervals.
\item For the  bounded regular interval $I \in \eR_\ep$ do the following:
\begin{enumerate} 
\item For each connected component of $P_\ep(S) \cap (I\times \bR)$, the part of $P_\ep(S)$ over  the regular regular intervals $I$  choose compatible upper and lower profiles $\Pi_\ep^+$ and $\Pi_\ep^-$ on that component.
\item For each connected component, choose compatible upper and lower samples $\Xi_\ep^+$ and $\Xi_\ep^-$ with spread $\sigma$.
\item For each connected component, take the PL approximation of the upper and lower samples.
\item Add each PL approximation to the approximation of $S$.
\end{enumerate}

\item For each regular point $x$ denote by $S_\ep(x)$ the part of $S_\ep$ over $x$, i.e., 
\[
S_\ep(x):=S_\ep\cap \{x\}\times \bR.
\]
\item For a noise interval $I_j(\ep)$  we observe first  that Corollary \ref{cor: seperate} implies  that the number of connected components of $P_\ep(S)$ over $I_j(\ep)$ is equal   to the number of components of $S$ over $I_j(\ep)$. Denote by $\eC_j(\ep)$ the set of connected components of $ P_\ep(S)$ over $I_j(\ep)$ 
\begin{enumerate}
\item  For every  $C\in \eC_j(\ep)$ of  $P_\ep(S)$ we set
\[
 U_C:=\max\{ y;\;\;(x,y)\in C\,\bigr\},\;\; L_C:=\min\{ y;\;\;(x,y)\in C\,\bigr\}.
 \]
 \item  Let $P_U^a$ denote the pixel whose top boundary lies on the line $\{y = U_C\}$ and intersects the line $\{x = S_\ep(a_j(\ep))\}$.  Similarly let $P_U^b$ denote the pixel whose top boundary lies on the line $\{y = U_C\}$ and intersects the line $\{x = S_\ep(b_j(\ep))\}$, let $P_L^a$ denote the pixel whose bottom boundary lies on the line $\{y = L_C\}$ and intersects the line $\{x = S_\ep(a_j(\ep))\}$ and $P_L^b$ denote the pixel whose bottom boundary lies on the line $\{y = L_c\}$ and intersects the line $\{x = S_\ep(b_j(\ep))\}$.  (These pixels lie on the corners of a rectangle which bounds the majority of $C$, see Figure \ref{fig: noise blocks}.)
 \item Let $U_C^a$ indicate the center of $P_U^a$ and similarly let $U_C^b$, $U_L^a$ and $U_L^b$ indicate the centers of $P_U^b$, $P_L^a$ and $P_L^B$.
 \item Denote by  $\eP_C(\ep)$ the convex quadrilateral   with vertices
\[
\begin{split}
B_j^-(\ep)= (b_j(\ep), L_C^b),\;\; B_j^+(\ep):=(b_j(\ep), U_C^b),\\
A_j^-(\ep)= (a_j(\ep), L_C^a),\;\; A_j^+(\ep):=(a_j(\ep), U_C^a)
\end{split}
\]
\end{enumerate}
\item    To the set $S_\ep$ constructed  at (4) add the noise  union
\[
\eN_\ep:=\bigcup_{j=1}^N\bigcup_{C\in \eC_j(\ep)}\eP_C(\ep).
\]
\end{enumerate}

This final set $S_\ep$ will be piecewise linear by construction, and will be a good approximation of the set. \qed
\label{alg: process}
\end{algo}

\newpage

\section{The Main Result}
\label{s: 5}
\setcounter{equation}{0}
Consider a generic $PL$ subset $S$ of the Euclidean  plane $\bE=\bR^2$.  In section \ref{s: 4} we created Algorithm \ref{alg: process} which creates a $PL$ approximation $S_\ep$ of $S$.  We wish to state that this approximation converges in a good way to $S$.  An appropriate language to state this condition is that of normal cycles.  

For any  compact   $PL$  subset $ X\subset \bE$ we denote by $N^X$ its normal cycle, \cite{Fu2,LC}. For the reader's convenience we have included in Appendix \ref{s: c} a brief survey of the  basic properties of the normal cycle.  In particular Appendix \ref{s: c} demonstrates how to recover both the Euler characteristic and the perimeter of a set using only calculations on the normal cycle.  Using similar techniques we can extract other geometric and topological information about a set using only its normal cycle.  Therefore weak convergence in normal cycles implies convergence of a great deal of important information, and is an appropriate condition for ``good convergence.''

In the final theorem we see that the approximation $S_\ep$ created by Algorithm \ref{alg: process} does satisfy this type of convergence.

\begin{theorem}  Suppose  $S$ is a generic $PL$ subset of $\bE$,  $\si(\ep)$ is a spread function satisfying (\ref{eq: spi}) and $S_\ep$ is the $PL$ approximation  of $S$ constructed  via the Algorithm \ref{alg: process}. Then the normal cycle   $N^{S_\ep}$ of $S_\ep$ converges weakly to the normal cycle  $N^S$ of $S$.
\label{th: main}
\end{theorem}

\begin{proof} Let us first outline the strategy.   Recall that $\eV_S$ indicates the set of vertices  of $S$  and denote by $\eX_S$   its projection on the $x$-axis.   Since $S$ is generic the projection $\bE\ni (x,y)\mapsto x\in\bR$ induces a bijection $\eV_S\ra \eX_S$.     The jump set $\eJ_S$ is contained in $\eX_S$. We will refer to  the vertices that project  in $\eJ_S$ as \emph{essential} vertices. The other vertices  will be called \emph{inessential}.

For every $c\in \eJ_S$  and any $\ep>0$  small we denote by $I_\ep(c)$ the noise interval that contains $c$.  We set $I_0(c):=\{c\}$. Fix $\ep_0>0$ such that,
\[
\dist\bigl(\,\Delta_\ep(c),\Delta_\ep(c')\,\bigr)\leq \frac{1}{4}\dist(c,c'),\;\;\forall c,c'\in \eJ_S,\;\;0\leq \ve\leq \ep_0.
\]
For $c\in \eJ_S$  and any $0\leq \ve\leq \ep_0$ we denote by $\eN_\ep(c)$ the  noise strip
\[
\eN_\ep(c):=\bigl\{\, (x,y)\in\bR^2;\;\;x\in I_\ep(c)\,\bigr\}.
\]
Note that $\eN_0(c)$ is the vertical line  $\{x=c\}$. Finally we set
\[
\eN_\ep=\bigcup_{c\in \eJ_S} \eN_\ep(c), \;\;\eR_\ep:=\bR^2\setminus \eN_\ep,\;\;\forall 0\leq \ep\leq\ep_0.
\]
For uniformity, we set $S_0:=S$. 

Observe that  there exists $0<\ep_1<\ep_0$ such that, if $\ep\in [0,\ep_1]$, the following hold.

\begin{itemize}

\item  Any  component of $\eN_0\cap S_0$ is contained  in a unique component   of $\eN_\ep\cap S_\ep$.  

\item If  $C$ is  a connected component   of $\eR_\ep\cap S_\ep$,   then  the closure  of $C$  intersects \emph{exactly two} connected components of $\eN_\ep\cap S_\ep$. 

\end{itemize}

For any $\ep\in [0,\ep_1]$ we construct a graph $\Gamma_\ep$ as follows.   The vertex set $\eV_\ep$ of $\Gamma_\ep$ consists  of the connected components of  $\eN_\ep\cap S_\ep$.    The edges  are the connected  components of $\eR_\ep\cap S_\ep$.  We have a well defined  map
\[
\Psi_\ep:  \eV_0 \ra  \eV_\ep,
\]
 that associates to the  a component $C$ of $\eN_0\cap S$  the unique component of $S_\ep\cap\eN_\ep$  that contains $C$. which This is easily seen to be    a bijection. Moreover, it induces an  isomorphism of graphs, i.e.,  the vertices $\bv,\bv'$ are adjacent in $\Gamma_0$ if and only if  the vertices  $\Psi_\ep(\bv)$ and $\Psi_\ep(\bv')$ are adjacent  in $\Gamma_\ep$.  For any $\bv\in\eV_0$ we denote by $E_\bv$ the set of edges of $\Gamma_0$ that are adjacent to $\bv$.

\begin{remark} The  graph  $\Gamma_0$  is known in the literature  as the \emph{Reeb graph} of the (stratified) Morse function
\[
S\ni (x,y)\mapsto x\in\bR.
\]
For more information about this concept we refer to \cite[VI.4]{EH}, or the original source \cite{Reeb1}.\qed
\end{remark}

For any  vertex $\bv\in \eV_0$  and $\ve\in [0,\ep_1]$ we denote by $C_{\bv,\ep}$ the component  of $\eN_\ep\cap S_\ep$  corresponding to $\Psi_\ep(v)$.   Similarly, for any edge $\bse=[\bv,\bv']$ of $\Gamma_0$  and any $\ep\in [0,\ep_1]$ we denote by $C_{\bse,\ep}$ the \emph{closure} of the component of $\eR_\ep\cap S_\ep$ corresponding to  the edge $[\Psi_\ep(\bv),\Psi_\ep(\bv')]$.

\begin{lemma}  For any $\ve\in [0,\ep_1]$ we  have
\begin{equation}
N^{S_\ep} =\sum_{\bv\in \eV_0} N^{C_{\bv,\ep}}+\sum_{\bse\in \eE_0} N^{C_{\bse,\ep}}-\sum_{\bv \in \eV_0}\sum_{\bse\in E_\bv} N^{C_{\bv,\ep}\cap C_{\bse,\ep}}.
\label{eq: key-normal}
\end{equation}
\end{lemma}

\begin{proof}   The proof use the inclusion-exclusion principle, i.e., the equality 
\[
N^{X\cup Y}=N^X+N^Y-N^{X\cap Y}
\]
for any compact $PL$ subsets  $X,Y\subset \bR^2$. 

Note that we have  a  decomposition
\begin{equation}
S_\ep=\left(\bigcup_{\bv\in\eV_0} C_{\bv, \ep}\right)\cup \left(\bigcup_{\bse\in \eE_0} C_{\bse,\ep}\right).
\label{eq:  decomp}
\end{equation}
We need to discuss separately the cases $\ve>0$ and $\ve=0$.

\smallskip

\noindent {\bf 1.} Assume that $\ep\in (0,\ep_1]$. In this case  we have
\begin{equation}
C_{\bv,\ep}\cap C_{\bv',\ep}=\emptyset=C_{\bse,\ep}\cap C_{\bse',\ep},\;\;\forall \bv\neq \bv',\;\;\bse\neq\bse'.
\label{eq: disj1}
\end{equation}
The equality (\ref{eq: key-normal})  now follows from inclusion-exclusion principle  applied to the decomposition (\ref{eq: decomp})   satisfying the overlap conditions (\ref{eq: disj1}).

\smallskip

\noindent {\bf 2.}  Assume that $\ep=0$. In this case the overlap conditions are more complicated.  We have
\begin{subequations}
\begin{equation}
C_{\bv,0}\cap C_{\bv',0}=\emptyset,\;\;\forall \bv\neq\bv',
\label{eq: disj2a}
\end{equation}
\begin{equation}
C_{\bse,0}\cap C_{\bse',0}=\emptyset\Llra \bse\cap\bse'=\emptyset,
\label{eq: disj2b}
\end{equation}
\end{subequations}
where the condition $\bse\cap\bse'=\emptyset$ signifies that the edges $\bse$ and $\bse'$ have no vertex in common. Moreover, 
\begin{equation}
\bigcap_{\bse\in A} C_{\bse, 0} = C_{\bv,0},\;\;\forall \bv\in\eV_0,\;\; A\subset E_\bv.
\label{eq: disj3}
\end{equation}
Using (\ref{eq: decomp}), (\ref{eq: disj2a}),  (\ref{eq: disj2b}),  (\ref{eq: disj3}) and the inclusion-exclusion principle we deduce
\[
\begin{split}
N^S = &\sum_{\bv\in \eV_0} N^{C_{\bv,0}}+\sum_{\bse\in \eE_0} N^{C_{\bse,0}}-\sum_{\bv \in \eV_0}\sum_{\bse\in E_\bv} N^{C_{\bv,0}\cap C_{\bse,0}}\\
&+ \sum_{\bv\in \eV_0} \sum_{\emptyset\neq A\subset E_\bv} (-1)^{|A|+1} N^{C_{\bv}\cap (\bigcap_{\bse\in A}) C_{\bse,0}} + \sum_{\bv\in \eV_0} \sum_{\emptyset\neq  A\subset E_\bv} (-1)^{|A|} N^{\bigcap_{\bse\in A} C_{\bse,0}}
\end{split}
\]
\[
\begin{split}
&=\sum_{\bv\in \eV_0} N^{C_{\bv,0}}+\sum_{\bse\in \eE_0} N^{C_{\bse,0}}-\sum_{\bv \in \eV_0}\sum_{\bse\in E_\bv} N^{C_{\bv,0}\cap C_{\bse,0}}\\
&+\sum_{\bv\in \eV_0}\, \underbrace{\left( \sum_{\emptyset\neq  A\subset E_\bv} \bigl(\,(-1)^{|A|+1}+  (-1)^{|A|}\,\bigr)\,\right)}_{=0}\,N^{C_{\bv,0}}
\end{split}
\]
\end{proof}
Theorem \ref{th: main} is now an immediate consequence of the  following result.

\begin{lemma}  For any $\bv \in \eV_0$, $\bse\in \eE_0$ we have
\begin{subequations}
\begin{equation}
\lim_{\ep\searrow 0}  N^{C_{\bv,\ep}}=  N^{C_{\bv,0}},
\label{eq: lima}
\end{equation}
\begin{equation}
\lim_{\ep\searrow 0}  N^{C_{\bse,\ep}}= N^{C_{\bse,0}},
\label{eq: limb}
\end{equation}
\begin{equation}
\lim_{\ep\searrow 0}N^{C_{\bse,\ep}\cap C_{\bv,\ep}}= N^{C_{\bse,0}\cap C_{\bv,0}},
\label{eq: limc}
\end{equation}
\end{subequations}
where the limits are understood in the sense of weak topology on the space of currents.
\label{lemma: lim}
\end{lemma}

The proof of this lemma relies on the General Convergence Theorem proved by Joseph Fu in \cite{Fu1}. 

\begin{theorem}[Approximation Theorem] Suppose $S$ is a PL subset of the plane and for each $\ep \in [0,\ep_1]$ $S_\ep$ is a PL subset of the plane with the following properties.

\begin{enumerate}
\item There is a compact set $K \subset \bR^2$ which contains each $S_\ep$.
\item There is a $M \in \bR$  such that
\[
{\rm mass}(N^{S_\ep})\leq M,\;\;\forall \ep.
\]
\item For almost every $\xi \in \Hom(\bR^2,\bR)$ and almost every $c \in \bR$ we have
\[
\lim_{\ep \searrow 0} \chi(S_\ep \cap \{\xi \geq c\}) = \chi(S \cap \{\xi \geq c\})
\]
\end{enumerate}
Then $N^{S_\ep}$  converges to $N^S$ as $\ep\ra 0$ weakly and in the flat norm.\qed
\label{thm: fu}
\end{theorem}

Proving  (\ref{eq: lima})-(\ref{eq: limc}) will involve proving each of the three conditions  in the Approximation Theorem.    We begin with an easy   consequence of the Approximation Theorem that will   be   very useful in the sequel.    First, let us define a \emph{convex polygon} in the plane to be the convex hull of a finite set.   Note that this definition allows for degenerate polygons such as line segments or points. We define the perimeter of a segment to be twice its length. For $2$-dimensional polygons the perimeter is defined in the usual way. We will denote by $L(P)$ the perimeter of a polygon.

\begin{lemma}
Suppose $(S_\ep)_{\ep>0}$ is a family of convex polygons in the plane that converge in the Hausdorff  metric to a convex polygon $S$. Then  $N^{S_\ep}$ converges weakly to $N^S$ as $\ep\ra 0$.
\label{lem: quad}
\end{lemma}

\begin{proof}
We argue by proving the conditions of Fu's Theorem.   Observe first that   there exists $R>0$ such that
\[
\dist(S_\ep, S)<R,\;\;\forall \ep
\]
 and thus  the condition (1) of the Approximation Theorem.   The  computations of \cite[Chap. 23]{Mor}  show that mass of the normal cycle of a convex polygon $P$ is  equal to $2\pi +L(P)$.    From  Hadwiger's characterization theorem \cite[Thm. 9.1.1]{KR}  we deduce that
 \[
 \lim_{\ep\ra 0} L(S_\ep)=L(S)
 \]
 and  thus condition (2) is also satisfied.

Therefore we must show that for almost every $\xi \in \Hom(\bR^2,\bR)$ and almost every $c \in \bR$ we have
\[
\lim_{\ep \searrow 0} \chi(S_\ep \cap \{\xi \le c\}) = \chi(S \cap \{\xi \le c \})
\]
Note that $S$ and each $S_\ep$ are all convex subsets of the plane.  Therefore any intersection with a half-plane will either be empty or be a contractible set.  Therefore to prove the convergence of Euler characteristic on half-planes we need only prove that a half plane  $H$ will only intersect $S_\ep$ for small $\ep$ if and only if it intersects $S$.    This is true since $H\cap S_\ep$ converges in the Hausdorff metric to $H\cap S$.
\end{proof}

In several places in the remainder of the proof of Lemma \ref{lemma: lim} it will be convenient to discuss the maximum rate at which $S$ can increase or decrease.  With that in mind we set:
\[
\alpha = \sup_{\beta \text{ is a slope of an edge of }S} |\beta|
\]

\smallskip

\noindent \textbf{Proof of (\ref{eq: lima}).} Note that each $C_{\bv,0}$ is a connected component of a subset of a line.  Therefore each $C_{\bv,0}$ is either a point or a line segment.  By construction of the approximation $S_\ep$, each $C_{\bv,\ep}$ is a rectangle which contains $C_{\bv,0}$.  Therefore to show that $C_{\bv,\ep}$ converges to $C_{\bv,0}$ in the Hausdorff metric we need only show that its width vanishes and its height converges to the height of $C_{\bv,0}$.  By construction the width of $C_{\bv,\ep}$ is $O(\ep \sigma(\ep))$, and since $\sigma$ is a spread this width vanishes as $\ep \to 0$.  Furthermore Proposition \ref{thm: difbound} implies that the difference in the height of $C_{\bv,\ep}$ and $C_{\bv,0}$ is $O(\alpha \ep + \alpha \ep \sigma(\ep))$, where both $\alpha \ep$ and $\alpha \ep \sigma(\ep)$ vanish as $\ep \to 0$.  Therefore $C_{\bv,\ep}$ converge to $C_{\bv,0}$ in the Hausdorff metric, and Lemma \ref{lem: quad} implies that
\[
\lim_{\ep \searrow 0} N^{C_{\bv,\ep}} = N^{C_{\bv,0}}
\]

\smallskip

\noindent \textbf{Proof of (\ref{eq: limc}).} The set $C_{\bv,\ep} \cap C_{\bse, \ep}$ is a line segment on the edge of the rectangle $C_{\bv,\ep}$.  We have already proven that $C_{\bv,\ep}$ converges to $C_{\bv,0}$ in the Hausdorff metric.  Since 
\[
C_{\bv,0}\cap C_{\bse,0} = C_{\bv,0}
\]
to prove that $C_{\bv,\ep} \cap C_{\bse, \ep}$ converges to $C_{\bv, 0} \cap C_{\bse,0}$ in the Hausdorff metric we need only prove that
\[
\lim_{\ep \searrow 0} \sup_{x \in C_{\bv, 0}} \inf_{y \in C_{\bse,\ep} \cap C_{\bv,\ep}} \dist(x,y) = 0.
\]
That is to say, we must show that the maximum distance from a point in $C_{\bv,0}$ to the set $C_{\bse,\ep} \cap C_{\bv,\ep}$ becomes arbitrarily small as $\ep \to 0$.  But note that $C_{\bse,\ep} \cap C_{\bv,\ep}$ is an interval which lies on the edge of a noise interval.  Its distance to a point of $C_{\bv,0}$ in $x$-coordinates is at most the width of the noise interval.  Its distance to a point of $C_{\bv,0}$ in $y$-coordinates is equal to the change of slope of $S$ over the noise interval, together with the error from $S$ to the approximation $S_\ep$.  Thus (once again using the fact that a noise interval has width proportional to $\ep \sigma(\ep)$ together with Proposition \ref{thm: difbound}) we see that
\[
\sup_{x \in C_{\bv, 0}} \inf_{y \in C_{\bse,\ep} \cap C_{\bv,\ep}} \dist(x,y) = O(\ep \sigma(\ep) + \alpha \ep)
\]
which implies that the $C_{\bse,\ep} \cap C_{\bv, \ep}$ converges to $C_{\bv,0} = C_{\bse,0}\cap C_{\bv,0}$ in the Hausdorff metric.  So Lemma \ref{lem: quad} implies that
\[
\lim_{\ep \searrow 0} N^{C_{\bse,\ep} \cap C_{\bv,\ep}} = N^{C_{\bse,0}\cap C_{\bv,0}}
\]

\smallskip

\noindent \textbf{Proof of (\ref{eq: limb}:)} To prove this limit we make further use the inclusion-exclusion principle.




Fix an edge $\bse$.  Let the set $\eX_{\bse,0}$ indicate the projection of vertices of $C_{\bse,0}$ to the $x$-axis (i.e. the subset of $\eX_S$ which contains only projections of vertices in $C_{\bse,0}$).  Write
\[
\eX_{\bse,0} = \bigl\{\,x_0,x_1,\dots,x_n\,\bigr\},
\]
where $x_0,x_1,\dots,x_n$ are arranged in increasing order.  For each integer $1 < i < n$ we define the set
\[
V_{x_i,0} := C_{\bse,0} \cap \{x = x_i\} 
\]
and for each integer $1 < i \le n$ we define the set
\[
R_{x_i,0} := C_{\bse,0} \cap \{x_{i-1} \le x \le x_i\}.
\]
Note that $R_{x_i,0} \cap R_{x_{i+1},0} = V_{x_i,0}$.

Let $\eX_{\bse,\ep}$ be the projection of the vertices of $C_{\bse,\ep}$ to the $x$-axis.  Note that by construction of $S_\ep$ each point in $\eX_{\bse,\ep}$ will have exactly two vertices of $C_{\bse,\ep}$ map onto it.  For each $\ep$ write
\[
\eX_{\bse,\ep} = \bigl\{\,x_{0,\ep},x_{1,\ep}, \dots, x_{m(\ep),\ep}\,\bigr\}
\] 
arranged in increasing order and where $m(\ep)$ is an integer depending on $\ep$.  
\[
V_{x_i,\ep} = C_{\bse,\ep} \cap \bigl\{\,x_{j_{i,\ep},\ep} \le x \le x_{k_{i,\ep},\ep}\bigr\}.
\]
Where $j_{i,\ep}$ is defined to be the largest integer such that $x_{j_{i,\ep},\ep} < x_i$ and conversely $k_{i,\ep}$ is defined to the smallest integer such that $x_{k_{i,\ep},\ep} > x_i$.  The set $V_{x_i,\ep}$ can be thought of as the part of $C_{\bse,\ep}$ between the line segments which cross over $x_i$ (see Figure \ref{fig: vep}).  Note that if the critical value $x_i$  belongs to $\eX_{\bse,\ep}$, then  the set $V_{x_i,\ep}$ will be a hexagon, otherwise it will be a quadrilateral.

\begin{figure}[h]
\centering{\includegraphics[height=2.25in,width=2.25in]{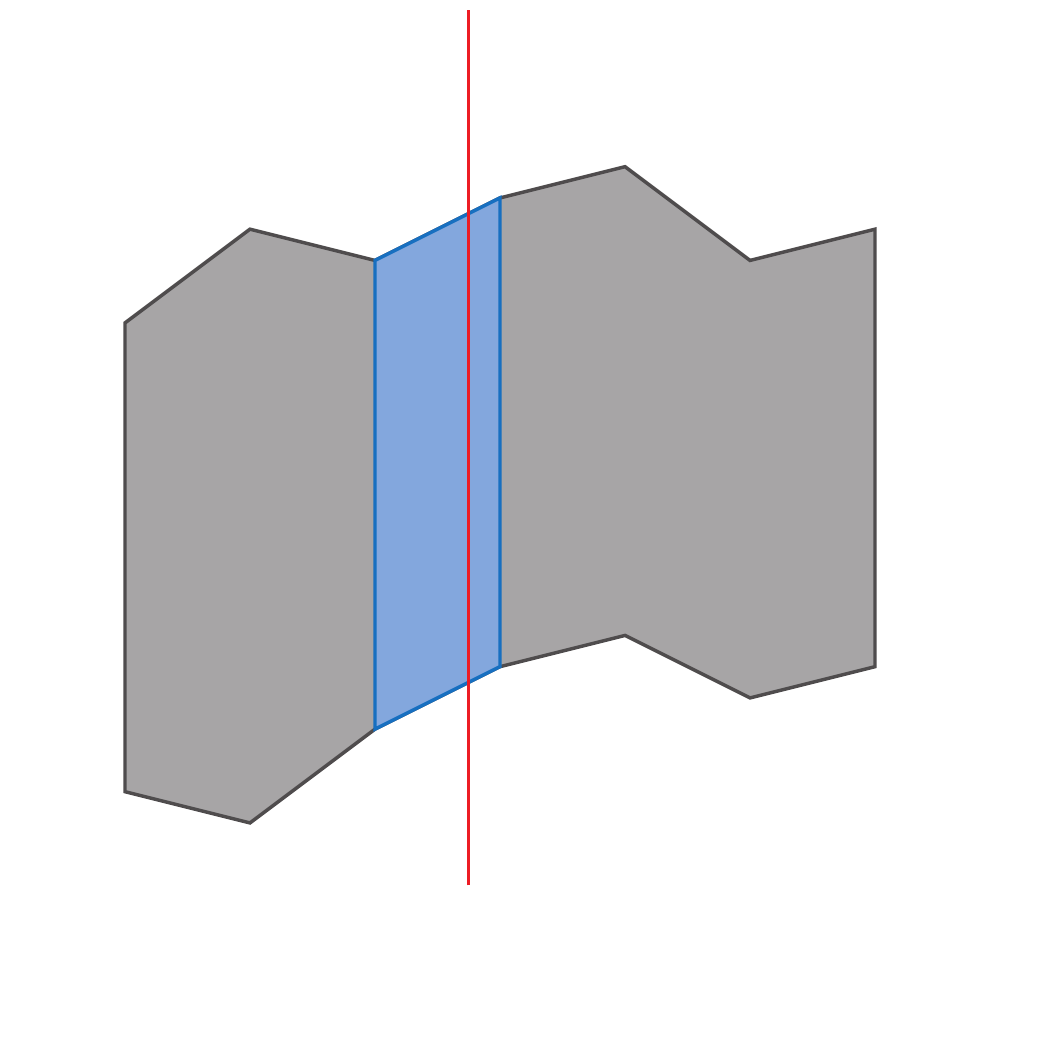}}
\caption{\sl In this picture the grey region is $C_{\bse,\ep}$, the red line is $\{x = x_i\}$ and the blue region is $V_{x_i,\ep}$.}
\label{fig: vep}
\end{figure}

Similarly,  for each $\ep > 0$ and each integer $1 < i \le n$ define the set
\[
R_{x_i,\ep} =: C_{\bse,\ep} \cap \{x_{k_{(i-1),\ep},\ep} \le x \le x_{j_{i,\ep},\ep}\}
\]
where $j_{i,\ep}$ and $k_{i,\ep}$ are defined as before.  Then the set $R_{x_i,\ep}$ is the part of $C_{\bse,\ep}$ which lies between $V_{x_{i-1},\ep}$ and $V_{x_i,\ep}$.

\begin{figure}[h]
\centering{\includegraphics[height=2.25in,width=2.25in]{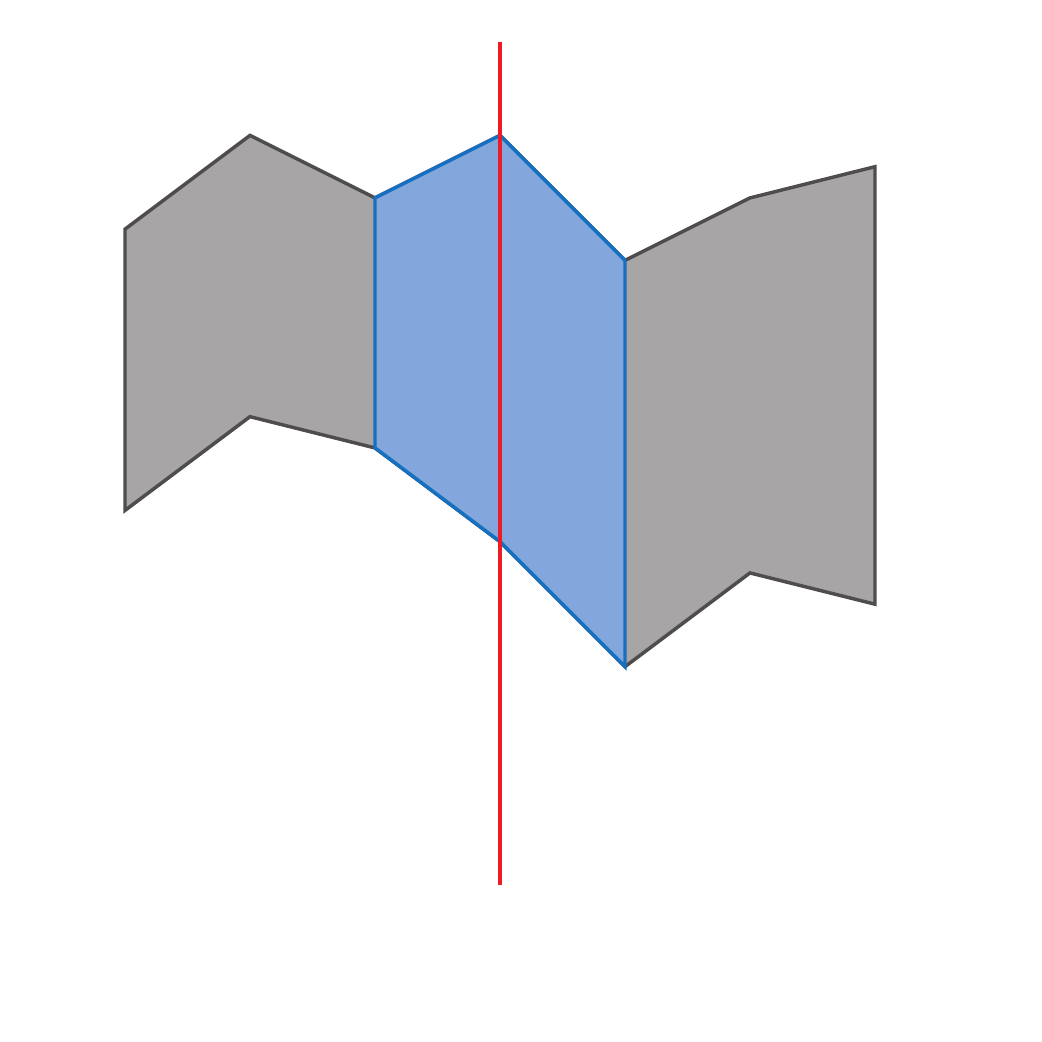}}
\caption{\sl The region $V_{x_i,\ep}$ may be a hexagon if the line $\{x = x_i\}$ occurs at a vertex of $C_{\bse,\ep}$.}
\label{fig: vep2}
\end{figure}

Dividing $C_{\bse,0}$ among the sets $R_{x_i,0}$ and recalling that these sets overlap on $V_{x_i,0}$ we see that
\begin{equation}
N^{C_{\bse,0}} = \sum_{i=1}^n N^{R_{x_i,0}} - \sum_{i=1}^{n-1} N^{V_{x_i,0}}
\label{eq: zeronorm}
\end{equation}

Dividing $C_{\bse,\ep}$ among the sets $R_{x_i,\ep}$ and the sets $V_{x_i,\ep}$ we see that
\begin{equation}
N^{C_{\bse,\ep}} = \sum_{i=1}^n N^{R_{x_i,\ep}} + \sum_{i=1}^{n-1} N^{V_{x_i,\ep}} - \sum_{i=1}^{n-1} N^{R_{x_i,\ep} \cap V_{x_i,\ep}} - \sum_{i=1}^{n-1} N^{V_{x_i,\ep} \cap R_{x_{i+1},\ep}}
\label{eq: epnorm}
\end{equation}
Although (\ref{eq: epnorm}) looks considerably more complicated than (\ref{eq: zeronorm}) note that (\ref{eq: epnorm}) will converge to (\ref{eq: zeronorm}) if the following three equations are true:
\begin{subequations}
\begin{equation}
\lim_{\ep \searrow 0} N^{R_{x_i,\ep}} = N^{R_{x_i,0}}
\label{eq: rconv}
\end{equation}
\begin{equation}
\lim_{\ep \searrow 0} N^{V_{x_i, \ep}} = N^{V_{x_i,0}}
\label{eq: vconvmid}
\end{equation}
\begin{equation}
\lim_{\ep \searrow 0} N^{R_{x_i,\ep} \cap V_{x_i,\ep}} = \lim_{\ep \searrow 0} N^{V_{x_i,\ep} \cap R_{x_{i+1},\ep}}=N^{V_{x_i,0}}
\label{eq: vconvint}
\end{equation}
\end{subequations}

We will prove these equations in reverse order.

\medskip

\noindent \textbf{Proof of (\ref{eq: vconvint}).} We start by proving
\[
\lim_{\ep \searrow 0} N^{R_{x_i,\ep} \cap V_{x_i,\ep}} = N^{V_{x_i,0}}
\]
Note that the intersection of $R_{x_i,\ep}$ and $V_{x_i,\ep}$ is a line segment.  The set $V_{x_i,0}$ is also a line segment, so to prove the convergence of normal cycles we can use Lemma \ref{lem: quad}.  For sufficiently small $\ep$ we can assume that no vertex of $S$ occurs in between $R_{x_i,\ep} \cap V_{x_i,\ep}$ and $V_{x_i,0}$.  Then the top vertex of $R_{x_i,\ep} \cap V_{x_i,\ep}$ will differ in $x$-coordinate from the top vertex of $V_{x_i,0}$ by at most the width of a line segment in $C_{\bse,\ep}$ and differ in $y$-coordinate by a number proportional to the width of the line segment.  This means that the distance from the top vertex of $R_{x_i,\ep} \cap V_{x_i,\ep}$ to the top vertex of $V_{x_i,0}$ is $O(\ep \sigma(\ep))$.  The same is true of the bottom vertices.  Therefore Lemma \ref{lem: quad} implies that
\[
\lim_{\ep \searrow 0} N^{R_{x_i,\ep} \cap V_{x_i,\ep}} =N^{V_{x_i,0}}.
\]
The above reasoning can be repeated to show that
\[
\lim_{\ep \searrow 0} N^{V_{x_i,\ep} \cap R_{x_{i+1},\ep}} =N^{V_{x_i,0}}
\]
which completes the proof of (\ref{eq: vconvint}).

\smallskip

\noindent \textbf{Proof of (\ref{eq: vconvmid}).} We once again make use of the  inclusion-exclusion principle satisfied by the  Normal Cycle.  Divide the set $V_{x_i,\ep}$ into the following two sets:
\[
V_{x_i,\ep}^0 := ((\infty, x_i] \times \bR) \cap V_{x_i,\ep}, \;\;V_{x_i,\ep}^1 := ([x_i,\infty) \times \bR) \cap V_{x_i,\ep}.
\]
Then we have
\[
N^{V_{x_i,\ep}} = N^{V_{x_i,\ep}^0} + N^{V_{x_i,\ep}^1} - N^{V_{x_i,\ep}^0 \cap V_{x_i,\ep}^1}
\]
so that (\ref{eq: vconvmid}) will hold if $N^{V_{x_i,\ep}^0}$, $N^{V_{x_i,\ep}^1}$ and $N^{V_{x_i,\ep}^0 \cap V_{x_i,\ep}^1}$ all converge to $N^{V_{x_i,0}}$.  

Note that the set $V_{x_i,\ep}^0 \cap V_{x_i,\ep}^1$ is a line segment.  If the top and bottom vertex of $V_{x_i,\ep}^0 \cap V_{x_i,\ep}^1$ are sample points of the approximation, then Proposition \ref{thm: difbound} implies that the Hausdorff distance from $V_{x_i,\ep}^0 \cap V_{x_i,\ep}^1$ to $V_{x_i,0}$ is $O(\alpha \ep)$.  If the top and bottom vertices of the line segment are not sample points of the approximation, note that they differ in width from a sample point by at most $\ep \sigma(\ep)$.  Therefore in all cases the Hausdorff distance from $V_{x_i,\ep}^0 \cap V_{x_i,\ep}^1$ to $V_{x_i,0}$ is $O(\alpha \ep + \alpha \ep \sigma(\ep))$.  This vanishes as $\ep \to 0$, so that Lemma \ref{lem: quad} implies that
\[
\lim_{\ep \searrow 0} N^{V_{x_i,\ep}^0 \cap V_{x_i,\ep}^1} = N^{V_{x_i,0}}
\]

Note that each set $V_{x_i,\ep}^0$ and $V_{x_i,\ep}^1$ must be a convex quadrilateral.  The width of their union $V_{x_i,\ep}$ is $O(\ep \sigma(\ep))$ by construction.  This implies that the widths of each of $V_{x_i,\ep}^0$ and $V_{x_i,\ep}^1$ vanish as $\ep \to 0$.  Therefore each set converges to $V_{x_i,\ep}^0 \cap V_{x_i,\ep}^1$ in Hausdorff distance as $\ep \to 0$.  This implies
\[
\lim_{\ep \searrow 0} N^{V_{x_i,\ep}^0} = \lim_{\ep \searrow 0} N^{V_{x_i,\ep}^1} = N^{V_{x_i,0}}
\]
completing the proof of (\ref{eq: vconvmid}).

\smallskip

\noindent \textbf{Proof of (\ref{eq: rconv}).} The major difficulty in this case is that the regions $R_{x_i,\ep}$ need not be convex polygons so we cannot apply Lemma \ref{lem: quad}.  Thus we must return to proving the three conditions of Theorem \ref{thm: fu} directly.

Note that each $R_{x_i,\ep}$ can be considered the PL approximation with spread $\sigma(\ep)$ of the region $R_{x_i,0}$.  The region $R_{x_i,0}$ is simply an elementary set which lies between two line segments.  Theorem \ref{thm: profileApprox} implies that the distance of the boundary of $R_{x_i,\ep}$ to the boundary of $R_{x_i,0}$ is $O(\ep + \ep \sigma(\ep) + (\ep \sigma(\ep))^2)$.  Therefore the sets $R_{x_i,\ep}$ must all be contained in some compact subset of $\bR^2$.

Since $R_{x_i,\ep}$ is the PL approximation of an elementary set with spread $\sigma(\ep)$, Corollary \ref{cor: piece curv} implies that the total curvature of its boundary converges to the total curvature of the boundary of $R_{x_i,0}$.  The computations of \cite[Chap. 23]{Mor} show that the mass of a normal cycle of a polygon is equal to the perimeter of the polygon plus the total curvature of its boundary.  Therefore the mass of $N^{R_{x_i,\ep}}$ cannot increase to infinity.

The final condition to prove from Theorem \ref{thm: fu} is that the Euler characteristic of intersection of $R_{x_i,\ep}$ with any generic half-plane converges   to the Euler characteristic of the intersection of $R_{x_i,0}$ with the same half-plane.  To prove this we will need a technical lemma whose proof we will defer until later.

\begin{figure}[h]
\centering{\includegraphics[height=2.25in,width=2.25in]{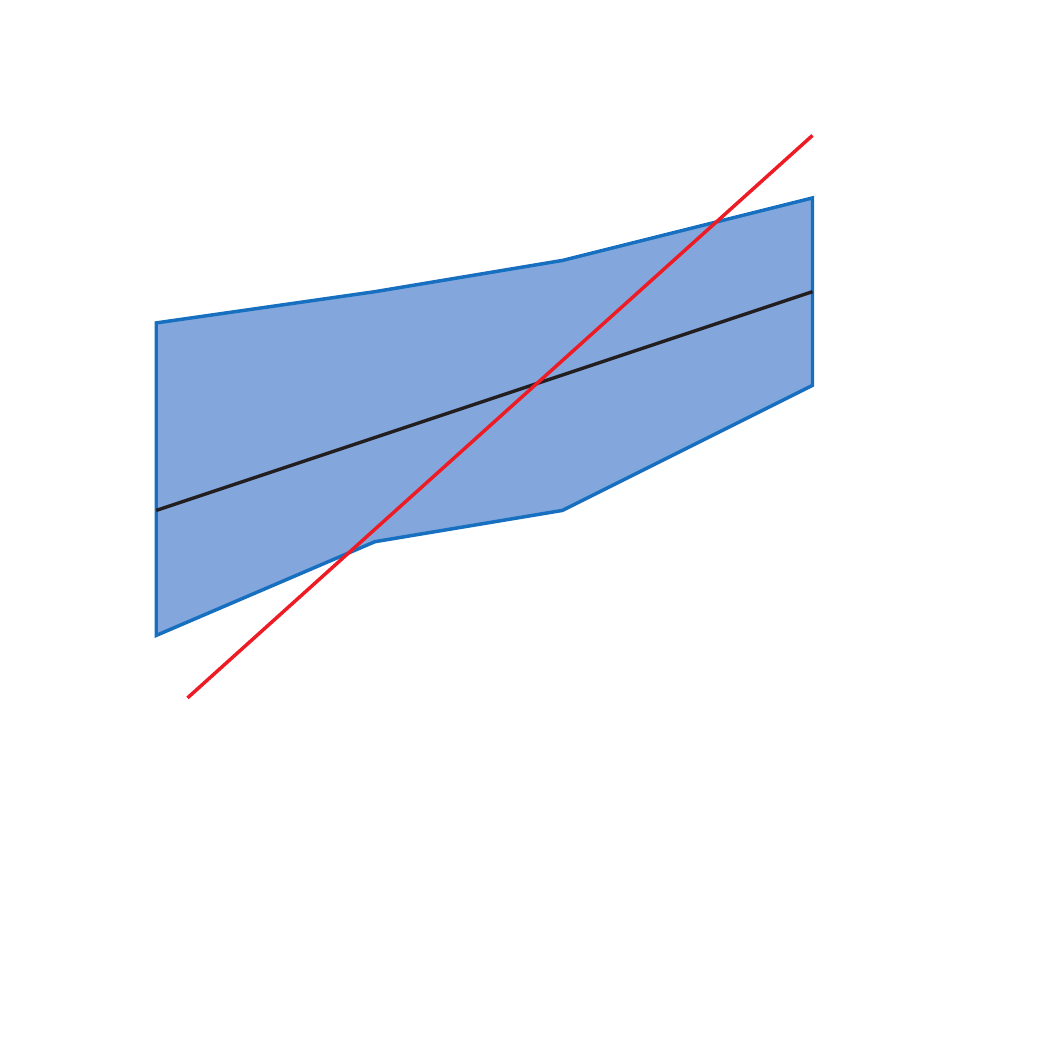}}
\caption{\sl The situation of Lemma \ref{lem: tech}.  The red line only intersects the boundary of the blue approximation twice.}
\label{fig: rep}
\end{figure}

\begin{lemma}
Let $S$ be a single non-vertical line segment and for each $\ep > 0$ let $\Xi_\ep$ an upper or lower sample of $S$ with spread $\sigma(\ep)$ and for every $\ep$ let $f_\ep$ indicate the PL-approximation  of $S$ determined by $\Xi_\ep$.

Let $\xi \in \Hom(\bR_2, \bR)$ and suppose that the line $\{\xi = c\}$ is not parallel to $S$ and intersects $S$.  Then there exists a $\ep^* > 0$ depending only on $S$ and $\xi$ such that for $0 < \ep < \ep_0$, the line $\{\xi = c\}$ intersects the graph of $f_\ep$ exactly once. \qed
\label{lem: tech}
\end{lemma}

Given this lemma, fix a $c \in \bR$ and a $\xi \in \Hom(\bR^2,\bR)$ such that the line $\{\xi = c\}$ is not parallel to a line segment of $R_{x_i,0}$.   We distinguish two cases.

\smallskip

\noindent {\bf A.} \emph{The line  $\{\xi = c\}$ does not intersect $R_{x_i,0}$.}   Theorem \ref{thm: profileApprox} implies that for all sufficiently small $\ep$, the line $\{\xi = c\}$ will not intersect $R_{x_i,\ep}$.  Since $R_{x_i,\ep}$ converges to $R_{x_i,0}$ in Hausdorff measure, the half-plane will contain $R_{x_i,\ep}$ for small $\ep$ if and only if it contains $R_{x_i,0}$.  The sets $R_{x_i,0}$ and $R_{x_i,\ep}$ both must be contractible since they are regions which lie between the graphs of two piecewise linear functions which do not intersect.  Therefore in this case
\[
\lim_{\ep \searrow 0} \chi(\{\xi \geq c\} \cap R_{x_i,\ep}) = \chi(\{\xi \geq c\} \cap R_{x_i,0}).
\]

\smallskip

\noindent {\bf B.} \emph{The line $\{\xi = c\}$ does intersect the set $R_{x_i,0}$.}  Since $R_{x_i,0}$ is convex, we clearly have
\[
\chi(\{\xi \geq c\} \cap R_{x_i,0}) = 1.
\]
Let $\ep \in (0,\ep^*(R_{x_i,0}))$, where $\ep^*(R_{x_i,0})$ is the constant guaranteed by Lemma \ref{lem: tech}.  Then the desired convergence in Euler characteristic is an immediate consequence of the following lemma.

\begin{lemma}(a) For any $d\geq c$ the intersection between the line $\{\xi=d\}$ and $R_{x_i,\ep}$ is either empty or   a closed   segment (possibly degenerate).

\noindent (b) $\chi\bigl(\, \{\xi \geq c\} \cap R_{x_i,\ep}\,\bigr)=1.$
\label{lemma: euler}
\end{lemma}

\begin{figure}[h]
\centering{\includegraphics[height=2.25in,width=2.25in]{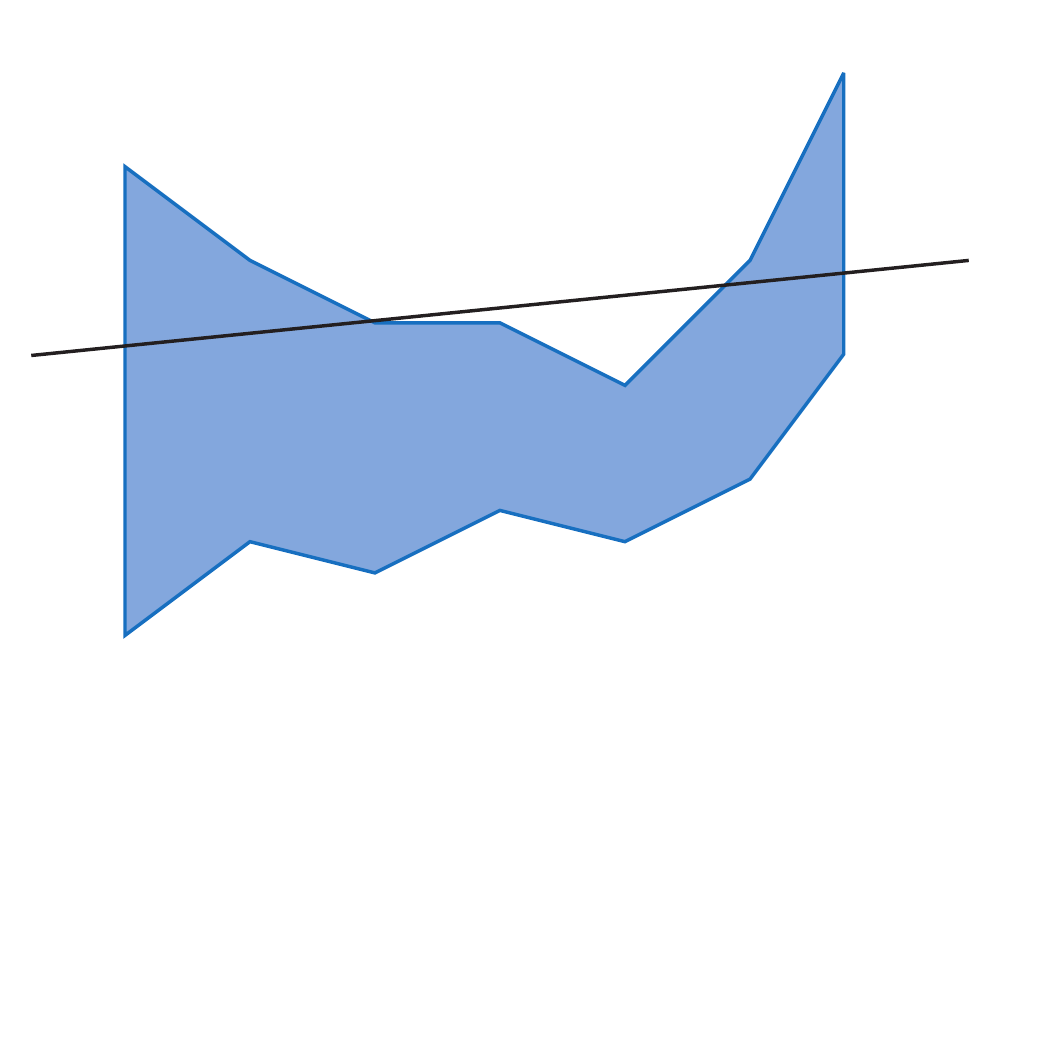}}
\caption{\sl For a line to intersect $R_{x_i,\ep}$ in a non-connected fashion it must intersect either the upper or the lower boundary multiple times.}
\label{fig: nonconnect}
\end{figure}

\begin{proof} (a) Assume  that the line $\{\xi=d \}$ is the  graph of the linear function $\ell(x)=mx+b$. The  boundary of $R_{x_i,\ep}$ has  four component: a top, a bottom  and two vertical side components. The top and boundary components are  graphs of $PL$ functions $T_\ep(x)$ and respectively $B_\ep(x)$ (see Figure \ref{fig: nonconnect}). Lemma \ref{lem: tech} implies   that each line $\{\xi = d\}$ intersects each of these components at most once so the intersection between the line $\{\xi=d\}$ and the boundary of $R_{x_i,\ep}$  consists of at most  four points.

We consider the set $\{\xi = d\} \cap R_{x_i,\ep}$ which is the union of some number of line segments.  Let $p_1, p_2, \dots p_{n}$, $n\leq 4$, be the endpoints of these line segments  arranged increasingly according to their $x$-coordinates.  We claim that $n\leq 2$ so that the intersection between $\{\xi=d\}$ and   $R_{x_i,\ep}$ is a (possibly degenerate) line segment.   To prove this we argue by contradiction.

Suppose  that the  intersection consists  of at least three points $p_1,p_2, p_3$. Let $(t_j,u_j)$ be the coordinates of $p_j$, $1\leq j\leq 3$,  so that $t_1<t_2<t_3$.  Think of the line $\{\xi=d\}$ as the trajectory of a particle  moving in the plane with constant velocity 
\[
x=t, \;\;y=\ell(t)=mt+b.
\]
  The particle  first enters $R_{x_i,\ep}$ at the moment   $t=t_1$ and exits at the moment $t=t_2$. Note that the corresponding exit point $p_2$ cannot be on  either one  of the vertical components of the boundary of $R_{x_i,\ep}$ so it must be either on the top part or on the bottom   of the boundary. Assume that $p_2$ is on the top part of the boundary.  (The case when $p_2$ is on the bottom  is dealt with in a similar fashion.)

Lemma \ref{lem: tech} implies that the particle will never intersect  the top part of the boundary ever again so that
\[
\ell(t)> T_\ep(t)\geq B_\ep(t),\;\;\forall t>t_2.
\]
This implies  that at the  moment  $t_3$ when the particle enters $R_{x_i,\ep}$ again, it must do so through  one of the vertical portions  of the boundary of $R_{x_i,\ep}$.  On the other hand at that moment we have
\[
u_3=\ell(t_3)> T_\ep(t_3)\geq B_\ep(t_3).
\]
so  the point $p_3$ cannot lie on the vertical segment
\[
\bigl\{ (t_3, y); B_\ep(t_3)\leq y\leq T_\ep(t_3)\,\bigr\}
\]
which  the rightmost part of the boundary of $R_{x_i,\ep}$.

(b) Set for simplicity
\[
R_{x_i,\ep}^{\xi \geq c}:= \bigl\{ \xi\geq c\} \cap R_{x_i,\ep}.
\]
Observe that   the  set $I_c:=\xi\bigl( R_{x_i,\ep}^{\xi \geq c}\,\bigr)\subset \bR$ is connected because $ \xi(R_{x_i,\ep})$ is a closed interval and 
\[
I_c= \xi(R_{x_i,\ep})\cap [c,\infty).
\]
The map  $\xi$ defines a continuous surjection $\xi: R_{x_i,\ep}^{\xi \geq c}\ra I_c$ with contractible fibers. The Vietoris-Biegle theorem, \cite[Thm.15,  Chap.6, Sec.9]{Spa} implies  that  $R_{x_i,\ep}^{\xi \geq c}$ has the same  coh
\end{proof}

To complete the proof of (\ref{eq: rconv}) (and consequently also prove (\ref{eq: limb}), Lemma \ref{lemma: lim} and the overall Theorem\ref{th: main}) we need only prove Lemma \ref{lem: tech}

\smallskip

\noindent {\bf Proof of Lemma \ref{lem: tech}.}  Let $\alpha$ be the slope of $S$.  Suppose that $\xi(x_1,x_2)=\xi_1x_1+\xi_2x_2$.  Then 
\[\xi_1 + \xi_2 \alpha \neq 0
\]
 since $S$ is not parallel to $\{\xi = c\}$.  The function  $f_\ep$  is differentiable every where except at a finite number of points. Theorem \ref{thm: profileApprox} implies that there exists an $\ep_0 > 0$ such that for $0 < \ep < \ep_0$, 
\[
\xi_1 + f_\ep'(x) \xi_2
\]
 is either strictly positive or strictly negative for all $x$ in the domain of $f_\ep$ since $f'_\ep(x)$ must be close to $\alpha$ and $\xi_1 + \xi_2 \alpha \neq 0$. We set  $\xi|_{f_\ep}(x) = \xi_1 x + \xi_2 f(x)$. The notation indicates that $\xi|_{f_\ep}$ is the restriction of $\xi$ to the graph of $f_\ep$. Note that  $\xi|_{f_\ep}$ is Lipschitz and 
\[
\frac{d \xi|_{f_\ep}}{dx} (x) = \xi_1 + \xi_2 f_\ep'(x),
\]
for all but finitely many $x$.

Now suppose that  $\{\xi = c\}$ intersects the graph of $f_\ep$ at points $(x_0,y_0)$ and $(x_1,y_1)$.  Then note $\xi|_{f_\ep}$  is equal to $c$ for $x_0$ and $x_1$.  Thus we have, \cite[Prop. 11.12]{Tay}
\[
0=\xi|_{f_\ep}(x_1)-\xi_\ep(x_0)=\int_{x_0}^{x_1} \frac{d \xi|_{f_\ep}}{dx} dx = \int_{x_0}^{x_1}\bigl(\xi_1 + \xi_2 f_\ep'(x)\,\bigr) dx.
\]

But $\xi_1 + \xi_2 f_\ep'(x)$ is either strictly positive or strictly negative.  Therefore the integral
\[
\int_{x_0}^{x_1} \frac{d \xi|_{f_\ep}}{dx} dx
\]
cannot be 0, contradicting our assumption that $\{\xi = c\}$ intersects the graph of $f_\ep$ at two points.  This  completes the proof of  Lemma \ref{lem: tech}
\end{proof}


\newpage
\appendix
\section{The Farey Series and Holes in Pixelations}
\label{s: a}
\setcounter{equation}{0}
An unintuitive feature of pixelations is that they do not preserve homotopy type, even for small values of $\ep$.  The Figure \ref{fig: half} in Section \ref{s: 4} gives an example of a contractible set whose $\ep$-pixelations always contain a cycle.  In this appendix we wish to expand on the nature of these ``holes'' as well as show that an arbitrary number of such holes can appear.

We start by defining a convenient class of sets.

\begin{definition}
The \emph{angle} with slopes $(\alpha,\beta)$,  where $\beta > \alpha$ is  the set
\[
A(\beta, \alpha) := \bigl\{\, (x,y)\in\bR^2: \;\; x\geq 0,\;\; (y-\alpha x)(y-\beta x)=0\,\bigr\}. 
\]
\label{def: angle}
\end{definition}
Angles are convenient to work with in the context of pixelations since they have nice self similarity properties.  Note that
\[
P_\ep(A(\beta, \alpha)) =  \ep P_1(A(\beta,\alpha))
\]
where the last equality holds because $A(\beta,\alpha)$ does not change under rescalings centered at the origin.  This means that different $\ep$-pixelations are simply retractions or expansions of any other $\ep$-pixelation, and are thus topologically equivalent.  Therefore if a cycle appears in any $\ep$-pixelation of an angle, it will appear in every $\ep$-pixelation of an angle.  We refer to these false cycles in the $\ep$-pixelations as \emph{holes}.

Figure \ref{fig: 2hole} is an example of a pixelation of  an angle that has a hole, and in particular, the pixelation is not  contractible, while an angle plainly is. 

\begin{figure}[ht]
\centering{\includegraphics[height=2in,width=2in]{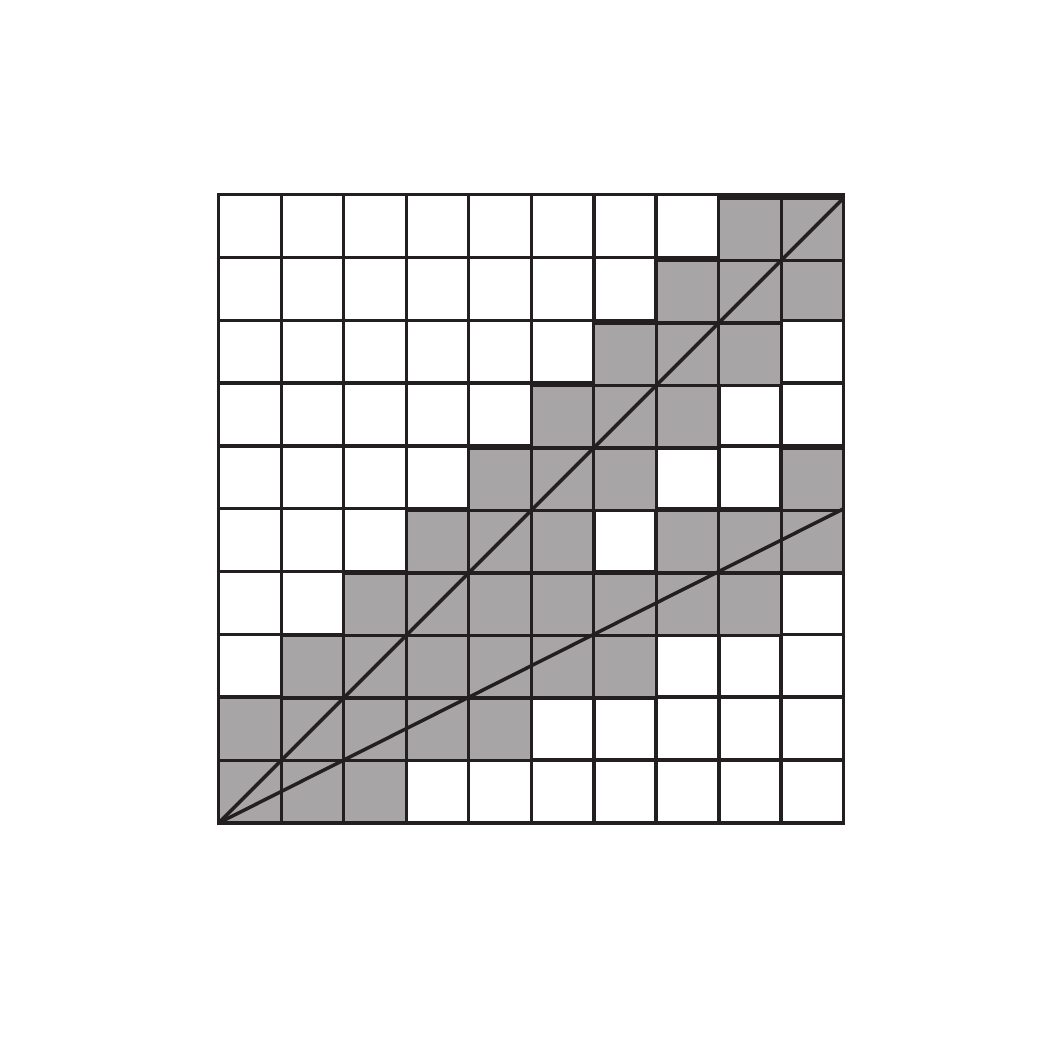}}
\caption{\sl The pixelation of the angle $A\left(1,\frac{1}{2}\right)$ contains 1 hole in the 7th column from the $y$-axis.  }
\label{fig: 2hole}
\end{figure}

When varying the slopes of the angle $A(\beta,\alpha)$ several things become apparent.  First, for a hole to appear either both slopes must be positive or both slopes must be negative. Second, the two lines must have fairly close slopes for holes to appear.  Finally, more holes will tend to appear when one of the slopes is close to $1$.

It is beyond the scope of this paper to classify the behavior of holes for general angles.  However, an easier situation occurs when the two slopes of the angle are adjacent members of a \emph{Farey series}.

The $n$-th Farey series $F_n$ is the \emph{increasing} finite sequence   consisting of the rational numbers between $0$ and $1$ which have denominator of $n$ or less when written in lowest terms.  For example,
\[
F_4 = { \frac{0}{1},\frac{1}{4},\frac{1}{3}, \frac{1}{2},\frac{2}{3},\frac{3}{4},\frac{1}{1}}
\]
An important property of the Farey sequences is that if $\frac{a}{b}$ and $\frac{c}{d}$ are consecutive  terms in a Farey series, then $bc - ad = 1$, \cite[Chap.III]{HW}. Using this property, we can prove the following fact.

\begin{proposition}
Suppose $\frac{a}{b}$ and $\frac{c}{d}$ are consecutive terms in a Farey series.  Denote by $g(a,b;c,d)$  the number of holes of the $\ep$-pixelation
\[
P_\ep \left(A\left(\frac{c}{d},\frac{a}{b}\right)\right),
\]
and by $X_{a,b,c,d}$ the  set of lattice points $(p,q)\in\bZ^2$ such that 
\begin{subequations}
\[
0<p\leq d,
\]
\[
a<q\leq \min(2a,b),
\]
\[
 bd+bc+b+d\leq 1+ pb+qd<2bd+bc+b+d.
 \]
\end{subequations}
Then
\[
g(a,b;c,d):= ad-\# X_{a,b,c,d}
\]
\label{prop: farey}
\end{proposition}

\begin{proof}
As noted in the earlier discussion, any $\ep$-pixelation of the angle is topologically equivalent to $P_1\left(A\left(\frac{c}{d},\frac{a}{b}\right)\right)$.  Therefore we need only  to consider the case $\ep=1$.  We set
\[
m_1:=\frac{c}{d},\;\;m_0:=\frac{a}{b}.
\]
For $j=0,1$ we denote  by $\ell_{j}$ the half-line
\[
y=m_{j}x,\;\;x\geq 0. 
\]
For each $i\in\bZ_{>0}$ we define (using the notations in Definition \ref{def: bu})
\begin{equation}
U_i := B_1\left(\ell_1 ,\frac{2i - 1}{2}\right)=\begin{cases}
\left\lfloor m_1(i-1)\right\rfloor, & m_1(i-1)\not\in \bZ\\
&\\
m_1(i-1)-1,& m_1(i-1)\in\bZ
\end{cases},
\label{eq: ui}
\end{equation}
\begin{equation}
L_i := T_1\left(\ell_0,\frac{2i - 1}{2}\right)=
\begin{cases}
\left\lceil m_0i\right\rceil,& m_0i\not\in\bZ\\&\\
m_0i+1,& m_0i\in\bZ
\end{cases}.
\label{eq: ti}
\end{equation}
In other words, $U_i$ indicates the lowest $y$-value of the $1$-pixelation of the  half-line $\ell_1$ within the $i$-th column $[i-1,i]\times \bR$, and $L_i$ indicates the highest $y$-value of the $1$-pixelation of the half-line $\ell_0$ in the same column.    Set $D_i:=U_i-L_i$ and observe that (\ref{eq: ui}) and (\ref{eq: ti}) imply
\begin{equation}
D_i<(m_1-m_0)i=\frac{i}{bd}.
\label{eq: di}
\end{equation}
Note also that
\begin{equation}
D_1=-2,\;\;D_2\in\{-1,-2\}.
\label{eq: init}
\end{equation}
More generally, (\ref{eq: di}) implies that
\begin{equation}
D_i\leq 0,\;\;\forall i=1,\dotsc, bd.
\label{eq: di1}
\end{equation}
The quantities $U_i$ and $L_i$  are useful because they allow us to easily state when there is a gap in a column or on the edge of a column.  The  $i$-th columns of the pixelations of the upper and lower half-lines overlap  if and only if
\[
U_i \le L_i
\]
and the two pixelations overlap on the right side of the $i$-th column if and only if
\[
U_i \le L_{i+1}
\]
A hole is a gap which is closed on both sides.  In particular, it is closed on the right side.  This means in each hole there is a column where the pixelations of $\ell_0$ and $\ell_1$ do not overlap on the interior of the column, but do overlap on the right side of the column.  In terms of $U_i$ and $L_i$ this is equivalent to the condition
\begin{equation}
L_i<  U_i \le L_{i+1}
\label{eq: ul condition 1}
\end{equation}
Note that this condition only occurs when $L_{i+1} > L_i$.  Since $\frac{a}{b}\leq 1$ we have 
\[
 L_{i+1} \le L_i + 1
\]
This means that the condition $L_{i+1} > L_i$ is equivalent to the condition $L_{i+1} = L_i + 1$.  

Suppose that $U_i - L_i = 2$.  Then it is impossible that $U_i \le L_{i+1}$ since that would imply that the lower line increased by at least the length of a pixel in the amount of time that the upper line increased by less than the length of a pixel.  Therefore for (\ref{eq: ul condition 1}) to occur, it must be true that $U_i - L_i = 1$.  Combining this fact with the above discussion about $L_i$ and $L_{i+1}$, we find that (\ref{eq: ul condition 1}) is equivalent to the condition
\begin{equation}
U_i - L_i = 1 \text{ and } L_{i+1} = L_i + 1.
\label{eq: ul condition 2}
\end{equation}

\medskip

Note that $L_i$ increases by $c$ every $d$ columns since $\ell_1$  has slope $\frac{c}{d}$.  Hence
\[
U_{i + d} = U_i + c,
\]
and similarly
\[
L_{i + b} = L_i + a.
\]
Therefore
\[
U_{i + bd} - L_{i + bd} = U_i + bc - (L_i + ad) = U_i - L_i + (bc - ad) = U_i - L_i + 1,
\]
so that
\begin{equation}
D_{i+bd}=D_i+1.
\label{eq: diff con 1}
\end{equation}
 We set
\[
k_1:=\min \{ k\in\bZ_{>0};\;\;D_k=1\,\bigr\}.
\]
We know that $k_1$ exists since (\ref{eq: diff con 1}) and (\ref{eq: init}) imply that $D_i$ takes on every integer value greater than or equal to $1$.  Note that  (\ref{eq: di}) implies that, if $k_1$ exists, then $k_1>bd$.  To finish the proof we will need a variety of facts about $k_1$. 

\begin{lemma}(a) $k_1:= (a+1)d+b+1 +bd=bd+ bc+b+d$.

\noindent (b)  For any $ i\in [k_1,k_1+bd)\cap\bZ$  we have
\begin{equation}
0\leq D_i < 2
\label{eq: u-l}
\end{equation}
(c) For any $i\geq k_1+bd$ we have $D_i>0$.
\label{lemma: tech}
\end{lemma}

\begin{proof}  (a)   The inequality (\ref{eq: di}) implies that  $k_1>bd$ so that 
\[
k_0:=k_1-bd>0
\]
and 
\[
D_{k_0}=0.
\]
Moreover, the minimality of $k_1$ coupled with (\ref{eq: diff con 1}) implies that
\[
k_0=\min\{i>0;\;\;D_i=0\,\}.
\]
If $D_i=0$  then   there  exists a  positive integer $\ell$ such that
\begin{subequations}
\begin{equation}
\ell-1\leq m_0 (i-1)<m_0i<\ell .
\label{eq: sl1}
\end{equation}
\begin{equation}
\ell <m_1(i-1)\leq \ell+1
\label{eq: sl2}
\end{equation}
\end{subequations}
When these conditions are satisfied we have $U_i=L_i=\ell$. The lattice point $( i-1,\ell)$ is  in the interior of the angle $A(m_1,m_0)$ spanned by the vectors
\[
\vec{u}_1=(b,a),\;\;\vec{u}_2=(d,c).
\]
Since $bc-ad=1$   we deduce that  the vectors $\vec{u}_1,\vec{u_2}$ form an integral basis of the lattice $\bZ^2$. We deduce that there exist two positive integers  $p,q$ such that
\[
(i-1,\ell) =p\vec{u}_1+q\vec{u_2},\;\;\mbox{i.e.,}\;\; i-1= pb +qd,\;\;\ell =pa+qc.
\]
Observe that
\[
m_0(i-1)= \frac{a}{b}( pb +qd) = ap + q\frac{ad}{b}= ap +q\frac{bc-1}{b}= ap+ qc -\frac{q}{b}=\ell-\frac{q}{b}.
\]
Hence (\ref{eq: sl1}) is satisfied if and only if  
\begin{equation}
a<q\leq b.
\label{eq: q}
\end{equation} On the other hand,
\[
m_1 (i-1)= \frac{c}{d}(pb+qd)=p\frac{bc}{d}+qc =p\frac{ad+1}{d} +qc= pa +qc+ \frac{p}{d}=\ell +\frac{p}{d}.
\]
Hence (\ref{eq: sl2}) is satisfied when
\begin{equation}
0<p \leq d. 
\label{eq: p}
\end{equation}
From the equality
\[
i=1+pb+qd
\]
we deduce that
\[
\min\{i>0; D_i=0\}= \min\{ 1+pb+qd;\;\;0<p\leq d,\;\; a<q\leq b\,\}=1+b +(a+1)d.
\]
This proves (a).

\medskip
\noindent (b) To prove the upper estimate in (\ref{eq: u-l})   we use (\ref{eq: di}). For $i <k_1+bd$ we have 
\[
D_i< 1+\frac{k_1}{bd}= 1+\frac{(a+1)d+b+1}{bd}=1+\frac{a+1}{b}+\frac{b+1}{bd}< 3.
\]
Thus $D_i\leq 2$ for all $i\in [k_1,k_1+bd)\cap\bZ$. If $D_i=2$ for some $i$ in this range then $D_{i-bd}=1$, contradicting the minimality of $k_1$.

To prove the lower estimate   part of (\ref{eq: u-l}) we recall that 
\[
k_0=k_1-bd=\min\{\, k>0;\;\;D_k=0\,\}.
\]
Let us observe that
\[
D_i\geq -1,\;\;\forall i\geq k_0.
\]
Indeed the inequality $D_i<-1$ takes place only if  $m_0i\geq\ell\in\bZ$ and $m_1(i-1)\leq\ell$. This implies 
\[
m_1(i-1)\leq m_0i \Longleftrightarrow \frac{i}{bd}\leq \frac{c}{d} \Longleftrightarrow i\leq bc =ad+1< k_0.
\]
Hence, for any $i\geq k_0+bd$ we have $D_i\geq 0$ proving part (b). 
\medskip
Part (c) follows from (b) and  (\ref{eq: diff con 1}).
\end{proof}

This lemma narrows down the possible locations of holes.  In particular, property (c) implies that holes can only occur in columns in the interval $[k_1, k_1 + bd)$.  Recall that for a hole to appear, we need $L_{i+1} = L_i + 1$ and $D_i = 1$.  In most cases when $L_i$ increases in this interval $D_i$ will also be equal to $1$.  However, if $U_{i+1} = U_i + 1$, $L_{i+1} = L_i + 1$ and $D_i = 0$, then no hole appears.  We want to count these instances where an increase in $L_i$ does not indicate the presence of a hole through the use of the set $X_{a,b,c,d}$.  We do this through a series of lemmas.

Define a  map
\[
\Phi:X_{a,b,c,d}\ra \bZ,\;\;(p,q)\mapsto 1+pb+qd.
\]
\begin{lemma} 
The map $\Phi$ is injective.
\end{lemma}

\begin{proof} 
Indeed if $\Phi(p,q)=\Phi(p',q')$ then $(p-p')b=(q'-q)d$. Since $b$ and $d$ are coprime we deduce that $(p-p')$ must be a multiple of $d$. Since $0<p,p'\leq d$ this happens if and only if $p=p'$.  It automatically follows that $q=q'$.
\end{proof}

We denote by $I_{a,b,c,d}$ the range of $\Phi$. From the definition of $X_{a,b,c,d}$ we  deduce that 
\[
I_{a,b,c,d}\subset [k_1,k_1+bd).
\]

\begin{lemma}  The following statements are equivalent.

\noindent (a) $i\in [k_1,k_1+bd)\cap\bZ$, $U_i=L_i$ and $L_{i+1}=L_i$.

\noindent (b) $i\in I_{a,b,c,d}$ 
\end{lemma}

\begin{proof}  Let  $i\in [k_1,k_1+bd)\cap\bZ$ such that  $L_{i+1}=L_i+1$ and  $D_i=0$.  Using the notation and the terminology employed in the proof of Lemma \ref{lemma: tech}(a) we deduce that the condition $D_i=0$  holds if and only if  that there exist positive  an integer $\ell$ such that (\ref{eq: sl1}) and (\ref{eq: sl2}) hold.  Moreover
\[
i= 1+pb +qd,
\]
where $q$ and $p$  are constrained by (\ref{eq: q}) and (\ref{eq: p}).  Once the condition $D_i=0$ is satisfied the condition $L_{i+1}=L_i+1=\ell+1$  is equivalent to 
\[
\ell\leq m_0(i+1)<\ell +1.
\]
\[
m_0(i+1)=m_0(i-1)+2m_0=\ell+\frac{2a-q}{b}\geq \ell.
\]
Hence $a<q\leq \min(b,2a)$, $0<p\leq d$, and $i=\Phi(p,q)$, $(p,q)\in X_{a,b,c,d}$, i.e., $i\in I_{a,b,c,d}$.
\end{proof}

Lemma \ref{lemma: tech}(c) shows that  the conditions of (\ref{eq: ul condition 2}) cannot be met in the $(k + bd)$-th column and beyond.  Therefore, the right edge of any hole must occur between the $k$-th column and the $(k + bd)$-th column.
Therefore the pixelation of the angle contains a hole for each time that $U_i-L_i=1$ and  $L_i$ increases between the $k$-th and $(k + bd)$-th column (including $k$ but not $(k + bd)$.) and   This is a total of $bd$ columns, and $L_i$ increases by $a$ every $b$ columns.  Hence $L_i$ will increase $ad$ times in this range.    On the other hand if $i\in I_{a,b,c,d}$ the $L_i$ increases  yet  $U_i=L_i$. Therefore $P_\ep\left(A\left(\frac{c}{d},\frac{a}{b}\right)\right)$ contains $ad-\# I_{a,b,c,d}$ holes.  The conclusion of the proposition follows from the injectivity of $\Phi$.
\end{proof}

\begin{remark} The proof   of  Proposition \ref{prop: farey}  gives us a simple  algorithm for computing  $g(a,b;c,d)$ that is easily implementable numerically. More precisely we have
\[
g(a,b;c,d)=\sum_{i=k_1}^{k_1+bd-1} (U_i-L_i)(L_{i+1}-L_i).\proofend
\]
\end{remark}

\begin{corollary}
\[
g(n,2n+1; 1,2)=2n.\proofend
\]
\label{cor: farey}
\end{corollary}

\begin{proof} In this case the set $X_{n,2n+1,1,2}$ is empty and the above equality follows immediately from Proposition \ref{prop: farey}.
\end{proof}

The above corollary shows that  the pixelation of even a very simple set can  have a complicated homotopy type.





\begin{ex} 

Let us consider  the following situation
\[
\frac{c}{d}=\frac{1}{2},\;\;\frac{a}{b}=\frac{2}{5}.
\]
Then
\[
U_{i+2}=U_i+1,\;\;L_{j+5}=L_j+2,\;\;D_{i+10}=D_i+1.
\]
Using (\ref{eq: ui}), (\ref{eq: ti}) and the above equalities we get
\[
\begin{tabular}{||c|r|r|r|r|r|r|r|r|r|r|r||}\hline\hline
$i$ & $1$ & $2$ & $3$ & $4$ & $5$& $6$ &  $7$ & $8$ & $9$ & $10$ & $11$\\ \hline\hline
 $U_i$ &   $-1$ &  $0$ & $0$ & $1$ &  $1$ & $2$ & $2$ & $3$ & $3$ & $4$&$4$\\ \hline
 $L_i$ &  $1$ &  $1$ & $2$ & $2$ & $3$ & $3$ & $4$ & $4$ & $4$ & $5$ & $5$ \\ \hline
 $D_i$ &  $-2$ & $-1$ & $-2$ & $-1$  & $-2$  & $-1$ & $-2$ & $-1$ & $-1$ & $-1$ & $-1$\\ \hline
\end{tabular}
\]
Using this table and  (\ref{eq: diff con 1}) we deduce that  $k=22$.  Next, using the above table and the equalities
 \[
U_{i+20}=U_i+10,\;\; L_{i+20}= L_i+8
\]
we obtain the following table
\[
\begin{tabular}{||c|c|c|c|c|c|c|c|c|c|c||}\hline\hline
$i$ & $22$ & $23$ & $24$ & $25$ & $26$& $27$ &  $28$ & $29$ & $30$ & $31$\\ \hline\hline
 $U_i$ &   $10$ &  $10$ & $11$ & $11$ &  $12$ & $12$ & $13$ & $13$ & $14$ & $14$\\ \hline
 $L_i$ &  $9$ &  $10$ & $10$ & $11$ & $11$ & $12$ & $12$ & $12$ & $13$ & $13$ \\ \hline
 $D_i$ &  $1$ & $0$ & $1$ & $0$  & $1$  & $1$ & $1$ & $1$ & $1$ & $1$\\ \hline
\end{tabular}
\]

In this case there are five holes.  Two holes are immediately apparent, since in the $23$rd and $25$th column a distance of $1$ is followed by a distance of $0$.  However note that in the $26$th column we have $D_{26}$ of $1$ and then $U_{26} = L_{27}$.  This indicates that a hole is formed merely because the pixelations of the upper and lower lines touch at a corner (compare to the case of Figure \ref{fig: 2hole}.)  Similar holes exist in the $27$th column and the $29$th column.  No holes exist past this point, since Lemma \ref{lemma: tech} implies that past this point $D_i > 0$.
\qed
\label{ex: ex}
\end{ex}

\section{Subanalytic Currents}
\label{s: b}
\setcounter{equation}{0}
In this appendix  we gather  without proofs a few facts  about  the subanalytic currents introduced by R. Hardt in \cite{Hardt, Hardt2}.  Our  terminology    concerning currents closely  follows that of Federer \cite{Feder} (see also the more accessible \cite{Kra, Mor}).  However, we changed some  notations to better resemble notations used in  algebraic topology. First we need to define  the subanalytic sets.

An  \emph{$\bR$-structure}
is a collection $\eS=\bigl\{\, \eS^n\,\bigr\}_{n\geq 1}$,
$\eS^n\subset \eP(\bR^n)$, with the following properties.

\begin{description}

\item[${\bf E}_1.$]   $\eS^n$ contains all the real algebraic subvarieties of $\bR^n$, i.e., the zero sets of  finite collections of polynomial in $n$ real variables.

\item[${\bf E}_2.$]  For every   linear map $L:\bR^n\ra \bR$, the half-plane $\{\vec{x}\in \bR^n;\;\;L(x)\geq 0\}$  belongs to $\eS^n$.

\item[${\bf P}_1.$] For every $n\geq 1$, the family $\eS^n$ is closed under  boolean operations, $\cup$, $\cap$ and complement.

\item[${\bf P}_2.$]  If $A\in \eS^m$, and $B\in \eS^n$, then $A\times B\in \eS^{m+n}$.

\item[${\bf P}_3.$]  If $A\in \eS^m$, and $T:\bR^m\ra \bR^n$ is an affine map, then $T(A)\in \eS^n$.

\end{description}

\begin{ex}[Semialgebraic sets]    Denote by $\bR_{alg}$ the collection of real semialgebraic sets.  Thus,  $A\in \bR^n_{alg}$ if and only if  $A$  is a finite  union of sets,  each of which is described by finitely many polynomial equalities and inequalities. The celebrated Tarski-Seidenberg theorem states that $\eS_{alg}$ is a structure.\qed
\end{ex}

Let $\eS$ be an $\bR$-structure. Then a set that belongs to one of
the  $\eS^n$-s is called   $\eS$-\emph{definable}. If $A, B$ are
$\eS$-definable, then a function $f: A\ra B$ is called
$\eS$-\emph{definable} if its graph $ \Gamma_f :=\bigl\{\, (a,b)\in A\times B;\;\;b=f(a)\,\bigr\}$ is $\eS$-definable.

Given a collection $\eA=(\eA_n)_{n\geq
1}$, $\eA_n\subset\eP(\bR^n)$, we can form a new structure
$\eS(\eA)$, which is the smallest structure containing   $\eS$ and
the sets in $\eA_n$. We say that $\eS(\eA)$ is obtained  from  $\eS$
by \emph{adjoining the collection $\eA$}.

\begin{definition} An $\bR$-structure is called \emph{$o$-minimal} (order minimal) or \emph{tame} if  it satisfies the property

\begin{description}
\item[{\bf T}]    Any set $A\in  \eS^1$ is a \emph{finite} union of open intervals $(a,b)$, $-\infty \leq a <b\leq \infty$, and singletons $\{r\}$. \qed
\end{description}

\end{definition}

\begin{ex} (a) (Tarski-Seidenberg)  The collection  $\bR_{alg}$ of real semialgebraic sets  is a tame structure.

\noindent (b) (A. Gabrielov, R. Hardt, H. Hironaka, \cite{Gab, Hardt2, Hiro})   A \emph{restricted} real
analytic function is a  function $f:\bR^n\ra \bR$ with the property
that there exists a real analytic function $\tilde{f}$ defined in an
open  neighborhood $U$ of the cube $C_n:=[-1,1]^n$ such that
\[
f(x)=\begin{cases}
\tilde{f}(x) & x\in C_n\\
0 & x\in \bR^n\setminus C_n.
\end{cases}
\]
we denote by $\bR_{an}$ the structure obtained from $\eS_{alg}$ by
adjoining the  graphs of all the restricted real analytic functions.
Then $\bR_{an}$ is a tame structure, and the $\bR_{an}$-definable
sets are called \emph{(globally) subanalytic sets}. \qed
\end{ex}

The  definable sets  and function of a tame structure have  rather remarkable \emph{tame} behavior which prohibits  many pathologies.  It is perhaps   instructive to give an example of function which is not definable in any tame structure. For example, the function $x\mapsto \sin x$ is not definable in a tame structure because the intersection of its graph with the horizontal axis is the  countable set $\pi\bZ$  which  violates  the tameness condition ${\bf T}$.

 We list below some of the nice properties of the sets and function definable  in a  fixed tame structure  $\eS$.  Their proofs can be found in \cite{Co, Dr}. We will interchangeably refer to  sets or functions definable in  a given tame structure $\eS$ as \emph{definable}, \emph{constructible} or  \emph{tame}.

\smallskip

\noindent \ding{227}   (\emph{Piecewise  smoothness of  tame
functions.})  Suppose $A$ is a definable  set, $p$ is a
positive integer, and $f: A\ra \bR$ is a definable function. Then
$A$ can be partitioned into finitely many   definable sets
$S_1,\dotsc, S_k$,     such that each  $S_i$ is a $C^p$-manifold,
and each of the restrictions $f|_{S_i}$ is a $C^p$-function.

\noindent  \ding{227} (\emph{Triangulability.})  For every   compact
definable set $A$, and any finite collection of definable  subsets
$\{S_1,\dotsc, S_k\}$, there exists  a compact simplicial complex
$K$, and a  definable homeomorphism $\Phi: |K|\ra A$ such that  all the sets $\Phi^{-1}(S_i)$ are unions of  relative
interiors of faces of $K$.

\noindent \ding{227} (\emph{Dimension.})  The  dimension of a
definable  set $A\subset \bR^n$ is the supremum over all the
nonnegative integers $d$ such that there exists a $C^1$  submanifold
of $\bR^n$ of dimension $d$ contained in $A$.  Then $\dim A
<\infty$, and $\dim (\cl(A)\setminus A) <\dim A$.

\noindent \ding{227}(\emph{Definable selection.}) Any tame map  $f: A\ra B$ (not necessarily continuous) admits a tame section, i.e.,  a tame map $s: B\ra A$ such that $s(b)\in f^{-1}(b)$, $\forall b\in B$.

\noindent \ding{227} (\emph{Local triviality of tame maps}) If $f: A\ra B$ is a tame continuous map, then there exists a  tame triangulation of $B$ such that over the relative interior of any face the map  $f$ is  a locally trivial fibration.

\noindent \ding{227} (\emph{The $o$-minimal Euler characteristic}) There exists a function $\chio:\eS\ra \bZ$ uniquely characterized  by the following  conditions.

\begin{itemize}

\item $\chio(X\cup Y)=\chio(X)+\chio(Y)-\chio(X\cap Y)$, $\forall X,Y\in \eS$.

\item  If $X\in \eS$ is compact, then $\chio(X)$ is the usual Euler characteristic of $X$.

\end{itemize}

\noindent \ding{227} (\emph{Finite volume.}) Any compact
$k$-dimensional tame set has finite $k$-dimensional Hausdorff
measure $\eH^k$. 

\noindent \ding{227} (\emph{Uniform volume bounds.}) If $f:A\ra B$ is a proper, continuous  definable map such that all the fibers  have dimensions $\leq k$, then   there exists $C>0$ such that
\[
 \eH^k\bigl(\, f^{-1}(b)\,\bigr) <C,\;\;\forall b\in B.
\]
\qed

Suppose $X$ is a $C^2$, oriented Riemann manifold of dimension $n$.
We denote by $\Omega_k(X)$ the space of $k$-dimensional currents in
$X$, i.e., the topological dual space of the space
$\Omega^k_{cpt}(X)$ of smooth, compactly supported $k$-forms on $X$.
We will denote by
\[
\lan\bullet,\bullet\ran: \Omega^k_{cpt}(X)\times \Omega_k(X)\ra \bR
\]
the natural pairing.  The boundary of a current $T\in \Omega_k(X)$
is  the $(k-1)$-current defined via the Stokes formula
\[
\lan \alpha, \pa T\ran :=\lan d\alpha, T\ran,\;\;\forall \alpha\in
\Omega^{k-1}_{cpt}(X).
\]
For every  $\alpha\in \Omega^k (X)$,  $T\in \Omega_m(X)$,  $k\leq m$
define $\alpha \cap T\in \Omega_{m-k}(X)$ by
\[
\lan\beta , \alpha \cap T  \ran =\lan \alpha\wedge \beta, T
\ran,\;\;\forall \beta\in \Omega^{n-m+k}_{cpt}(X).
\]
We have
\[
\lan \beta, \pa (\alpha \cap T)\ran = \lan \,d\beta, (\alpha\cap
T),\ran = \lan \alpha\wedge d\beta, T\ran
\]
\[
=(-1)^k \lan d(\alpha\wedge \beta) -d\alpha\wedge \beta, T\ran =
(-1)^k \lan \beta, \alpha \cap\pa T\ran +(-1)^{k+1} \lan \beta,
d\alpha \cap T\ran
\]
which yields the \emph{homotopy formula}
\begin{equation}
\pa (\alpha\cap T)= (-1)^{\deg \alpha} \bigl(\, \alpha \cap \pa
T-(d\alpha) \cap T\,\bigr). \label{eq: homotop}
\end{equation}

We say that a set  $S\subset \bR^n$ is \emph{locally subanalytic}
if for any $p\in \bR^n$ we can find an open ball $B$ centered at
$p$ such that $B\cap S$ is globally subanalytic.

\begin{remark} There is a rather subtle distinction between globally subanalytic and locally subanalytic sets. For example, the graph of the function $y=\sin(x)$ is a locally subanalytic subset of $\bR^2$, but it is not a globally subanalytic  set. Note that a compact, locally subanalytic set is globally subanalytic.\qed
\end{remark}

If $S\subset \bR^n$ is an orientable, locally subanalytic, $C^1$
submanifold of $\bR^n$ of dimension $k$,  then any orientation
$\ori_S$ on $S$ determines a  $k$-dimensional current $[S,\ori_S]$
via the equality
\[
\lan \alpha, [S, \ori_S]\ran:=\int_S \alpha,\;\;\forall \alpha\in
\Omega^k_{cpt}(\bR^n).
\]
The integral in the right-hand side is well defined because any
bounded, $k$-dimensional  globally subanalytic set has finite
$k$-dimensional Hausdorff  measure.  For any open, locally
subanalytic   subset $U\subset \bR^n$ we  denote by $[S,\ori_S]\cap
U$  the   current $[S\cap U, \ori_S]$.

For any  locally subanalytic subset $X\subset \bR^n$ we denote by
$\eC_k(X)$ the    Abelian subgroup of $\Omega_k(\bR^n)$
generated  by currents of the form $[S,\ori_S]$,  as above, where
$\cl(S)\subset X$. The above operation $[S,\ori_S]\cap U$, $U$ open
subanalytic extends to a morphism of  Abelian groups
\[
\eC_k(X)\ni T\mapsto T\cap U\in\eC_k(X\cap U).
\]
We will refer to the elements of $\eC_k(X)$ as \emph{subanalytic
(integral) $k$-chains} in $X$.

Given compact subanalytic sets $A\subset X\subset \bR^n$ we set
\[
\eZ_k(X,A)=\bigl\{ T\in \eC_k(\bR^n);\;\;\supp T\subset X,\;\;\supp
\pa T\subset A\,\bigr\},
\]
and
\[
\eB_k(X, A)=\bigl\{ \pa T + S;\;\;T\in \eZ_{k+1}(X,A)), \;\;S\in
\eZ_k(A)\,\bigr\}.
\]
We set
\[
\eH_k(X,A):=\eZ_k(X,A)/\eB_k(X,A).
\]
R. Hardt has proved in \cite{Hardt2} that the assignment
\[
(X,A)\longmapsto \eH_\bullet(X,A)
\]
satisfies the Eilenberg-Steenrod   homology axioms with
$\bZ$-coefficients. This implies   that $\eH_\bullet(X,A)$
is naturally isomorphic  with the integral homology  of the pair.

To describe the intersection theory of subanalytic chains  we need
to recall a  fundamental result of R. Hardt, \cite[Theorem
4.3]{Hardt}.  Suppose $E_0, E_1$ are two oriented  real Euclidean
spaces of dimensions $n_0$ and respectively $n_1$,  $f:E_0\ra E_1$
is a real analytic map, and $T\in \eC_{n_0-c}(E_0)$ a subanalytic
current of codimension $c$.  If $y$ is a regular value of $f$,  then
the fiber $f^{-1}(y)$  is  a submanifold  equipped with a natural
coorientation and thus defines a subanalytic  current $[f^{-1}(y)]$
in $E_0$  of codimension $n_1$, i.e., $[f^{-1}(y)]]\in
\eC_{d_0-d_1}(E_0)$. We would like to define the intersection of $T$
and $[f^{-1}(y)]$  as a subanalytic current  $\lan T, f, y\ran\in \eC_{n_0-c-n_1}(E_0)$.     It  turns out that  this
is possibly quite often, even in cases when $y$ is  not a regular
value.

\begin{theorem}[Slicing Theorem]  Let $E_0$, $E_1$, $T$ and $f$ be  as above, denote by $dV_{E_1}$ the Euclidean volume form on $E_1$,   by $\bom_{n_1}$ the volume of the unit ball in $E_1$, and set
\[
\eR_f(T):=\bigl\{ y\in E_1;\;\codim (\supp T )\cap f^{-1}(y) \geq
c+ n_1,\;\codim (\supp \pa T )\cap f^{-1}(y)\geq
c+n_1+1\,\bigr\}.
\]
For  every $\ve>0$ and $y\in E_1$ we define $T\bullet_\ve
f^{-1}(y)\in \Omega_{n_0-c-n_1}(E_0)$ by
\[
\bigl\lan\, \alpha, T\bullet_\ve f^{-1}(y)\,\bigr\ran
:=\frac{1}{\bom_{n_1}\ve^{n_1}}\bigl\lan\, (f^*dV_{E_1})\wedge\alpha ,
T\cap \bigl(\,f^{-1}(B_\ve(y)\,\bigr)\,\bigr\ran,\;\;\forall
\alpha\in \Omega^{n_0-c-n_1}_{cpt}(E_0).
\]
Then  for every $y\in \eR_f(T)$, the currents $T\bullet_\ve
f^{-1}(y)$ converge weakly as $\ve>0$   to a  subanalytic  current
$\lan T, f,y\ran\in \eC_{n_0-c-n_1}(E_0)$ called  the
\emph{$f$-slice} of $T$ over $y$. Moreover,   the map
\[
\eR_f\ni y\mapsto \lan T, f, y\ran\in \eC_{d_0-c-d_1}(\bR^n)
\]
is continuous in the   locally   flat  topology.\qed
\label{th: slice}
\end{theorem}

\section{Normal cycles of subanalytic  sets}
\label{s: c}
\setcounter{equation}{0}
We follow the  presentation in \cite{LC}. Let $\bsV$ be an oriented real Euclidean  vector space of dimension  $n$. Denote by $\bsV\dual$ its  dual,  and by $\Sigma\dual$ the unit sphere in $\bsV\dual$.  We identify  the cotangent bundle $T^*\bsV$ with the product $\bsV\dual\times \bsV$.    We have two canonical projections
 \[
 p: \bsV\dual\times\bsV\ra \bsV\dual,\;\;\pi: \bsV\dual\times\bsV\ra \bsV.
 \]
Let $\lan -,-\ran: \bsV\dual\times \bsV\ra   \bR$ denote   the canonical pairing
 \[
 \bsV\dual\times \bsV\ni (\xi, x)\mapsto  \lan\xi, x\ran:= \xi(x)\in\bR.
 \]
The Euclidean  metric $(-,-)$ on $\bsV$  defines isometries (the classical lowering/raising the indices operations)
 \[
 \bsV\ni x\mapsto x_\dag\in \bsV\dual,\;\;\bsV\dual\ni \xi\mapsto\xi^\dag\in\bsV,
 \]
 \[
 \lan x_\dag,  y\ran =(x,y),\;\; \lan \xi, y\ran =(\xi^\dag, y),\;\;\forall x,y\in\bsV,\;\;\xi\in\bsV\dual.
 \]
  Let $\alpha\in\Omega^1(T^*\bsV)$ denote the canonical $1$-form on the cotangent bundle.  More explicitly, if $x^1,\dotsc, x^n$
 are Euclidean  coordinates on $\bsV$, and $\xi_1,\dotsc, \xi_n$ denote the induced Euclidean coordinates on $\bsV\dual$, then
 \[
 \alpha=\sum_i \xi_i dx^i.
 \]
We denote by $\omega\in \Omega^2(T^*\bsV)$ the associated symplectic form
\[
\omega=-d\alpha=\sum_i dx^i\wedge d\xi_i.
\]

For any closed subanalytic subset $X\subset \bsV\dual\times \bsV$ we denote  by $\eC_k(X)$   the Abelian group of  subanalytic, $k$-dimensional currents  with support on $X$; see  Appendix \ref{s: b}. If $S\in \eC_k(\Sigma\dual\times \bsV)$, and $\xi\in \Sigma\dual$, we denote by $S_\xi$ the $p$-slice of $S$ over $\xi$,
 \[
 S_\xi:=\lan S, p, \xi\ran\in \eC_{k-\dim\Sigma\dual}(\Sigma\dual\times \bsV)
 \]
 As explained in Appendix \ref{s: b}, the slice  $S_\xi$  exists for all $\xi$ outside a codimension $1$ subanalytic subset of $\Sigma\dual$ and it is supported on the fiber $p^{-1}(\xi)\cap\supp S$.    More precisely, $S_\xi$ is well defined if the  fiber $p^{-1}(\xi)\cap \supp S$ has the expected  dimension, $\dim S-\dim \Sigma\dual$.
  
   If $S$   is  the current of integration along an oriented $k$-dimensional  manifold, then for generic $\xi$ the slice $S_\xi$ is the current of integration along the fiber $S\cap p^{-1}(\xi)$ equipped with a canonical orientation.  In general, the slice gives a precise meaning   as a current to the intersection of $S$ with the fiber $p^{-1}(\xi)$, provided that this intersection has the ``correct'' dimension.

If $X\subset \bsV$ is a compact subanalytic set,  $\xi\in \Sigma\dual$,   and $x\in X$ we set
\[
X_{\xi>\xi(x)}:=\bigl\{  y\in X;\;\;\xi(y)>\xi(x)\,\bigr\},\;\; i_X(\xi, x):=1-\lim_{r\searrow 0}\chi\bigl(\, B_r(x)\cap X_{\xi>\xi(x)}\,\bigr),
\]
where $\chi$ denotes the Euler characteristic of a topological space. If $x\in \bsV\setminus X$ we set $i_X(\xi,x)=0$.  For generic $\xi\in\Sigma\dual$,  we have $i_X(\xi,x)=0$, for all but finitely many points $x\in  X$.

We have the following existence and uniqueness result due to   J. Fu,  \cite[Tm. 3.2]{Fu2}.

\begin{theorem} Let $X$  be a compact subanalytic subset   of $\bsV$. Then there exists exactly  one subanalytic  current $\bsN\in \eC_{n-1}(\Sigma\dual\times\bsV)$  satisfying the following conditions.

\begin{enumerate}

\item The current $\bsN$ is a cycle, i.e., $\pa N=0$.

\item  The current $\bsN$ has compact support.

\item The current $\bsN$ is Legendrian, i.e.,
\[
\lan \alpha\cup \eta,  \bsN\ran=0,\;;\forall \eta\in \Omega^{n-2}\bigl(\, \Sigma\dual\times\bsV\,\bigr).
\]

\item   For any smooth function   $\vfi\in C^\infty(\Sigma\dual\times\bsV)$   we have
\begin{equation}
\lan \vfi dV_{\Sigma\dual}, N\ran=\int_{\Sigma\dual} \Bigl(\sum_{x\in X}  \vfi(\xi,x) i_X(\xi,x)\Bigr)\,dV_{\Sigma^\dual}.
\label{eq: morse-sl}
\end{equation}
\end{enumerate}
\label{th: fu1-norm}
\end{theorem}

\begin{remark}  Using \cite[Thm. 4.3.2.(1)]{Feder} we deduce that the  equality   (iv) is equivalent with the  condition
\begin{equation*}
\bsN_\xi= \sum_{x\in X} i_X(\xi,x)\delta_{(\xi,x)},\;\;\mbox{for almost all $\xi\in\Sigma\dual$},
\tag{$\ast$}
\label{tag: ast}
\end{equation*}
where $\delta_{(\xi,x)}$ denotes the  canonical $0$-dimensional current determined by the point $(\xi,x)$. The points $x$ for  which $i(x,\xi)\neq 0$ should be viewed as  critical points of the function $-\xi: X\ra \bR$; see \cite[\S 5.4]{KaSch}. Thus, the slice  $N_\xi$   records  both the collection of critical points of $-\xi|_X$ and  their Morse indices.    \qed
\end{remark}

\begin{definition} The cycle $\bsN$  whose existence and uniqueness is postulated  by Theorem \ref{th: fu1-norm} is called the \emph{normal cycle} of the  compact subanalytic set $X$ and it is denoted by $\bsN^X$. Using the  metric identification between the unit sphere in $\bsV\dual$ and the unit sphere $S(\bsV)\subset \bsV$ we will think of $\bsN^X$ as a $(n-1)$-dimensional cycle on $S(\bsV)\times \bsV$.\qed
\end{definition}

\begin{ex} (a) If  $X$ is a  compact smooth  submanifold of $\bsV$, then  $\bsN^X$  can be identified with the integration current defined by the total space of the unit sphere bundle associated to the normal bundle of the embedding $X\hra \bsV$.

\smallskip

\noindent (b) If $X$ is a bounded domain in $\bsV$ with sufficiently regular boundary $\pa X$, then  we   have a unit outer normal vector field 
\[
\bn:  \pa X\ra  S(\bsV)
\]
and the   normal cycle  $\bsN^X$  is the integration current defined by  the graph of the above map.

\smallskip

\noindent(c)  If $f: \bsV\ra [0,\infty)$ is a proper, $C^2$, subanalytic function, then the normal cycle  of  of the sublevel set $\{f\leq \ve\}$ converges as $\ve\ra 0$ to the normal cycle  of the level set $\{f=0\}$.

\smallskip  

\noindent (d) For a very intuitive description of the normal cycle of a  compact $PL$ subset we refer to \cite{CMS, Win}.
\qed
\label{ex: normal}
\end{ex}

One can show that the map that associates  to a compact subanalytic set  its normal cycle   is injective; see \cite{Ber,LC}. This means that   a compact  subanalytic set is completely determined by its normal cycle.       The actual  reconstruction process is based on  a ``motivic'' Radon transform.

One very useful property of normal  cycles is the   \emph{inclusion-exclusion property}, \cite[Thm. 4.2]{Fu2}, \cite[\S 4]{LC}
\begin{equation}
\bsN^{X\cup Y}=\bsN^X +\bsN^Y-\bsN^{X\cap Y},
\label{eq: incl-excl}
\end{equation}
for any compact  subanalytic sets $X$, $Y$.    

The notion of    normal cycle is closely related to the concept of \emph{curvature measure}.    Let us  observe that  the cotangent bundle  is   equipped with several canonical $SO(\bsV)$-invariant $n$-forms.  To describe them  fix an oriented    orthonormal basis  $(\bse_1,\dotsc,\bse_n)$ of $\bsV$.  Denote by $(x^1,\dotsc, x^n)$ the associated Euclidean coordinates, and by $\xi_1,\dotsc, x_n$ the dual coordinates  on  $\bsV\dual$. For $t>0$ we define
\[
\Bom_{\bsV,t}= (dx^1+d\xi_1)\wedge \cdots \wedge (dx^n +t d\xi_n)\in \Omega^n(\bsV\dual \times \bsV).
\]
Set
\[
\rho:=\sqrt{\xi_1^2+\cdots +\xi_n^2}
\]
and denote by $\pa_\rho$ the radial vector field on $(\bsV\dual\setminus 0)\times \bsV$
\[
\pa_\rho=\frac{1}{\rho} \sum_j \xi_j\pa_{\xi_j}.
\]
We set
\[
s_k=\pa_\rho\xi_k=\pa_\rho\inpr ds_k,\;\;k=1,\dotsc, n.
\]
On $(\bsV\dual\setminus 0)\times \bsV$ we have
\[
\frac{1}{t}\pa_\rho\inpr \Bom_{\bsV,t}=\pa_\rho\inpr\bigwedge_{j=1}^n (dx^j+ td\xi_j)=\pa_\rho\inpr\bigwedge_{j=1}^n \bigl(\,dx^j+ td(\rho s_j)\,\bigr)
\]
\[
=\sum_k (-1)^{k-1} s_k \bigwedge_{j\neq k} (dx^j+td\xi_j) =\sum_k (-1)^{k-1} s_k \bigwedge_{j\neq k}(dx^j +t\rho ds_j +ts_jd\rho )
\]
We denote by $\eta_{\bsV,t}$ the restriction of  $\frac{1}{t}\pa_\rho\inpr \Bom_{\bsV,t}$ to $\Sigma\dual\times \bsV$.  Along this manifold we have $\rho=1$ and we deduce
\[
\eta_{\bsV,t}=\sum_k (-1)^{k-1} s_k \bigwedge_{j\neq k}\bigl(\, dx^j +tds_j \bigr)=:\sum_{j=0}^{n-1}t^j\eta_{n-1-j}
\]
We denote by $\bom_k$ the volume of the unit $k$-dimensional ball, and by $\bsi_{k-1}$ the area of its boundary. Then
\begin{equation}
\bom_k=\frac{\pi^{\frac{k}{2}}}{\Gamma\bigl(\frac{k}{2}+1\bigr)},\;\;\bsi_{k-1}= k\bom_k.
\label{eq: vol}
\end{equation}
We  set
\[
\hat{\eta}_k:=\frac{1}{\bsi_{n-1-k}}\eta_k,\;\;k=0,\dotsc, n-1
\]
Using the metric identification between $\bsV$ and $\bsV\dual$ we will think of the forms $\hat{\eta}_k$ as forms on $S(\bsV)\times \bsV$, the unit sphere bundle associated to the tangent bundle of $\bsV$.

If $X$ is  compact  subanalytic set, then the quantities 
\[
\lambda_k(X):=\lan \hat{\eta}_k, \bsN^X\ran
\]
are called the \emph{curvature measures} of $X$.  We have the celebrated    \emph{Weyl tube formula}, \cite{Weyl}.

\begin{theorem}  If  $X\subset \bsV$ is a  smooth compact submanifold  of $\bsV$ of dimension $m$ and $ T_r(X)$ denotes the tube of radius $r$ around $X$,
\[
T_r(X):=\bigl\{\, v\in\bsV;\;\;\dist(v,X)\leq r\,\bigr\}
\]
then
\[
\begin{split}
{\rm vol}\, (T_r(X)) &= \sum_{k=0}^{m}  \lambda_k(X) \bom_{n-k} r^{n-k}\\
& = \lambda_{m}(X) \bom_{n-m} r^{n-m}+\lambda_{m-1}(X) \bom_{n-m+1}r^{n-m+1}+\cdots.
\end{split}
\]
Moreover
\[
\lambda_m(X)= {\rm vol}\,(X),\;\;\lambda_0(X)=\chi(X).
\]
In general, $\lambda_{m-2k-1}(X)=0$, while $\lambda_{m-2k}(X)$ can be expressed as the integral with respect to the Riemannian volume on $X$ of an universal polynomial of degree $k$ in the curvature of the induced metric on $X$.\qed
\label{th: weyl}
\end{theorem}

In general, for any  compact  subanalytic set $X$ we have
\[
\lambda_0(X)=\chi(X).
\]
\begin{ex} Suppose $\dim \bsV=2$.  Using polar coordinates  $\rho,\theta$ in the plane $\bsV\dual$ we have
\[
\xi_1= \rho\cos\theta,\;\;\xi_2=\rho\sin\theta
\]
\[
d\xi_1= \cos\theta d\rho-\rho\sin\theta d\theta,\;\;d\xi_2= \sin\theta d\rho+\rho\cos\theta d\theta,
\]
\[
\Bom_{\bsV,t}= (dx^1+ t\cos\theta d\rho-t\rho\sin\theta d\theta)\wedge (dx^2+  t\sin\theta d\rho+t\rho\cos\theta d\theta)
\]
\[
=dx^1\wedge dx^2+ t(\cos\theta d\rho-\rho\sin\theta d\theta)\wedge dx^2+ tdx^1\wedge (\sin\theta d\rho+\rho\cos\theta d\theta) + t^2 \rho d\rho \wedge d\theta,
\]
\[
\frac{1}{t}\pa_\rho \Bom_{\bsV,t}= \cos\theta dx^2-\sin\theta dx^1  + t\rho d\theta,
\]
\[
\eta_{\bsV,t}:= \cos\theta dx^2-\sin\theta dx^1+ td\theta,
\]
\[
\eta_0=d\theta,\;\;  \hat{\eta}_0=\frac{1}{\bsi_1}d\theta=\frac{1}{2\pi}d\theta,
\]
\[
\eta_1=\cos\theta dx^2-\sin\theta dx^1,\;\;\hat{\eta}_1= \frac{1}{\bsi_0}\eta_1= \frac{1}{2}\bigl(\, \cos\theta dx^2-\sin\theta dx^1\,\bigr).
\]
Suppose now that  $X$ is a compact smooth domain with connected boundary.  We orient the boundary using the  outer-normal-first convention and we fix an arc length parametrization   of the boundary compatible  with this orientation
\[
x_1=x_1(s),\;\;x_2=x_2(s),\;\;s\in [0,L],
\]
where $L$ is the length of the boundary.  Denote by $\tau(s)$ the unit tangent vector
\[
\btau(s)=\bigl(\,x_1'(s),x_2'(s)\,\bigr)
\]
and by $\bn(s)$ the unit outer normal. The frame $(\bn(s),\btau(s))$    and, using polar coordinates, we  can write
\[
\bn(s)=\bigl(\cos \theta(s), \sin\theta(s)\,\bigr).
\]
The normal cycle  of $X$ is the current of integration given by the oriented, closed  path
\[
[0,L]\ni s\stackrel{\Phi}{\longmapsto} \bigl(\cos \theta(s), \sin\theta(s);\; x_1(s), x_2(s)\,\bigr)\in S(\bsV)\times \bsV.
\]
Since frame $(\bn(s),\btau(s))$ is positively oriented we deduce that $\btau9s)$ is obtained from $\bn(s)$ via a counterclockwise rotation by $\frac{\pi}{2}$.   This implies that
\[
x_1'(s) =-\sin \theta(s),x_2'(s)=\cos \theta(s).
\]
Hence
\[
\lan \hat{\eta}_1, \bsN^X\ran =\int_0^L \Phi^*\hat{\eta}_1=\frac{1}{2}\int_0^L \bigl(\, (x_1'(s)^2+(x_2'(s))^2\,\bigr)=\frac{1}{2}L=\frac{1}{2}\times \mbox{perimeter of $X$}.
\]
\qed
\end{ex}

\section{The Approximation Algorithm}
\label{s: d}
\setcounter{equation}{0}
In this section we describe Algorithm \ref{alg: process} loosely in terms of a computer program.  The input for this algorithm is a  pixelation $P_\ep(S)$, where $S$ is a compact set.  Since $P_\ep(S)$ is compact, it is contained within some $m\times m$ rectangle of $\ep$-pixels in the plane.  Associate to each pixel a value in $[1,m] \cap \bZ \times [1,m] \cap \bZ$ indexed as a matrix (see Figure \ref{fig: index}).  Using this labeling we can encode $P_\ep(S)$ as a $m\times m$ matrix $A$ where the entry $a_{ij} = 1$ if and only if the pixel associated to $(i,j)$ lies in $P_\ep(S)$, and is $0$ otherwise.  Throughout the algorithm we will refer to the matrix $A$ constructed in this manner.  The notation $C[i,j]$ will indicate the center of the pixel (in $P_\ep(S)$) corresponding to the entry $a_{ij}$.

\begin{figure}[h]
\centering{\includegraphics[height=2.25in,width=2.25in]{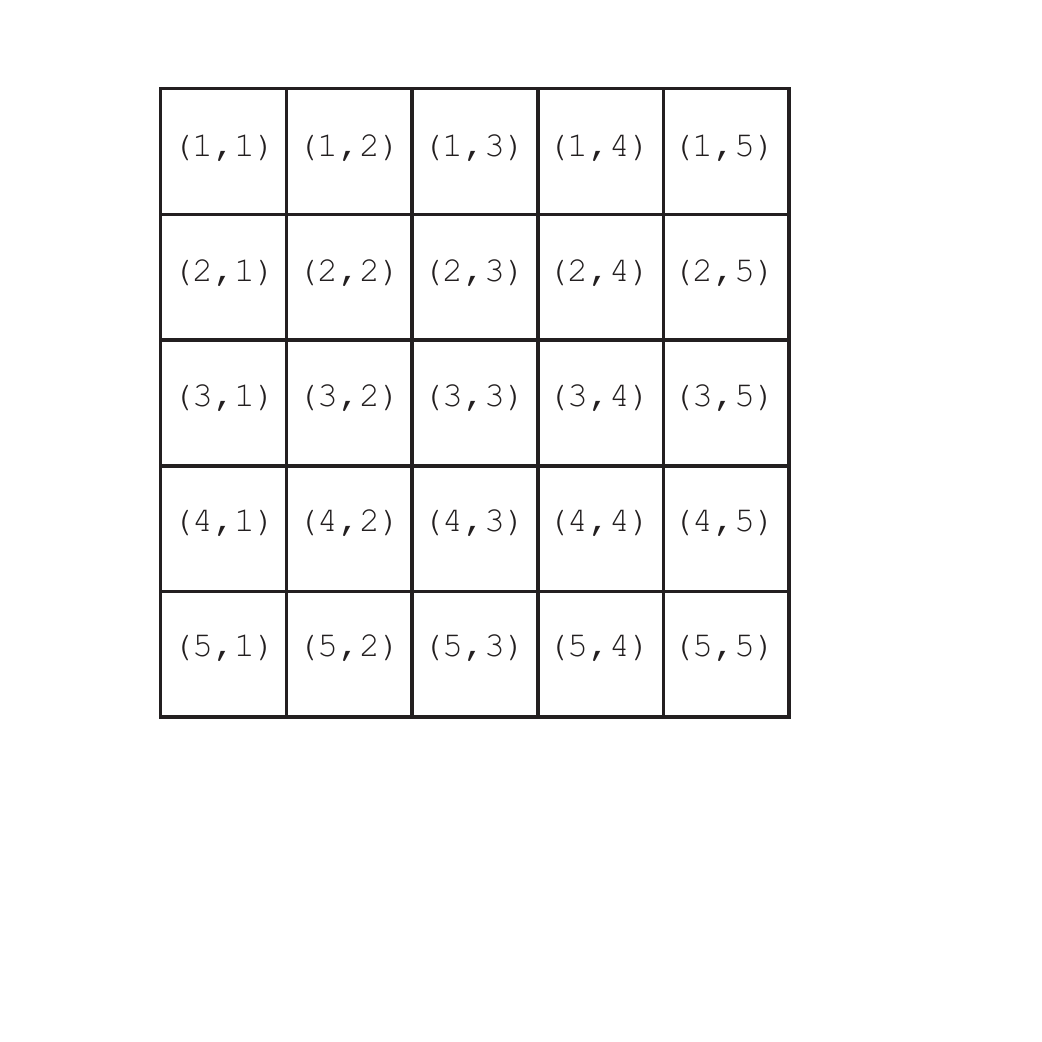}}
\caption{\sl A $5\times 5$ grid of pixels properly indexed.}
\label{fig: index}
\end{figure}

The algorithm requires a choice of spread to function.  Set
\[
\si(\ep) =\lfloor m^{r}\rfloor,
\]
where $r$ is a fixed  rational number $r\in (\frac{1}{2}, 1)$. Note that
\begin{equation}
\lim_{\ep \searrow 0} \ep (\sigma(\ep))^2 = \infty, \;\;\lim_{\ep \searrow 0} \ep \sigma(\ep) = 0.
\label{eq: spread}
\end{equation}
The  output of the algorithm will be a  $PL$ set $S_\ep$ that decomposes in a canonical  fashion as a finite union of trapezoids  with vertical bases. We will refer to such regions as \emph{polytrapezoids}.  We allow for degenerate trapezoids,  such as    points, segments, or triangles. 

\medskip

The algorithm uses  several basic subroutines.  The first one is the   the subroutine   $\stack$. Its input is a   list 
\[
C=C_1,\dotsc, C_m,\;\; C_i=0,1,
\]
which will be a column from $A$.  The output  of $\stack$  is a  list  of  nonnegative integers
\[
\bn(C); \;\;b_1\leq t_1<b_2\leq t_2<\cdots <b_{\bn(C)}\leq t_{\bn(C)},
\]
where  $\bn(C)$ is the number of stacks in the column encoded by    $C$,   and the location of the bottom and top pixel in the $j$-th stack is determined by    the integers  $b_j, t_j$. More formally
\[
C_k=1 \Llra \exists 1\leq j\leq \bn(C):\;\; b_j\leq k\leq t_j.
\]
  If  $C=C_i$, the $i$-th column  of $A$, i.e.,
  \[
  C_i= a_{i,1},\dotsc, a_{i,m}
  \]
 then we will denote the output $\stack(C_i)$ by
 \[
 \bn_i,\;\; b_{i,1}\leq t_{i,1}<\cdots <b_{i,\bn_i}\leq t_{i,\bn_i}.
 \]
 A number $1\leq i\leq m-1$ is called a \emph{jump point} if
 \[
 \bn_i\neq \bn_{i+1}.
 \]
 
 \medskip
 The next subroutine that we need is called $\jump$. Its input is an integer $k\in [1, m)$  and the  output is an integer $j_k=\jump(k)$ defined by as follows.
 If  
\[
\bigl\{ i\in [k,m)\cap\bZ;\;\; i\;\mbox{is a jump point}\,\bigr\}=\emptyset,
\]
 then we set
\[
\jump(k):= m+1.
\]
Otherwise
 \[
 \jump(k)=\min\bigl\{ i\in [k,m)\cap\bZ;\;\; i\;\mbox{is a jump point}\,\bigr\}.
 \]
The  noise region is  determined by  a finite collection of  intervals
\[
[\ell_1, r_1],\dotsc, [\ell_\alpha, r_\alpha]\subset [1,m]
\]
where the integers  $\ell_k, r_k$ are determined inductively as follows.
\[
\ell_1= \max\bigl(\, \jump(1)-2\si(\ep), 1\,\bigr),
\]
\[
 r_1=\min\bigl(\, m,  \jump(1)+2\si(\ep)\,\bigr).
\]
Suppose that $\ell_1,r_1,\dotsc, \ell_j,r_j$ are determined. If $\jump(r_j)>m$    we stop. Otherwise  we set
\[
\ell_{j+1}= \max\bigl(\, \jump(r_j)-2\si(\ep), 1\,\bigr),
\]
\[
r_{j+1}=\min\bigl(\, m,  \jump(r_j)+2\si(\ep)\,\bigr).
\]
The intervals $[\ell_1,r_1],\dotsc ,[\ell_\alpha,r_\alpha]$ may not be disjoint,  but their union is a   \emph{disjoint} union of intervals
\[
[a_1, b_1],\dotsc, [a_J,b_J],\;\; b_i<a_{i+1}.
\]
The intervals $[a_j,b_j], 1\leq j\leq J$ are the  \emph{noise intervals}.  The intervals
\[
[1,a_1], [b_1,a_2],\dotsc, [b_{J-1},a_J], [b_J, m]
\]
are  the \emph{regular intervals}.

\medskip

The heart of the algorithm consists of two  procedures, one for dealing  with the  noise intervals and and the other for dealing with the regular intervals.     These  procedures  will return a number of polytrapezoids,    

First some notation. Given   a collection of  points
\[
B_0, T_0,\dotsc, B_N, T_N\in \bR^2
\]
such that
\[
x(B_i)= x(T_i),\;\;y(B_i)\leq y(T_i),\;\;\forall i=0,\dotsc, N,
\]
\[
x(B_{j-1})<x(B_{j}),\;\;\forall 1\leq j\leq N,
\]
  we denote by $\polygon(B_0, T_0, \dotsc, B_N, T_N)$ the   region surrounded by the simple closed $PL$-curve obtained    as the union of line segments
  \[
  [B_0, B_1],\dotsc, [B_{N-1}, B_N], 
  \]
  \[
  [B_N,T_N],\dotsc, [T_1,T_0], [T_0, B_0].
  \]
 Note that  each of the quadrilaterals $B_{i-1}B_iTIT_{i-1}$ is a (possibly degenerate) trapezoid with vertical bases.
 
 \medskip
 
 Consider first the regular intervals.   Given a regular interval $I:=[p,q]$  we observe that  the number of stacks $\bn_i$ is independent of $i\in [p,q]$. We denote this shared number by $\bn=\bn(I)$.

 We  construct inductively a sequence of numbers $i_0 <\cdots < i_N$  as follows: 
 \begin{itemize}
 
 \item We set $i_0=p$. 
 
 \item If $q-p<2\si(\ep)$ we set $N=1$ and $i_1=q$. 
 
 \item If  $i_0,\dotsc, i_k$ are  already constructed, then,  if $q-i_k< 2\si(\ep)$ we set $N=k+1$ and $i_{k+1}=q$,  else $i_{k+1}=i_k+\si(\ep)$.
 \end{itemize}
 
 Note that if $q-p> \si(\ep)$, then $N\geq 1$, $i_0=p$, $i_N=q$ and
 \[
 N=1\;\;\mbox{if}\;\; q-p<\si(\ep).
 \]
 We have
 \[
 \stack(C_{i_k}) = \bn, \;\;b_{i_k,1}, t_{i_k,1},\dotsc, b_{i_k, \bn}, t_{i_k, \bn}.
 \]
 For $j=1,\dotsc, \bn$, and $k=0,\dotsc, N$ we denote by $B_{k,j}$ the center  of the $\ep$-pixel corresponding to the element entry $b_{i_k,j}$ in the column $C_{i_k}$. Similarly we denote by $T_{k,j}$ the center of the pixel corresponding to the entry $t_{i_k,j}$ of the column $C_{i_k}$.  For $1\leq j\leq \bn(I),
$ we set
 \[
 \eP_j(I) := \polygon(B_{0,j}, T_{0,j}, \dotsc,   B_{N,j}, T_{N,j}).
  \]
  Define
 \[
 \eP(I)=\bigcup_{j=1}^{\bn(I)} P_j(I),\;\; \eP_{\mathrm{reg}}:=\bigcup_{I\;\mathrm{regular\; interval}} P(I).
 \]
 \medskip
 Suppose now that $I=[p,q]$ is a noise interval.  We  modify the column 
 \[
 C_p= a_{p,1},\dotsc, a_{p,m}
 \]
  to a column 
  \[
  C'_p= a'_{p,1},\dotsc, a'_{p,m},
  \]
  by setting
  \[
  a'_{p,k}:= \begin{cases}
  1, & \mbox{if}\;\sum_{i=p}^q a_{i,k}>0\\
  &\\
  0, & \mbox{if}\;\sum_{i=p}^q a_{i,k}>0.
  \end{cases}
  \]
  We apply the subroutine $\stack$ to the new column $C'_p$and the output is
  \[
  \stack(C'_p)= \bn(I),\;\;b_1\leq t_1<\cdots <b_n\leq t_n.
  \]
 For $j=1,\dotsc, \bn(I)$ we  set
 \[
 B_{0,j}:=C[p,b_j],\;\; T_{0,j}:= C[p,t_j],
 \]
 \[
 B_{1,j}:=C[q,b_j],\;\; T_{0,j}:= C[q,t_j],
 \]
 (recall that $C[i,j]$ is defined as the center of the pixel associated to $(i,j)$). Next,  for $j=1,\dotsc, \bn(I)$ we  define the rectangle 
 \[
 \eR_j(I) := \polygon(B_{0,j}, T_{0,j}, B_{1,j}, T_{1,j}),
 \]
 and we set
 \[
 \eR(I)= \bigcup_{j=1}^{\bn(I)} \eR_j(I),\;\;  \eP_{\rm noise}:=\bigcup_{I\;\mathrm{noise\; interval}}R(I).
 \]
 The output of the algorithm is the polytrapezoid
  \[
  \eP_\ep(A):=\eP_{\rm regular}\cup \eP_{\rm noise}.
  \]

\end{document}